\documentclass[11pt]{article}
\usepackage{fullpage}
\usepackage{preamble}
\usepackage{import}
\usepackage{amsmath}
\usepackage{bbm}
\usepackage[utf8]{inputenc}
\usepackage{mathtools, bm}
\usepackage{amssymb, bm}
\usepackage{xcolor}
\usepackage{emptypage}
\usepackage{pdfpages}

\usepackage[
backend=biber,
style=alphabetic,
sorting=nyt,
maxnames=4,
maxalphanames=4,
backref=true
]{biblatex}
\title{Residues of Rankin--Selberg Zeta integrals and the split non-tempered Gan--Gross--Prasad conjectures}
\author{Paul Boisseau}

\begin{document}
\maketitle
\begin{abstract}
    We construct a regularization of the Rankin--Selberg period on general linear groups for non-tempered automorphic representations using residues of Zeta integrals. We prove that it satisfies the global non-tempered Gan--Gross--Prasad conjecture and its Ichino--Ikeda refinement. We also build a local version of our regularization and show that it defines a non-zero invariant linear form on non-tempered representations. Combined with previous works of Chan, Chen and Chen, this settles the conjectures over local fields of characteristic zero.
\end{abstract}

\setcounter{tocdepth}{1}
\tableofcontents

\section{Introduction}
In \cite{GGP}, Gan, Gross and Prasad made a series of conjectures on restrictions problems for classical groups over local fields, and on periods of automorphic forms over global fields. These conjectures were concerned with tempered and generic representations, and have sparked a lot of interest in recent years (see e.g. \cite{BP23} for a recent survey). In \cite{GGP2}, the authors generalized their framework to non-tempered representations. This paper is concerned with the \emph{split} setting for Bessel models, i.e. the situation where the classical groups are general linear groups and the period is the Rankin--Selberg integral. Our results prove the "relevance implies distinction" direction of the non-tempered conjectures, both globally and locally. In what follows, we take $F$ to be a number field with adele ring $\bA$. All algebraic groups will be defined over $F$.

\subsection{The split global non-tempered Gan--Gross--Prasad conjecture}

We first state our global result. We start by recalling the content of the conjecture of \cite{GGP2}.

\subsubsection{Rankin--Selberg periods of Arthur type representations}
\label{subsubsec:RS_for_Arthyr}

Let $n$ be a positive integer. Consider $G=\GL_n \times \GL_{n+1}$ and $H=\GL_n$ embedded in $G$ via the diagonal embedding $h \mapsto \left(h,\begin{pmatrix}
    h & \\
    & 1
\end{pmatrix} \right)$. Write $\cA(G)$ for the space of automorphic forms on $G$, and set $[H]=H(F) \backslash H(\bA)$ which is given a right-invariant measure. We define the \emph{Rankin--Selberg} period to be the (a priori non-convergent) integral 
\begin{equation*}
    \varphi \in \cA(G) \mapsto \cP_H(\varphi):=\int_{[H]} \varphi(h) dh.
\end{equation*}
If $\varphi$ ranges inside a cuspidal representation $\Pi=\Pi_n \boxtimes \Pi_{n+1}$ (so that the integral converges), a celebrated theorem of \cite{JPSS83} states that $\cP_H$ vanishes on $\Pi$ if and only if $L(1/2,\Pi_n \times \Pi_{n+1})= 0$. The conjecture in \cite{GGP2} aims to extend this result to automorphic representations of Arthur type, whose definition we now recall.

For any non-negative integer $d$, let $S_d$ be the unique $d$-dimensional irreducible algebraic representation of $\SL_2(\cc)$. A global Arthur parameter $\psi_n$ for $\GL_n$ is a formal sum
\begin{equation*}
    \psi_n=\sum_{i=1}^m \sigma_i \boxtimes S_{d_i},
\end{equation*}
where $\sigma_1, \hdots , \sigma_m$ are (unitary) cuspidal automorphic representations of some smaller $\GL_{n_1}, \hdots, \GL_{n_m}$ respectively and where $\sum_{i=1}^m n_i d_i=n$. To $\psi$ one can associate a global Arthur-packet $\Pi_{\psi_n}$ which contains a unique automorphic representation $\Pi_n$ of $\GL_n$. It can be written as a parabolic induction
\begin{equation}
        \label{eq:Arthur_decomp}
    \Pi_{n}=\Speh(\sigma_1,d_1) \times \hdots \times \Speh(\sigma_m,d_m), 
\end{equation}
where for any cuspidal automorphic representation $\sigma$ of $\GL_k$ and any $d \geq 1$, $\Speh(\sigma,d)$ is the discrete automorphic representation of $\GL_{dk}$ spanned by residues of Eisenstein series induced from $\sigma \Val{\det}^{(d-1)/2} \times \hdots \times \sigma \Val{\det}^{(1-d)/2}$. If $d=0$, we write $\Speh(\sigma,0)$ for the trivial representation of the trivial group. If $d_1=\hdots=d_m=1$, the parameter is said to be \emph{tempered}, otherwise it is \emph{non-tempered}. Automorphic representations of $\GL_n$ that lie in global Arthur packets are said to be of \emph{Arthur type}. These notions naturally extend to parameters and representations of $G$, and we write Arthur parameters of $G$ as pairs $\psi=(\psi_n,\psi_{n+1})$.

Following \cite{GGP2}, we say that an Arthur parameter $\psi=(\psi_n,\psi_{n+1})$ of $G$ is \emph{relevant} if it decomposes as
\begin{align}
    \psi_n&=\sum_{i=1}^{m_1} \sigma_{1,i} \boxtimes S_{d_{1,i}} + \sum_{j=1}^{m_2} \sigma_{2,j}^\vee \boxtimes S_{d_{2,j}-1}, \label{eq:relevance_defi_1} \\ 
     \psi_{n+1}&=\sum_{i=1}^{m_1} \sigma_{1,i}^\vee \boxtimes S_{d_{1,i}-1} + \sum_{j=1}^{m_2} \sigma_{2,j} \boxtimes S_{d_{2,j}}, \label{eq:relevance_defi_2}
\end{align}
where $\sigma^\vee$ stands for the contragredient of $\sigma$, and we ask that all the $d_{1,i}$'s and $d_{2,j}$'s are positive. The conjecture in \cite{GGP2} can now be stated as follows.

\begin{conj}
    \label{conj:global_GGP_intro}
    Let $\Pi$ be an automorphic representation of Arthur type of $G$ with Arthur parameter $\psi$. There exists an extension $\cP_H^*$ of $\cP_H$ defined on $\Pi$. Moreover, $\cP_H^*$ does not vanish on $\Pi$ if only if the following conditions are satisfied.
    \begin{enumerate}
        \item $\psi$ is relevant.
        \item For all place $v$ of $F$ we have $\Hom_{H(F_v)}(\Pi_v,\cc) \neq \{0\}$.
        \item The quotient of special values of $L$-functions defined in \cite[Conjecture~9.1]{GGP2} is non-zero.
    \end{enumerate} 
\end{conj}

\subsubsection{The tempered case}

If the parameter $\psi$ is tempered, so that is is relevant, then Conjecture~\ref{conj:global_GGP_intro} is known. Indeed, in that case we have $P=M_P N_P$ a standard parabolic subgroup of $G$ and $\sigma$ is a cuspidal representation of $M_P$ such that $\Pi$ is the parabolic induction $I_P^G \sigma$. Decompose it as $I_{P_n}^{\GL_n} \sigma_n \boxtimes I_{P_{n+1}}^{\GL_{n+1}} \sigma_{n+1}$ and consider the quotient of completed Rankin--Selberg $L$-functions
\begin{equation}
    \label{eq:tempered_L}
    \cL(\sigma):=\frac{L(1/2,I_{P_n}^{\GL_n} \sigma_n \times I_{P_{n+1}}^{\GL_{n+1}} \sigma_{n+1})}{L(1,\sigma,\Ad_P)},
\end{equation}
where $\Ad_P$ is the adjoint representation of the Langlands dual group $\widehat{M_P}(\cc)$ on the Lie algebra of $\widehat{N_P}(\cc)$. In \cite{IY}, Ichino and Yamana built a regularization $\cP_H^{\mathrm{IY}}$ defined on a subspace $\cA(G)^{\mathrm{reg}} \subset \cA(G)$. Because $\psi$ is tempered, any Eisenstein series induced from $\Pi$ lies in $\cA(G)^{\mathrm{reg}}$. Moreover, \cite[Theorem~1.1]{IY} shows that $\cP_H^{\mathrm{IY}}$ computes the global Zeta integral $Z$ of \cite{JPSS83}. It follows that if $\varphi=\otimes_v \varphi_v \in \Pi$ we have a finite set $\tS$ of places of $F$ such that 
\begin{equation}
    \label{eq:IY_facto}
    \cP_H^{\mathrm{IY}}(E(\varphi,0))=Z(E(\varphi,0))=\cL(\sigma) \prod_{v \in \tS} Z_{\sigma,v}^\natural(\varphi_v),
\end{equation}
where the $Z_{\sigma,v}^\natural(\varphi_v)$ are local Zeta integrals on $I_P^G \sigma_v$ normalized by the local version of \eqref{eq:tempered_L}. Because the $\sigma_v$'s are generic, \cite{JPSS83} implies that they can be chosen to be non-zero. This proves Conjecture~\ref{conj:global_GGP_intro} for tempered parameters.

\subsubsection{The non-tempered case}
\label{subsubsec:non_tempered}

The first main result of this paper is an extension of \cite{IY} to non-tempered parameters. It turns out that the period $\cP_H^{\mathrm{IY}}$ is not the correct regularization in that case as it automatically vanishes on non-generic (hence in particular non-tempered) automorphic representations. We now explain how to circumvent this issue.

Let $\Pi$ be a automorphic representation of Arthur type of $G$ with relevant Arthur parameter $\psi$. We write $\Pi=I_P^G \pi$, where $\pi$ is the tensor product of Speh representations determined by \eqref{eq:relevance_defi_1} and \eqref{eq:relevance_defi_2}. By \cite{MW89}, there exist $P_\pi$ a standard parabolic subgroup of $G$, $\sigma$ a cuspidal automorphic representation of $M_{P_\pi}$ and $\nu$ an unramified character of $M_{P_\pi}$ such that $\Pi$ is spanned by residues of Eisenstein series induced from $I_{P_\pi}^G \sigma$ at $\nu$. We denote this map by $E^{P,*}$. 

Let $\fa_{P,\cc}^*$ be the vector space of unramified characters of $M_P$. The relevance condition of \eqref{eq:relevance_defi_1} and \eqref{eq:relevance_defi_2} can be twisted by a subspace $\fa_{\pi,\cc}^* \subset \fa_{P,\cc}^*$. We refer to \eqref{eq:a_pi_defi} for an explicit description in coordinates (see also \S\ref{subsec:example} below where we write down this subspace in an example). We can further identify $\fa_{\pi,\cc}^*+\nu$  as a subspace of $\fa_{P_\pi,\cc}^*$. It is contained in an union of singularities of $\lambda \mapsto \cL(\sigma_\lambda)$, where we write $\sigma_\lambda$ for $\sigma \otimes \lambda$. These singularities are all affine hyperplanes coming from the numerator of \eqref{eq:tempered_L}. By multiplying $\cL(\sigma_\lambda)$ by the appropriate product of affine linear forms and restricting, we obtain a meromorphic function on $\fa_{\pi,\cc}^*+\nu$ which we denote by $\cL(\lambda,\pi)$. It may have poles at $\lambda=\nu$, and we further write $\cL^*(\pi)$ for its regularization at this point. We emphasize that, due to the straightforward nature of the poles of Rankin--Selberg $L$-functions, our residues are simply defined by multiplying by products of affine linear forms and evaluating. Moreover, $\cL^*(\pi)$ can be written explicitly in terms of special values of $L$-functions, and in particular 
\begin{equation}
    \label{eq:big_L_factor}
    \cL^*(\pi) \neq 0 \iff \frac{\displaystyle\prod_{i,j} L^*\left(\frac{d_{1,i}-d_{2,j}+1}{2}, \sigma_{1,i} \times \sigma_{2,j} \right)}{\displaystyle \prod_{k=1}^2 \prod_{1 \leq i<j \leq m_k} L\left(\frac{d_{k,i}+d_{k,j}}{2}, \sigma_{k,i} \times \sigma_{k,j}^\vee \right)} \neq 0,
\end{equation}
where $L^*$ stands for the residue of the $L$-function if it is evaluated at a pole. 

For every place $v$ of $F$, using Jacquet integrals the local factor $Z_{\sigma,v}^\natural(\phi_v)$ extends to a meromorphic function $\lambda \in \fa_{P_\pi,\cc}^* \mapsto Z_{\sigma,v}^\natural(\phi_v,\lambda)$ so that $\lambda \mapsto \cP_H^{\mathrm{IY}}(E(\phi,\lambda))$ is meromorphic. We denote by $\mathrm{Res} \; \cP_H^{\mathrm{IY}}(E(\phi,\nu))$ its residue at $\nu$ obtained by multiplying by the same products of affine linear forms as above. Note that it is not obviously well-defined as the local factors $Z_{\sigma,v}^\natural$ may have additional poles at $\nu$. Our main result in the global setting is a refinement of Conjecture~\ref{conj:global_GGP_intro}.

\begin{theorem}
    \label{thm:GGP_global_intro}
    The following assertions hold.
    \begin{itemize}
        \item The residue $\mathrm{Res} \; \cP_H^{\mathrm{IY}}(E(\phi,\nu))$ is well-defined and factors through $I_{P_\pi}^G \sigma_\nu \twoheadrightarrow \Pi$. We denote by $\cP_\pi$ the resulting $H(\bA)$-invariant linear form on $\Pi$.
        \item (Ichino--Ikeda conjecture) For $\varphi=E^{P,*}(\phi,\nu) \in \Pi$ with $\phi=\otimes \phi_v \in I_{P_\pi}^G \sigma$, there exists a finite set of places $\tS$ such that
        \begin{equation*}
            \cP_{\pi}(\varphi)=\cL^*(\pi)  \prod_{v \in \tS} Z_{\sigma,v}^\natural(\phi_v,\nu),
        \end{equation*}
        where all the local factors are regular.
        \item For every place $v$ of $F$ the linear form $Z_{\sigma,v}^\natural(\cdot,\nu)$ is non-zero so that 
        \begin{equation*}
            \cP_\pi \neq 0 \iff \cL^*(\pi) \neq 0.
        \end{equation*}
    \end{itemize}
\end{theorem}

As indicated by the notation, $\cP_\pi$ depends on the choices of the parabolic $P$ and representation $\pi$. However, choosing different inducing data $P'$ and $\pi'$ will result in a functional equation between $\cP_\pi$ and $\cP_{\pi'}$ (see Proposition~\ref{prop:independence_choice_couple}). Moreover, the vanishing of $\cL^*(\pi)$ only depends on $\Pi$. The second point can be seen as a split avatar of \cite[Conjecture~3.2]{II}. We also note that our quotient $\cL^*(\pi)$ is not exactly the one predicted by \cite[Conjecture~9.1]{GGP2}, which is in fact often zero due to the presence of poles in the denominator, and therefore would need to be regularized.

\subsection{The split local non-tempered Gan--Gross--Prasad conjecture}

We now pick a place $v$ of $F$ and address the issue raised by the second point of Conjecture~\ref{conj:global_GGP_intro}, namely the vanishing of $\Hom_{H(F_v)}(\Pi_v,\cc)$. This space is known to be of dimension at most $1$ by \cite{AGRS}, and if it is non-zero $\Pi_v$ is said to be \emph{distinguished}. In \cite{GGP2}, Gan, Gross and Prasad made a conjecture on distinction for smooth irreducible representation of Arthur type. We extend it to a more general version for representations of \emph{weak} Arthur type, and then state our main local result which is one direction of this conjecture.

\subsubsection{An enhanced local conjecture for representations of weak Arthur type}
\label{subsubsec:local_non_tempered_GGP}

Let $W_v$ be the Weil group of $F_v$, and let $WD_{v}$ be the Weil-Deligne group of $F_v$ which is $W_v$ is $v$ is Archimedean, and $W_v \times \SL_2(\cc)$ is $v$ is non-Archimedean. By local class field theory, the abelianization of $W_v$ is isomorphic to $F_v^\times$ and we let $\Val{\cdot}$ be the normalized absolute value of $F_v^\times$, which we identify with a morphism $WD_v \to \cc^\times$ trivial on the second factor. We define a \emph{weak} Arthur parameter $\psi_{v,n}$ to be a morphism $WD_v \times \SL_2(\cc) \to \GL_n(\cc)$ of the form
\begin{equation*}
    \psi_{v,n}=\bigoplus_{i=1}^m (M_i \otimes \Val{\cdot}^{\nu_i} )\boxtimes S_{d_i},
\end{equation*}
where each $M_i$ is a finite dimensional irreducible representation of $WD_v$ whose restrictions to $W_v$ and $\SL_2(\cc)$ are bounded and algebraic respectively, and where $\nu_1, \hdots ,\nu_m$ are real numbers with $\Val{\nu_1}, \hdots, \Val{\nu_m}<1/2$. 

One can associate to the weak Arthur parameter $\psi_{v,n}$ a Langlands parameter, and therefore a smooth irreducible representation $\Pi_{v,n}$ of $\GL_n(F_v)$ by the local Langlands correspondence (\cite{Langlands89}, \cite{Henniart}, \cite{HT} and \cite{Scholze}). More precisely, to each $M_i$ the correspondence associates an irreducible square integrable (unitary) representation $\delta_{v,i}$ of some $\GL_{n_i}(F_v)$ and we have 
\begin{equation}
    \label{eq:pi_v_n_defi}
    \Pi_{v,n}=\Speh(\delta_{v,1},d_1) \Val{\det}^{\nu_1} \times \hdots \times \Speh(\delta_{v,m},d_m) \Val{\det}^{\nu_m}.
\end{equation}
where if $\delta_v$ is an irreducible square integrable representation of some $\GL_k(F_v)$, then $\Speh(\delta_v,d)$ is the unique irreducible quotient of $ \delta_v \Val{\det}{\frac{d-1}{2}} \times \hdots \times \delta_v \Val{\det}^{\frac{-(d-1)}{2}}$. By \cite[Proposition~I.9]{MW89}, such $\Pi_{v,n}$ is indeed irreducible. This defines a class of representations of \emph{weak} Arthur type of $G(F_v)$. By \cite{Vo} and \cite{Ta}, it contains the unitary dual of $G(F_v)$.

By setting $(M_i \otimes \Val{\cdot}^{\nu_i} )^\vee=(M_i^\vee \otimes \Val{\cdot}^{-\nu_i} )$ and copying \eqref{eq:relevance_defi_1} and \eqref{eq:relevance_defi_2}, we get a notion of relevance for weak Arthur parameters $\psi_v$ of $G(F_v)$. Our enhanced version of the conjecture is the following.

\begin{conj}
    \label{conj:local_GGP_intro}
    Let $\Pi_v$ be a representation of weak Arthur type of $G(F_v)$ with parameter $\psi_v$ Then $\Hom_{H(F_v)}(\Pi_v,\cc) \neq \{0\}$ if and only if $\psi_v$ is relevant.
\end{conj}

If $v$ is non-Archimedean, the direct implication for unitarizable representations was shown in \cite{Gu}, and the full conjecture for Arthur type representations was proved in \cite{Chan}. However, this last proof is not constructive and does not give an explicit element in $\Hom_{H(F_v)}(\Pi_v,\cc)$. For $v$ Archimedean, the direct implication is obtained in \cite{CC} for the weak Arthur variant.

Note that the local component $\Pi_v$ of an Arthur type automorphic representation $\Pi$ of $G$ is of weak Arthur type, unconditionally on the Ramanujan conjecture. It is straightforward that if the global Arthur parameter of $\Pi$ is relevant, so will that of $\Pi_v$. However, it is not obvious that the converse holds as illustrated in \cite[Remark~9.2]{GGP2}.

\subsubsection{Residues of local Zeta functionals}
\label{subsubsec:residues_zeta_intro}
 We now state the main local result of this paper, which logically comes before Theorem~\ref{thm:GGP_global_intro}. 
 
 We take $v$ a place of $F$ and $\Pi_v$ a weak Arthur type representation of $G(F_v)$ with parameter $\psi_v$. As in the global setting, we may write $\Pi_v$ as a parabolic induction $I_P^G \pi_v$, where $\pi_v$ is a product of essentially Speh representations, and further realize it as a quotient of an induction $I_{P_\pi}^G \delta_{v,\nu}$, where $\delta_v$ is a unitary square integrable representation of $M_{P_\pi}(F_v)$ and $\nu \in \fa_{P_\pi}^*$. This induction is chosen relatively to the decompositions \eqref{eq:relevance_defi_1} and \eqref{eq:relevance_defi_2} so that $I_{P_\pi}^G \delta_{v,\nu}$ may not be a standard module. For $\lambda \in \fa_{P_\pi,\cc}^*$ in general position, we have the Zeta integral $Z_{\delta_v}(\phi_v,\lambda)$ on $I_{P_\pi}^G \delta_{v,\lambda}$ built using a Jacquet integral. It is meromorphic in $\lambda$, and we consider the normalized version $Z^\natural_{\delta_v}(\phi_v,\lambda)$ obtained by taking the quotient by the local version $\cL(\delta_{v,\lambda})$ of \eqref{eq:tempered_L}. 

 If we now assume that $\psi_v$ is relevant, we can define an affine subspace $\fa_{\pi,\cc}^*+\nu \subset \fa_{P_\pi,\cc}^*$ as in \S\ref{subsubsec:non_tempered}. It is not contained in any of the singularities of $Z^\natural_{\delta_v}(\phi_v,\lambda)$. Therefore, the restriction $R Z^\natural_{\delta_v}$ of $Z^\natural_{\delta_v}$ is a well-defined meromorphic functional on $\fa_{\pi,\cc}^*+\nu$, but it may a priori have poles at $\nu$ coming from the denominator of \eqref{eq:tempered_L}. $RZ^\natural_{\delta_v}$ should be interpreted as a residue of a local Zeta integral. Our second result is an Ichino--Ikeda refinement of the non-tempered Gan--Gross--Prasad conjecture which is the split avatar of \cite[Conjecture~3.1]{II}.

 \begin{theorem}
    \label{thm:local_GGP_intro}
    Let $\Pi_v$ be an irreducible representation of $G(F_v)$ of weak Arthur type with relevant parameter. Then $RZ^\natural_{\delta_v}(\cdot,\lambda)$ is regular at $\lambda=\nu$. It factors through the quotient $I_{P_\pi}^G \delta_{v,\nu} \twoheadrightarrow \Pi_v$ and defines a non-zero element in $\Hom_{H(F_v)}(\Pi_v,\cc)$. 
 \end{theorem}

 As explained in \S\ref{subsubsec:local_non_tempered_GGP}, the Archimedean case of Theorem~\ref{thm:local_GGP_intro} was the final missing piece to complete the proof of the local non-tempered conjecture \ref{conj:local_GGP_intro} in the Arthur type setting. However, we emphasize that our result is also new in the non-Archimedean case, as we explicitly build a non-zero element in $\Hom_{H(F_v)}(\Pi_v,\cc)$.
 
 We provide two proofs of Theorem~\ref{thm:local_GGP_intro}. The first is uniform in $v$ and relies on a globalization argument using a weak version of Theorem~\ref{thm:GGP_global_intro}. We use it to complete the proof of our global result. The second method is independent and local, but only works if $v$ is non-Archimedean. 

\subsection{Remarks on Theorems~\ref{thm:GGP_global_intro} and \ref{thm:local_GGP_intro}}

\subsubsection{Relation with the BZSV formalism}

We explain the connection of our global result Theorem~\ref{thm:GGP_global_intro} with \cite[Conjecture~14.3.5]{BZSV} which relates automorphic periods and special values of $L$-functions (a priori in the function field case). The formalism is roughly the following. Let $\LAG$ be a reductive group over $F$, and let $M$ be an Hamiltonian $\LAG$-variety so that one can define a (possibly regularized) period $\cP_M$. Assume that $M$ has a dual Hamiltonian variety $\check{M}$ with action of the Langlands dual group ${}^L \LAG$. If $\Pi$ is an automorphic representation of $\LAG$ with Arthur parameter $\psi$ (possibly non-tempered), we obtain an action of the hypothetical Langlands group $\cL_F$ of $F$ on $\check{M}$ by composition. One can then define a subvariety $\check{M}_{\mathrm{slice}}$ of $\check{M}$ stable under this action. If $\check{M}$ is polarized (which is the case here), \cite[Conjecture~14.3.5]{BZSV} says that
\begin{equation}
    \label{eq:BZSV_equality}
    \cP_M(\varphi) \sim \sum_{i=1}^m L(0,T_{x_i} \check{M}_{\mathrm{slice}}^\shear), \quad \varphi \in \Pi.
\end{equation}
where $\{x_1, \hdots, x_m\}$ is the finite set of $\cL_F$-fixed points on $\check{M}_{\mathrm{slice}}$, $L(0,T_{x_i} \check{M}_{\mathrm{slice}}^\shear)$ is some special value of $L$-function coming from the action of $\cL_F$ on the tangent space $T_{x_i} \check{M}_{\mathrm{slice}}$, and $\sim$ means that \eqref{eq:BZSV_equality} holds up to local factors.

In our case, we have $M=T^* (H \backslash G)$ and $\check{M}=T^* \mathrm{std}_{{}^L G}$ the cotangent bundle of the standard representation of ${}^L G$. One can check that fixed points of $\psi$ correspond to the \emph{correlators} of \cite[Section~4]{GGP2}. In particular, as asserted by \cite[Lemma~4.3]{GGP2}, $\psi$ is relevant if and only if it has a fixed point on $\check{M}$. However, we claim that Theorem~\ref{thm:GGP_global_intro} lies slightly outside of the framework of \eqref{eq:BZSV_equality} as both sides of this equation are undefined in our case. 

On the LHS, the regularized period $\cP_M$ is $\cP_H^{\mathrm{IY}}$. But as soon as the Arthur parameter of $\Pi$ is non-tempered the Eisenstein series induced from $\Pi$ don't lie in $\cA(G)^\mathrm{reg}$. This is reflected by the fact that we have to take residues to define $\cP_\pi$. In particular, $\cP_\pi$ should not be interpreted as an extension of $\cP_H$ to the whole space of automorphic forms of Arthur type, but rather as a regularized integral along a degeneration of $[H]$. More precisely, we explain in \S\ref{subsubsec:residue_free} how to associate to $\Pi$ two standard parabolic subgroups $P_{+,\pi}$ and $P_+$ as well as a Weyl element $w_+$. Set $P_{+,H}=P_+ \cap H$ and write the standard Levi decomposition $P_{+,H}=M_{+,H}N_{+,H}$. Using the construction of \cite{Zydor}, we consider a linear form $\cP^{P_+}$ which regularizes the integral along $Z^{\infty}_+ M_{+,H}(F) N_{+,H}(\bA) \backslash H(\bA)$, where $Z^\infty_+$ is a certain subgroup of the center of $M_{+,H}$. Then we show in Proposition~\ref{prop:alternative_construction} that 
\begin{equation}
    \label{eq:period_alternative}
    \cP_\pi(\varphi)=\cP^{P_+}\left(E^{P_+}\left(M(w_+) \varphi_{P_{+,\pi}}\right)\right), \quad \varphi \in I_P^G \pi,
\end{equation}
where $\varphi_{P_{+,\pi}}$ is the constant term of $\varphi$ along $P_{+,\pi}$, $M(w_+)$ is an intertwining operator, and $E^{P_+}$ is a generalized Eisenstein series. The relevance condition on $\psi$ implies that this series transforms appropriately under the action of $Z_+^\infty$. 

On the RHS of \eqref{eq:BZSV_equality}, the corresponding failure is that the set of fixed points is infinite as soon as $\psi$ is not tempered. Because $\cP^{P_+}$ can be interpreted as a product of Petersson inner-products and of a smaller Zeta integral on the Levi $M_+$ (see Remark~\ref{rem:diagonal_Arthur} and Proposition~\ref{prop:Zydor_RS}), one can easily cook up a new ${}^L M_+$-Hamiltonian variety which will yield the correct $L$-factor from \eqref{eq:big_L_factor}. However, it is at the moment not clear to us how to define it for general $\check{M}$ and $\psi$, but we believe that the regularization procedure of $\cP_H^{\mathrm{IY}}$ carried out by Theorem~\ref{thm:GGP_global_intro} should have a spectral analogue on $\check{M}$. We plan to investigate this question in a future work.

 \subsubsection{Relation to non-tempered Gan--Gross--Prasad conjectures on unitary groups}

 We expect our results to have applications to the other conjectures of \cite{GGP2} for unitary groups. In the global case, the tempered conjecture from \cite{GGP} was proved in \cite{BPCZ} using a comparison of relative trace formulae pioneered by \cite{JR}. On the general linear groups' side, the spectral expansion of the trace formulae involves Rankin--Selberg and Flicker--Rallis periods of automorphic forms (see \cite{Fli} for the latter). To prove the conjectures of \cite{GGP2} using this approach, it is necessary to derive the non-tempered contributions to this trace formula. This was done for the Flicker--Rallis period in \cite{Ch}, and we will compute in a subsequent work \cite{BoiRS} the fine expansion of its Rankin--Selberg counterpart. It involves the linear forms $\cP_\pi$ from Theorem~\ref{thm:GGP_global_intro}. Therefore, it is crucial to understand their connections with $L$-functions. 

\subsubsection{Rankin--Selberg integrals for non-tempered representations}

If $F_v$ is non-Archimedean, Rankin--Selberg integrals for Speh representations have been studied in the equal rank case in \cite{LM20}, and in the almost equal rank case in \cite{AKS}. These integrals are respectively split Fourier--Jacobi and Bessel functionals and are built using degnerate Whittaker models. Our approach is different as we consider Rankin--Selberg integrals on standard modules and show that they factor through the Langlands quotient. Our proof only requires the knowledge of the derivatives of the standard module and not of the non-tempered representations themselves, which enables us to obtain our non-vanishing result. Nevertheless, it would be interesting to know if $RZ_{\delta_v}^\sharp$ can be written as an integral of a degenerate Whittaker function.
 
 \subsection{About the proofs: computation of residues of Zeta integrals}

 \subsubsection{About Theorem~\ref{thm:GGP_global_intro}}

To prove Theorem~\ref{thm:GGP_global_intro}, we compute the residue of the meromorphic functional $\lambda \mapsto \cP_{H}^{\mathrm{IY}}(E(\cdot,\lambda))$, defined on the induction $I_{P_\pi}^G \sigma_{\lambda}$, in two ways. The first is to start from the relation $\cP_{H}^{\mathrm{IY}}(E(\phi,\lambda))=Z(E(\phi,\lambda))$ given by \cite[Theorem~1.1]{IY}. Because the Zeta integral admits the Euler product expansion \eqref{eq:IY_facto}, this readily yields a similar expansion for its residue $\mathrm{Res} \;\cP_{H}^{\mathrm{IY}}(E(\phi,\lambda))$ for $\lambda$ in general position, and thus gives the second point in Theorem~\ref{thm:GGP_global_intro}. However, it is not clear from this expression that this linear form factors through $I_{P_\pi}^G \sigma_{\lambda} \twoheadrightarrow I_P^G \pi_\lambda$. To prove this fact, we realize $\cP_H^{IY}$ using the truncation techniques of \cite{Zydor}, which we show agrees with the ones of \cite{IY}. This allows us to express $\cP_{H}^{\mathrm{IY}}(E(\phi,\lambda))$ as a sum of truncated periods of Eisenstein series, which can be thought of as a \emph{reversed Maa\ss--Selberg relation}. By carefully analyzing the constant term of Eisenstein series series and their decomposition according to cuspidal exponents, we show that $\mathrm{Res} \;\cP_{H}^{\mathrm{IY}}(E(\cdot,\lambda))$ factors through a certain regularized intertwining operator $M^*(w_\pi^*,\lambda)$. Because the latter realizes the map $I_{P_\pi}^G \sigma_{\lambda} \twoheadrightarrow I_P^G \pi_\lambda$, we obtain the desired factorization property. To go beyond $\lambda$ in general position, we have to deal with local factors and use Theorem~\ref{thm:local_GGP_intro}.
 
 We emphasize that the singular affine hyperplanes of $\cP_{H}^{\mathrm{IY}}(E(\phi,\lambda))$ are not those of the Eisenstein series $E(\phi,\lambda)$. Indeed, if $\cH$ is one such affine hyperplane, we can show that 
 \begin{equation*}
    \underset{\cH}{\Res} \left(\cP_{H}^{\mathrm{IY}}(E(\phi,\lambda))\right)=\cP_{H}^{\mathrm{IY}} \left( \underset{\cH}{\Res} \; E(\phi,\lambda) \right)=Z\left( \underset{\cH}{\Res} \; E(\phi,\lambda) \right),
 \end{equation*}
 where the last point is \cite[Theorem~1.1]{IY}. But because residual Eisenstein series are not generic, this is zero. Therefore, the factorization property of $\mathrm{Res} \;\cP_{H}^{\mathrm{IY}}(E(\cdot,\lambda))$ is less straightforward. 

 \subsubsection{About Theorem~\ref{thm:local_GGP_intro}}

 In the local setting, the key point is to prove that $RZ^\natural_{\delta_v}(\cdot,\lambda)$ factors through $I_{P_\pi}^G \delta_{v,\lambda} \twoheadrightarrow I_P^G \pi_{v,\lambda}$ for $\lambda \in \fa_{\pi,\cc}^*+\nu$ in general position. Indeed, from there it is enough to prove the regularity and non-vanishing of $RZ^\natural_{\delta_v}(\cdot,\nu)$ if $I_{P_\pi}^G \delta_{v,\nu}$ is a standard module, which then follows from \cite{JPSS83}. Because we know that this factorization property holds globally (at this point only in general position), it also does locally using a globalization argument.  This yields the first proof of Theorem~\ref{thm:local_GGP_intro}.
 
 Our second proof parallels the global method. We realize the map $I_{P_\pi}^G \delta_{v,\lambda} \twoheadrightarrow I_P^G \pi_{v,\lambda}$ as a local intertwining operator $N_v(w_\pi^*,\lambda)$ and write a local analogue of the reversed Maa\ss--Selberg equation. To do this, we first prove in Proposition~\ref{prop:Whittaker_explicit_discrete} an asymptotic formula for Whittaker functionals which mirrors that of matrix coefficients from \cite{Cas}. In fact, this formula is alluded to in \cite{CasLetter}. We then obtain an expression for the local Zeta integral $Z_{\delta_v}(\phi_v,\lambda)$ which lets us compute $Z^\natural_{\delta_v}(\phi_v,\lambda)$ and prove the factorization property.

\subsection{A simple example}
\label{subsec:example}

To help the reader get a better understanding of Theorems~\ref{thm:GGP_global_intro} and~\ref{thm:local_GGP_intro}, we explain their proofs on a simple example. We take $G=\GL_1 \times \GL_2$ and $H=\GL_1$. We choose $\Pi$ to be the trivial character of $G$ (either locally or globally). Then $P_\pi=P_0$ is the Borel subgroup of upper-triangular matrices, and $\sigma$ (or $\delta_v$) is $1$ the trivial character of the torus. The space $\fa_{P_\pi,\cc}^*$ is $3$-dimensional, and we write elements $\lambda \in \fa_{P_\pi,\cc}^*$ as $\lambda=(a,(b,c))$. The space $\fa_{P,\cc}^*$ is $2$-dimensional, and $\fa_\pi^*$ is the anti-diagonal subspace $(z,-z)$. Finally, $\nu=(0,(1/2,-1/2)) \in \fa_{P_\pi}^*$. 

\subsubsection{The global factorization}

We first study the global case. Let $\zeta$ be the completed Zeta function of $F$. By \eqref{eq:IY_facto}, we have $\tS$ a finite set of places such that for $\lambda \in \fa_{P_\pi,\cc}^*$ in general position and $\phi=\otimes_v \phi_v \in I_{P_\pi}^G 1$
\begin{equation}
    \label{eq:Euler_product}
    \cP(E(\phi,\lambda))=\frac{\zeta(a+b+1/2)\zeta(a+c+1/2)}{\zeta(b-c+1)} \prod_{v \in \tS} Z^\natural_{1}(\phi_v,\lambda).
\end{equation}
There are two singular hyperplanes passing through $\nu$ which are directed by the affine linear forms $\Lambda(\lambda)=-(a+c+1/2)$ and $\Lambda'(\lambda)=a+b-1/2$. Note that the singular affine hyperplane $b-c=1$ of $E(\phi,\lambda)$ is not a pole of $\cP(E(\phi,\lambda))$.

Let $w$ be the non-trivial element in the Weyl group of $\GL_2$. By the reversed Maa\ss--Selberg relation, for a sufficiently positive truncation parameter $T \in \rr$ we have
\begin{align*}
     \cP(E(\phi,\lambda))=&\cP^T(E(\phi,\lambda))-\frac{e^{(a+b+1/2)T}}{a+b+1/2}\cP^{P_0}(\phi,\lambda)-\frac{e^{(a+c+1/2)T}}{a+c+1/2}\cP^{P_0}(M(w,\lambda)\phi,w\lambda) \\
     &+\frac{e^{(a+c-1/2)T}}{a+c-1/2}\cP^{\overline{P_0}}(\phi,\lambda)+\frac{e^{(a+b-1/2)T}}{a+b-1/2}\cP^{\overline{P_0}}(M(w,\lambda)\phi,w\lambda),
\end{align*}
where $\cP^T$ is a truncated period and $\cP^{P_0}$ and $\cP^{\overline{P_0}}$ are relative periods. Let $K_H$ be the maximal compact subgroup of $H$. Then we have the simple expressions
\begin{equation*}
    \cP^{P_0}(\phi,\lambda)=\int_{K_H} \phi(k) dk, \quad \cP^{\overline{P_0}}(\phi,\lambda)=\int_{K_H} \phi(wk) dk.
\end{equation*}
We write $\underset{\Lambda}{\Res} \; \cP(\phi,\lambda)$ for the evaluation of $\Lambda(\lambda)\cP(\phi,\lambda)$ along the zero locus of $\Lambda$. Because we know that $\cP^T(E(\phi,\lambda))$ inherits the analytic behaviour of $E(\phi,\lambda)$, we conclude that 
\begin{equation*}
     \underset{\Lambda}{\Res} \; \cP(\phi,\lambda)=\cP^{P_0}(M(w,\lambda)\phi,w \lambda).
\end{equation*}
Note that along the hyperplane $\Lambda(\lambda)=0$, we have $\Lambda'(\lambda)=b-c-1$. Moreover, $\Lambda^{-1}(\{0\}) \cap \Lambda^{',-1}(\{0\})=\fa_{\pi,\cc}^* +\nu$. Therefore, for $\lambda$ in this affine subspace we obtain 
\begin{equation*}
     \underset{\Lambda'}{\Res} \; \underset{\Lambda}{\Res} \; \cP(\phi,\lambda)=\cP^{P_0}(M^*(w,\lambda)\phi,w \lambda).
\end{equation*}
This is the desired factorization property.

\subsubsection{The local factorization}

We now go to the local setting and take $v$ a place of $F$. Let $ W_{P_0,1}^{\psi}(\cdot,\lambda)$ be the Jacquet functional defined on the induction $I_{P_0}^G \lambda$ with respect to some non-trivial character $\psi$ of $F_v$. In what follows, the formulae hold for suitable choices of measures and $\psi$, but we suppress here this dependence.

Let $\phi_v \in I_{P_0}^G 1$. Write $\zeta_v$ for the local Zeta function of $F_v$. Our formula for the Jacquet functional from Proposition~\ref{prop:Whittaker_explicit_discrete} says that there exists $\varepsilon>0$ such that for $t \in F_v^\times$ with $\Val{t}<\varepsilon$ we have
\begin{equation*}
    W_{P_0,1}^{\psi}\left(\left(t,\begin{pmatrix}
        t & \\
        & 1
    \end{pmatrix}\right),\phi_v,\lambda \right)=\frac{\zeta_v(c-b)}{\zeta_v(b-c+1)}\Val{t}^{a+b+1/2}\phi_v(1)+\frac{\zeta_v(b-c)}{\zeta_v(b-c+1)}\Val{t}^{a+c+1/2} N_v(w,\lambda)\phi_v(1),
\end{equation*}
where $N_v(w,\lambda)$ is Shahidi's normalized intertwining operator $I_{P_0}^G \lambda \to I_{P_0}^G w \lambda$. Here $Z_{1}(\phi_v,\lambda)$ is the integral of $W_{P_0,1}^{\psi}(\phi_v,\lambda)$ along $F_v^\times$. Using the conditions on the supports of Whittaker functions from \cite[Proposition~6.1]{CS}, we conclude that that there exist $t_1, \hdots, t_{m+1} \in F_v^\times$ such that 
\begin{align*}
    Z_1(\phi_v,\lambda)&=\sum_{i=1}^m W_{P_0,1}^\psi(t_i,R(e_{K,H})\phi_v,\lambda) +\Val{t_{m+1}}^{a+b+1/2}\frac{\zeta_v(c-b)\zeta_v(a+b+1/2)}{\zeta_v(b-c+1)}R(e_{K,H})\phi_v(1) \\
    &+\Val{t_{m+1}}^{a+c+1/2}\frac{\zeta_v(b-c)\zeta_v(a+c+1/2)}{\zeta_v(b-c+1)}N_v(w,\lambda)R(e_{K,H})\phi_v(1),
\end{align*}
where $R(e_{K,H})$ is the convolution by the characteristic function of $\oo_{F_v}^\times$. Note that \eqref{eq:tempered_L} reads
\begin{equation*}
    \cL(\delta_{v,\lambda})=\frac{\zeta_v(a+b+1/2)\zeta_v(a+c+1/2)}{\zeta_v(b-c+1)}.
\end{equation*}
It follows that $Z_1^\natural(\phi_v,\lambda)$ is regular at $\lambda=\nu$ and that 
\begin{equation*}
    RZ_1^\natural(\phi_v,\nu)=N_v(w,\nu)R(e_{K,H})\phi_v(1).
\end{equation*}
This linear form factors through $I_{P_0}^G \nu \to 1$ as claimed.

\subsection{Organization of the paper}

This paper is organized as follows. In \S\ref{sec:preliminaries}, we introduce some notation on automorphic forms. In \S\ref{sec:poles}, we recall the description of the discrete spectrum of $\GL_n$ from \cite{MW89}. In \S\ref{sec:IYZ_periods}, we introduce the formalism of \cite{Zydor} needed to define the regularization of $\cP_H$. In particular, we show that it coincides with the one of \cite{IY}. Finally, we compute the residue of the period using the reversed Maa\ss--Selberg relation. We then use this result in \S\ref{sec:RS_non_tempered} to build the extension $\cP_\pi(\cdot,\lambda)$ and to prove that it factors through the quotient for $\lambda$ in general position. We also provide the alternative description $\cP^{P_+}$ of \eqref{eq:period_alternative}. In \S\ref{sec:ggp_conj_non_tempered}, we build on our global results to give the first proof of Theorem~\ref{thm:local_GGP_intro}. We use it to complete the proof of Theorem~\ref{thm:GGP_global_intro}. Finally, in \S\ref{sec:local_zeta} we switch to the non-Archimedean setting and provide an independent proof of Theorem~\ref{thm:local_GGP_intro} by computing residues of local Zeta integrals.

\subsection{Acknowledgement}

The author thanks Rapha\"el Beuzart-Plessis for helpful discussions and comments. He is also grateful to Wee Teck Gan and Erez Lapid for suggestions on an earlier version of this text.

This work was partly funded by the European Union ERC Consolidator Grant, RELANTRA, project number 101044930. Views and opinions expressed are however those of the author only and do not necessarily reflect those of the European Union or the European Research Council. Neither the European Union nor the granting authority can be held responsible for them. The author was also supported by the Max Planck Institute for Mathematics in Bonn, that he thanks for its hospitality
and financial support.

\section{Preliminaries on automorphic forms}
\label{sec:preliminaries}

\subsection{General notation}

Let $F$ be a field of characteristic zero. All algebraic groups are defined over $F$.

\subsubsection{Reductive groups, parabolic subgroups, characters} Let $G$ be a connected reductive group. Let $Z_G$ be the center of $G$. Let $N_G$ be the unipotent radical of $G$ and let $X^*(G)$ be the group of $F$-algebraic characters of $G$. Set $\fa^*_G=X^*(G) \otimes_{\zz} \rr$ and $\fa_G=\Hom_\zz(X^*(G),\rr)$. Let $ \langle \cdot,\cdot \rangle : \fa_G^* \times \fa_G \to \rr$ be the canonical pairing.

Let $P_0$ be a minimal parabolic subgroup of $G$. Let $M_0$ be a Levi factor of $P_0$. We say that a parabolic subgroup of $G$ is standard (resp. semi-standard) if it contains $P_0$ (resp. if it contains $M_0$). If $P$ is a semi-standard parabolic subgroup of $G$, we will denote by $N_P$ its unipotent radical and by $M_P$ its unique Levi factor containing $M_0$. We have a decomposition $P=M_P N_P$.

Let $A_G$ be the maximal central $F$-split torus of $G$. If $P$ is a semi-standard parabolic subgroup of $G$, set $A_P=A_{M_P}$. We set $\fa_0^*=\fa_{P_0}^*$, $\fa_0=\fa_{P_0}$ and $A_0=A_{P_0}$.

Let $P \subset Q$ be semi-standard parabolic subgroups of $G$. The restriction maps $X^*(Q) \to X^*(P)$ and $X^*(A_P) \to X^*(A_Q)$ induce dual decompositions $\fa_P=\fa_P^Q \oplus \fa_Q$ and $\fa_P^*=\fa_P^{Q,*} \oplus \fa_Q^*$. In particular, we have projections $\fa_0 \to \fa_P^Q$ and $\fa_0^* \to \fa_P^{Q,*}$ denoted by $X \mapsto X_P^Q$ which only depend on the Levi factors $M_P$ and $M_Q$. If $Q=G$, we omit the exponent $G$ in the previous notation.

Set $\fa_{P,\cc}^Q=\fa_{P}^Q \otimes_\rr \cc$ and $\fa_{P,\cc}^{Q,*}=\fa_{P}^{Q,*} \otimes_\rr \cc$. We still denote by $\langle \cdot, \cdot \rangle$ the pairing obtained by extension of scalars. We have decompositions $ \fa_{P,\cc}^{Q}=\fa_{P}^{Q} \oplus i \fa_{P}^{Q}, \quad \fa_{P,\cc}^{Q,*}=\fa_{P}^{Q,*} \oplus i \fa_{P}^{Q,*}$, where $i^2=-1$. We denote by $\Re$ and $\Im$ the real and imaginary parts associated to these decompositions.

\subsubsection{Roots, coroots, weights} \label{subsubsec:roots} Let $P$ be a standard parabolic subgroup of $G$. Let $\Delta_0^P \subset \fa_0^{P,*}$ (resp. $\Sigma_0^P \subset \fa_0^{P,*}$) be the set of simple roots (resp. of roots) of $A_0$ in $M_P \cap P_0$. If $P=G$, we write $\Delta_0$ and $\Sigma_0$. Let $\Delta_P$ (resp. $\Sigma_P$) be the image of $\Delta_0 \setminus \Delta_0^P$ (resp. $\Sigma_0 \setminus \Sigma_0^P$) by the projection $\fa_0^* \to \fa_P^*$. More generally, for $P \subset Q$ let $\Delta_P^Q$ (resp. $\Sigma_P^Q$) be the projection of $\Delta_0^Q \setminus \Delta_0^P$ in $\fa_P^{Q,*}$ (resp. $\Sigma_0^Q \setminus \Sigma_0^P$). Let $\Delta_P^{Q,\vee} \subset \fa_P^Q$ be the set of simple coroots. If $\alpha \in \Delta_{P}^Q$, we denote by $\alpha^\vee$ the associated coroot. By duality, let $\hat{\Delta}_P^Q$ be the set of simple weights. Set 
\begin{align*}
    \fa_{P}^{Q,+}&=\left\{ H \in \fa_P \; | \; \langle \alpha, H \rangle > 0, \; \forall \alpha \in \Delta_{P}^Q \right\}. \\
    \fa_{P}^{Q,*,+}&=\left\{ \lambda \in \fa_P^* \; | \; \langle \lambda, \alpha^\vee \rangle > 0, \; \forall \alpha \in \Delta_{P}^Q \right\}.
\end{align*}
If $Q=G$, we drop the exponent. We denote by $\overline{\fa_{P}^{Q,+}}$ and $\overline{\fa_{P}^{Q,*,+}}$ the closure of these open subsets in $\fa_P^Q$ and $\fa_P^{Q,*}$ respectively. If $\lambda \in \fa_P^*\setminus \{0\}$, we write $\lambda > 0$ if $\lambda$ is a nonnegative linear combination of the simple roots $\Delta_P$.

We say that a functional $\Lambda$ on $\fa_{P,\cc}^*$ is an affine linear form if it is of the form $\Lambda(\lambda)=\langle \lambda, \gamma^\vee \rangle - a$ for $\gamma^\vee \in \fa_P$ and $a \in \cc$. We call its set of zeros an affine hyperplane. If $\gamma^\vee$ is a coroot, then it is an affine root hyperplane. By "$\lambda \in \fa_{P,\cc}^*$ in general position", we mean that $\lambda$ lies outside of a countable union of affine hyperplanes. 

\subsubsection{Weyl group} Let $W$ be the Weyl group of $(G,A_0)$, which is by definition the quotient of the normalizer $N_{G(F)}(A_0(F))$ by the centralizer $Z_{G(F)}(A_0(F))$. It acts on $\fa_0$ and by duality on $\fa_0^*$. We henceforth fix an invariant inner product $(\cdot,\cdot)$ on $\fa_0^*$ for this action. If $w \in W$, we write again $w$ for a representative in $G(F)$ which we choose as in \cite[Section~2]{KS}. 

Let $P=M_P N_P$ and $Q=M_Q N_Q$ be two standard parabolic subgroups of $G$. Let ${}_Q W_P$ be the set of $w \in W$ such that $M_P \cap w^{-1} P_0 w=M_P \cap P_0$ and $M_Q \cap w P_0 w^{-1} = M_Q \cap P_0$.

Let $w \in {}_Q W_P$. Set $P_w=(M_P \cap w^{-1} Q w)N_P$. By \cite[Lemme~V.4.6.]{Renard}, $P_w$ is a standard parabolic subgroup of $G$ included in $P$, with standard Levi factor $M_P \cap w^{-1} M_Q w$. In the same way, $Q_w=(M_Q \cap w P w^{-1})N_Q$ is standard parabolic subgroup of $G$ included in $Q$, with standard Levi factor $M_Q \cap w M_P w^{-1}$. Note that $w \Sigma_{P_w}^P \subset \Sigma_{Q_w}$ and $w^{-1} \Sigma_{Q_w}^Q \subset \Sigma_{P_w}$. Set
\begin{align*}
    W(P;Q)&=\{ w \in {}_Q W_P \; | \;  P_w=P \}=\{w \in {}_Q W_P \; | \; M_P \subset w^{-1} M_Q w \}, \\
    W(P,Q)&=\{w \in {}_Q W_P \; | \; M_P = w^{-1} M_Q w \}.
\end{align*}
Note that $w \in {}_Q W_P$ implies $w \in W(P_w,Q_w)$. Set
\begin{equation*}
    W(P)=\bigcup_{Q} W(P,Q).
\end{equation*}
Write $w_P$ for the longest element in $W(P)$.

If $R$ is another standard parabolic subgroup of $G$, we write ${}_Q W^R_P$ (resp. $W^R(P;Q)$ and $W^R(P,Q)$) for ${}_{Q \cap M_R} W_{P\cap M_R}$ (resp. $W(P \cap M_R; Q \cap M_R)$ and $W(P \cap M_R, Q \cap M_R)$) relatively to the reductive group $M_R$. 

\begin{lem}[{\cite[Lemma~2.2.1.1]{Ch}}] \label{lem:W_CH}
Let $P,Q,R$ be standard parabolic subgroups of $G$ such that $Q \subset R$.
\begin{enumerate}
    \item For any $w \in {}_R W_P$, we have ${}_Q W_{R_w}^R w \subset {}_Q W_P$.
    \item For any $w_2 \in {}_Q W_P$, there is a unique decomposition $w_2=w_1 w$ with $w \in {}_R W_P$ and $w_1 \in {}_Q W_{R_w}^R w$. Moreover, $w_2 \in W(P;Q)$ if and only if $w \in W(P;R)$ and $w_1 \in W^R(R_w;Q)$.
\end{enumerate}
    
\end{lem}

\subsection{Functions on automorphic quotients}

We now assume that $F$ is a number field. Let $G$ be a connected reductive group over $F$.

\subsubsection{Automorphic quotients} Let $\bA$ be the adele ring of $F$, let $\bA_f$ be its ring of finite adeles. Set $F_\infty=F \otimes_\qq \rr$. Let $V_F$ be the set of places of $F$ and let $V_{F,\infty} \subset V_F$ be the subset of Archimedean places. For $v \in V_F$, let $F_v$ be the completion of $F$ at $v$. If $v$ is non-Archimedean, let $q_v$ be the cardinality of the residual field of $F_v$ and $\oo_{v}$ be its ring of integers. Let $\Val{\cdot}$ be the absolute value $\bA^\times \to \rr_+^\times$ given by taking the product of the normalized absolute values $\Val{\cdot}_v$ on each $F_v$. 

Let $P=M_P N_P$ be a semi-standard parabolic subgroup of $G$. Set 
\begin{equation*}
    [G]_{P}=M_P(F) N_P(\bA) \backslash G(\bA).
\end{equation*}
Let $A_{P,\qq}$ be the maximal $\qq$-split subtorus of the Weil restriction $\Res_{F/\qq}A_P$, and let $A_P^\infty$ be the neutral component of $A_{P,\qq}(\rr)$. Set 
\begin{equation*}
    [G]_{P,0}=A_P^\infty M_P(F) N_P(\bA) \backslash G(\bA).
\end{equation*}
If $P=G$, we simply write $[G]$ and $[G]_0$ for $[G]_G$ and $[G]_{G,0}$ respectively.

Let $P$ be a semi-standard parabolic subgroup of $G$. There is a canonical morphism $ H_P : P(\bA) \to \fa_P$ such that $\langle \chi, H_P(g) \rangle=\log \Val{\chi(g)}$ for any $g \in P(\bA)$ and $\chi \in X^*(P)$. The kernel of $H_P$ is denoted by $P(\bA)^1$. We extend it to $ H_P : G(\bA) \to \fa_P$ which satisfies: for any $g \in G(\bA)$ we have $H_P(g)=H_P(p)$ whenever $g \in pK$ with $p \in P(\bA)$. If $P=P_0$, we write $H_0=H_{P_0}$.

We set
\begin{equation*}
    [G]_P^1=M_P(F) N_P(\bA) \backslash P(\bA)^1 K.
\end{equation*}
If $P=G$, we simply write $[G]^1$. Let $\rho_P$ be the unique element in $\fa_P^{*}$ such that for every $m \in M_P(\bA)$ we have $\Val{\det(\mathrm{Ad}_P(m))}=\exp(\langle 2 \rho_P,H_P(m) \rangle)$, where $\mathrm{Ad}_P$ is the adjoint action of $M_P$ on the Lie algebra of $N_P$. For every $g \in G(\bA)$, we then set $ \delta_P(g)=\exp(\langle 2 \rho_P,H_P(g) \rangle)$.

Let $K=\prod_{v \in V_F} K_v \subset G(\bA)$ be a "good" maximal compact subgroup in good position relative to $M_0$. We write $K=K_\infty K^\infty$ where $K_\infty=\prod_{v \in V_{F,\infty}} K_v$ and $K^\infty=\prod_{v \in V_F \setminus V_{F,\infty}} K_v$. By a \emph{level} $J$ of $G$, we mean an open-compact subgroup $J$ of $G(\bA_f)$.

\subsubsection{Haar measures} \label{subsubsec:first_measure} 

We take a Haar measure $dg$ on $G(\bA)$, with factorization $dg=\prod_v d g_v$ where for all place $v$, $dg_v$ is a Haar measure on $G(F_v)$. This implicitly implies that for almost all place $v$ the volume of $K_v$ is $1$. 

Let $P$ be a semi-standard parabolic subgroup of $G$. We equip $\fa_P$ with the Haar measure that gives covolume $1$ to the lattice $\Hom(X^*(P),\zz)$. We equip $A_P^\infty$ with the Haar measure compatible with the isomorphism $A_P^\infty \simeq \fa_P$ induced by $H_P$. If $P \subset Q$, we equip $\fa_P^Q=\fa_P/\fa_Q$ with the quotient measure. 

For each $v \in V_F$, we give $K_v$ the invariant probability measure. This yields a product measure on $K$. If $N$ is an unipotent group, we give $N(\bA)$ the Haar measure whose quotient by the counting measure on $N(F)$ gives $[N]$ volume $1$. We equip $M_P(\bA)$ with the unique Haar measure such that
\begin{equation}
\label{eq:Levi_measure}
    \int_{G(\bA)} f(g) dg= \int_{N_P(\bA)} \int_{M_P(\bA)} \int_K f(nmk) \exp(-\langle 2 \rho_P,H_P(m) \rangle) dkdmdn
\end{equation}
for every continuous and compactly supported function $f$ on $G(\bA)$. We equip $M_P(\bA)^1$ with the Haar measure compatible with the isomorphism $M_P(\bA)^1 \times A_P^\infty \to M_P(\bA)$.

We give $[G]_P$ the quotient of our measure on $G(\bA)$ by the product of the counting measure on $M_P(F)$ with our measure on $N_P(\bA)$. Moreover, note that the action of $a \in A_P^\infty$ by left translation on $[G]_P$ multiplies the measure by $\delta_P^{-1}(a)$. By taking the quotient of the measure on $[G]_P$ by that of $A_P^\infty$, we obtain a "semi-invariant" measure on $[G]_{P,0}$.

\subsubsection{Functions}

We say that a function $\varphi : G(\bA) \to \cc$ is smooth if it is right-invariant by a level $J$, and if for all $g_f \in G(\bA_f)$ the map $g_\infty \in G(F_\infty) \mapsto \varphi(g_f g_\infty)$ is smooth in the usual sense. The space of smooth functions receives an action by right-translation of $G(\bA)$ and $\cU(\fg_\infty)$ (the enveloping algebra of the complexification of $\fg_\infty$ the Lie algebra of $G(F_\infty)$). We denote it by $R$.

For every semi-standard parabolic subgroup $P$ of $G$, we have a height function $\norm{\cdot}_P$ as in \cite[Section~2.4.1]{BPCZ}. For every smooth function $\varphi$ on $[G]_P$, every $X \in \cU(\fg_\infty)$ and $N \in \rr$, set
\begin{equation*}
    \norm{\varphi}_{N,X}:=\sup_{x \in [G]_P} \norm{x}_P^N \Val{(R(X)\varphi)(x)}.
\end{equation*}
If $X=1$, we simply write $\norm{\varphi}_N$. The function $\varphi$ is said to be of \emph{rapid decay} if all the $\norm{\varphi}_{N,X}$ are finite, and of \emph{uniform moderate growth} if there exists $N \in \rr$ such that all the $\norm{\varphi}_{N,X}$ are finite.

Let $L^2([G]_{P,0})$ the space of functions on $[G]_P$ that transform by $\delta_P^{1/2}$ under left-translation by $A_P^\infty$ such that the Petersson norm
\begin{equation*}
    \langle \varphi, \varphi \rangle_{P,\Pet}=\norm{\varphi}_{P,\Pet}:=\int_{[G]_{P,0}} \Val{\varphi(g)}^2 dg
\end{equation*}
is finite.

\subsection{Automorphic representations}

We keep the assumption that $F$ is a number field and that $G$ is connected reductive over $F$. We take $P$ to be a semi-standard parabolic subgroup of $G$.

\subsubsection{Automorphic forms and representations}
We denote by $\cA(G)$ the space of automorphic forms on $G$ as in \cite[Section~2.7]{BPCZ}. More generally we have a space of automorphic forms $\cA_P(G)$ which are functions on $[G]_P$. 

Let $\cA_P^0(G)$ be the subspace of $\varphi \in \cA_P(G)$ which transform by $\delta_P^{1/2}$ under left-translation by $A_P^\infty$. Set $\cA_{P,\disc}(G)=\cA_P^0(G) \cap L^2([G]_{P,0})$. The spaces $\cA_{P}^0(G)$ and $\cA_{P,\disc}(G)$ are given the topology of \cite[Section~2.7.1]{BPCZ}.

We define a discrete automorphic representation of $G(\bA)$ to be a topologically irreducible subrepresentation of $\cA_{\disc}(G)$. Let $\Pi_{\disc}(G)$ be the set of such representations. For $\pi \in \Pi_{\disc}(G)$, let $\cA_\pi(G)$ be the $\pi$-isotypic component of $\cA_{\disc}(G)$. Note that $\pi$ has trivial central character on $A_G^\infty$.

For $\pi \in \Pi_{\disc}(M_P)$, let $\cA_{P,\pi}(G)$ be the subspace of $\varphi \in \cA_{P,\disc}(G)$ such that for all $g \in G(\bA)$ the map $ m \in [M_P] \mapsto \delta_P(m)^{-1/2} \varphi(mg)$ belongs to $\cA_\pi(M_P)$. For any $\lambda \in \fa_{P,\cc}^*$, set $\pi_\lambda=\pi \otimes \exp( \langle \lambda, H_{M_P}(\cdot) \rangle )$, and for $\varphi \in \cA_{P,\pi}(G)$ define
\begin{equation*}
    \varphi_\lambda(g)=\exp( \langle \lambda, H_{P}(g) \rangle ) \varphi(g).
\end{equation*}
The map $\varphi \mapsto \varphi_\lambda$ identifies $\cA_{P,\pi}(G)$ with a subspace of $\cA_P(G)$ denoted by $\cA_{P,\pi,\lambda}(G)$.

For any standard parabolic subgroup $Q \subset P$ and any $\varphi \in \cA_P(G)$, we have a constant term $\varphi_Q$ defined by
\begin{equation*}
    \varphi_Q(g)=\int_{[N_Q]} \varphi(ng) dn, \quad g \in [G]_Q.
\end{equation*}
Let $\cA_{P,\cusp}(G) \subset \cA^0_P(G)$ be the subspace of $\varphi$ such that $\varphi_Q=0$ for all $Q \subsetneq P$. Let $\Pi_{\mathrm{cusp}}(G)$ be the set of topologically irreducible subrepresentations of $\cA_{\mathrm{cusp}}(G)$, where we equip this space with the subspace topology from $\cA(G)$. It is a subset of $\Pi_{\disc}(G)$. 

We say that $\varphi \in \cA_{P,\disc}(G)$ is \emph{residual} if it is orthogonal to $\cA_{P,\mathrm{cusp}}(G)$ under $\langle \cdot,\cdot \rangle_{P,\Pet}$. 

Let $\pi \in \Pi_{\disc}(M_P)$. By \cite{Flath}, it decomposes as $\pi=\otimes'_v \pi_v$. For every place $v$, we write $I_P^G \pi_v$ for the smooth parabolic induction of $\pi_v$ for $G(F_v)$.

\subsubsection{Intertwining operators} Let $Q$ be a standard parabolic subgroup of $G$. Let $\pi \in \Pi_\disc(M_P)$. Let $w \in W(P,Q)$ and $\lambda \in \fa_{P,\cc}^*$ such that $\langle \Re(\lambda), \alpha^\vee \rangle$ is large enough for any $\alpha \in \Delta_P$ such that $w \alpha <0$. For $\varphi \in \cA_{P,\pi}(G)$, consider the absolutely convergent integral
\begin{equation*}
    (M(w,\lambda) \varphi)_{w \lambda}(g)=\int_{(N_Q \cap w N_P w^{-1})(\bA) \backslash N_Q(\bA)} \varphi_\lambda(w^{-1} ng)dn, \quad g \in [G]_Q.
\end{equation*}
By \cite{Langlands} and \cite{BL}, it admits a meromorphic continuation to $\fa_{P,\cc}^*$. This defines a continuous intertwining operator for any regular $\lambda$
\begin{equation}
\label{eq:global_operator}
    M(w,\lambda) : \cA_{P,\pi}(G) \to \cA_{Q,w\pi}(G).
\end{equation}
By \cite[Theorem~2.3]{BL}, the singularities of $M(w,\lambda)$ are located along affine root hyperplanes.

If $R$ is another standard parabolic and if $w_1 \in W(P,Q)$ and $w_2 \in W(Q,R)$, by \cite[Theorem~2.3.5]{BL} we have the functional equation
\begin{equation}
\label{eq:M_equation}
    M(w_2,w_1\lambda)M(w_1,\lambda)\varphi=M(w_2 w_1,\lambda)\varphi.
\end{equation}

\subsubsection{Eisenstein series}
\label{subsubsec:Eisenstein}
Let $P \subset Q$ be standard parabolic subgroups of $G$. For any $\varphi \in \cA_{P,\disc}(G)$ and $\lambda \in \fa_{P,\cc}^*$ we define
\begin{equation}
\label{eq:partial_Eisenstein}
    E^Q(g,\varphi,\lambda)=\sum_{\gamma \in P(F) \backslash Q(F)} \varphi_\lambda(\gamma g)=\sum_{\gamma \in M_Q \cap P(F) \backslash M_Q(F)} \varphi_\lambda(\gamma g), \quad g \in G(\bA).
\end{equation}
We call it a generalized Eisenstein series. The sum is absolutely convergent for $\Re(\lambda)$ in a suitable cone. It admits a meromorphic continuation to $\fa_{P,\cc}^*$ by \cite{Langlands} and \cite{BL}. If $Q=G$, we write $E(g,\varphi,\lambda)$. By \cite[Theorem~2.3]{BL}, the singularities of $E^Q(\varphi,\lambda)$ are along affine root hyperplanes.

For regular $\lambda$, let $E_Q(\varphi,\lambda)$ be the constant term of $E(\varphi,\lambda)$ along $Q$. By \cite[Lemma~6.10]{BL}, 
\begin{equation}
\label{eq:constant_term}
    E_Q(\varphi,\lambda)=\sum_{w \in {}_Q W_P} E^Q(M(w,\lambda)\varphi_{P_w},w \lambda).
\end{equation}
If $\varphi \in \cA_{P,\mathrm{cusp}}(G)$, this formula reduces to 
\begin{equation}
\label{eq:constant_term_cuspidal}
    E_Q(\varphi,\lambda)=\sum_{w \in W(P;Q)} E^Q(M(w,\lambda)\varphi,w \lambda).
\end{equation}
We also have the following easy relation between intertwining operators and Eisenstein series.

\begin{lem}
\label{lem:ME=EM}
    Let $\varphi \in \cA_{P,\disc}(G)$. Let $Q, Q'$ be two standard parabolic subgroups of $G$ such that $P \subset Q$. Let $w \in W(Q,Q')$. Then for regular $\lambda \in \fa_{P,\cc}^*$ we have
    \begin{equation}
    \label{eq:ME=EM}
        M(w,\lambda)E^Q(\varphi,\lambda)=E^{Q'}(M(w,\lambda) \varphi,w\lambda).
    \end{equation}
\end{lem}

\begin{proof}
    This holds in the region of absolute convergence, and then for $\lambda$ in general position by analytic continuation.
\end{proof}

\subsection{Exponents and cuspidal components} \label{subsubsec:exponents} Let $\varphi \in \cA_P(G)$. By \cite[Section~I.3.2.]{MW95}, $\varphi$ admits a finite decomposition
\begin{equation}
\label{eq:exponents}
    \varphi(g)=\sum_i q_i(H_P(g))  \exp(\langle \lambda_i, H_P(g) \rangle)\varphi_i(g), \quad g \in G(\bA),
\end{equation}
where $q_i \in \cc[\fa_P]$, $\lambda_i \in \fa_{P,\cc}^*$ and $\varphi_i \in \cA_P^0(G)$. The set of distinct exponents $\lambda_i$ occurring in \eqref{eq:exponents} is uniquely determined by $\varphi$ and is denoted $\cE_P(\varphi)$. For $Q \subset P$, let $\cE_Q(\varphi)$ be the exponents of $\varphi_Q$.

Let $\varphi \in \cA_{P}^0(G)$. As cuspidal automorphic forms are of rapid decay (\cite[Section~I.2.18.]{MW95}), for every $\varphi_0 \in \cA_{P,\mathrm{cusp}}(G)$ the pairing $\langle \varphi, \varphi_0 \rangle_{P,\Pet}$ makes sense. By \cite[Section~I.2.18.]{MW95}, there exists a unique $\varphi^{\mathrm{cusp}} \in \cA_{P,\mathrm{cusp}}(G)$ such that for all $\varphi_0 \in \cA_{P,\mathrm{cusp}}(G)$ we have $    \langle \varphi, \varphi_0 \rangle_{P,\Pet}=\langle \varphi^{\mathrm{cusp}}, \varphi_0 \rangle_{P,\Pet}$.

If $\varphi \in \cA_P(G)$ admits a decomposition \eqref{eq:exponents}, we set
\begin{equation*}
    \varphi^{\mathrm{cusp}}(g)=\sum_i q_i(H_P(g))  \exp(\langle \lambda_i, H_P(g) \rangle)\varphi_i^{\mathrm{cusp}}(g).
\end{equation*}
It is independent of the choice of the decomposition. The tuple $(\varphi_Q^{\mathrm{cusp}})_{Q \subset P}$ (where $\varphi_Q^{\mathrm{cusp}}$ is the cuspidal component of the constant term $\varphi_Q$) is called the family of cuspidal components of $\varphi$. The following result is classical.

\begin{lem}[{\cite[Proposition~I.3.4]{MW95}}] \label{lem:vanishing_cusp_comp}
    Let $\varphi \in \cA_P(G)$. If all the cuspidal components of $\varphi$ are zero, then $\varphi=0$.
\end{lem}

If $\varphi \in \cA_{P,\disc}(G)$, for any regular $\lambda$ we have $E(\varphi,\lambda) \in \cA(G)$. If this Eisenstein series is proper (that is if $P\neq G$), then $E(\varphi,\lambda)$ is orthogonal to cusp forms, i.e. $E(\varphi,\lambda)^\cusp=0$.

\subsection{Pseudo-Eisenstein series} \label{subsubsec:pseudo_eisenstein} 

Let $\cP\cW(\fa_{P,\cc}^*)$ be the Paley--Wiener space of functions on $\fa_{P,\cc}^*$ obtained as Fourier transforms of compactly supported smooth functions on $\fa_P$. If $\cV$ is a finite-dimensional subspace of $\cA_{P,\cusp}(G)$, we define $\cP\cW_{P,\cV}$ to be the space of $\cV$-valued entire functions on $\fa_{P,\cc}^*$ of Paley--Wiener type. We write $\cP \cW_P$ for the direct sum of all the $\cP \cW_{P,\cV}$. For $\Phi \in \cP\cW_{P}$ and any $\kappa \in \fa_P^*$, consider 
\begin{equation}
\label{eq:F_Phi_defi}
    F_\Phi(g)=\int_{\substack{\lambda \in \fa_{P,\cc}^* \\ \Re(\lambda)=\kappa}} \Phi(\lambda)(g)\exp(\langle \lambda,H_P(g) \rangle) d\lambda, \quad g \in G(\bA).
\end{equation}
It is independent of the choice of $\kappa$. We define the pseudo-Eisenstein series associated to $\Phi$ by
\begin{equation*}
    E(g,F_\Phi)=\sum_{\gamma \in P(F) \backslash G(F)} F_\Phi(\gamma g), \quad g \in [G].
\end{equation*}
where this sum is actually over a finite set which depends on $g$ by \cite[Lemma~5.1]{Art78}. This pseudo-Eisenstein series is rapidly decreasing. Moreover, by \cite[Section~II.1.11]{MW95} we have
\begin{equation}
\label{eq:pseudo_Eisenstein_unfold}
    E(g,F_\Phi)=\int_{\Re(\lambda)=\kappa} E(g,\Phi(\lambda),\lambda) d\lambda, \quad g \in [G],
\end{equation}
for any $\kappa$ in the region of the absolute convergence of Eisenstein series.

\section{Discrete Eisenstein series on \texorpdfstring{$\GL_n$}{GLn}}
\label{sec:poles}

The goal of this section is to fix specific notation in the case $G=\GL_n$ and to recall the classification of the discrete automorphic spectrum of $\GL_n$ from \cite{MW89}.

\subsection{Choice of coordinates}
\label{sec:choice_coord}

\subsubsection{Standard parabolic subgroups} 
\label{subsubsec:blocks}

We choose $P_0$ to be the standard Borel subgroup of upper triangular matrices, and $M_0=T_0$ to be the diagonal maximal torus. The group $K$ is the standard maximal compact subgroup of $\GL_n(\bA)$. 

If $P$ is a standard parabolic subgroup of $\GL_n$, its standard Levi factor is of the form $M_P=\GL_{n_1} \times \hdots \times \GL_{n_m}$ for some integers $n_1, \hdots, n_m$. With this notation, we associate to $P$ the tuple $\underline{n}(P):=(n_1,\hdots,n_m)$. This completely characterizes $P$ among the standard parabolic subgroups of $\GL_n$. We will often write $M_P=\prod \GL_{n_i}$, where we implicitly assume that the product is taken in the order $i=1, \hdots, m$.

We identify $\fa_P$ and $\fa_P^*$ with $\rr^m$ in the following way. The canonical basis $(e_i^*)$ of $X^*(P)$ is sent to the canonical basis of $\rr^m$ (still denoted $(e_i^*)$). Then we let $(e_i)$ be the dual basis of $\fa_P$ under the pairing $\langle \cdot,\cdot,\rangle$, which is sent to the canonical basis of $\rr^m$ (still denoted by $(e_i)$). In these coordinates, the injections $\fa_P^* \to \fa_0^*$ and $\fa_P \to \fa_0$ read
\begin{equation*}
    (\lambda_1, \hdots, \lambda_m) \mapsto \left( \underbrace{\lambda_1,\hdots,\lambda_1}_{n_1}, \hdots , \underbrace{\lambda_m,\hdots,\lambda_m}_{n_m} \right), \quad  (H_1,\hdots,H_m) \mapsto  \left( \underbrace{\frac{H_1}{n_1},\hdots,\frac{H_1}{n_1}}_{n_1}, \hdots , \underbrace{\frac{H_m}{n_m},\hdots,\frac{H_m}{n_m}}_{n_m} \right).
\end{equation*}
We equip the space $\fa_0$ with its canonical inner product $(\cdot,\cdot)$. By restriction, we obtain a Euclidean structure $(\cdot,\cdot)$ on $\fa_P$, but the basis $(e_i)$ is not orthonormal. This yields an identification $\fa_P^* \simeq \fa_P$, and under this identification any $e_i^*$ is sent to $n_i e_i$. 

Let $1 \leq i<j \leq m$ and let $\alpha \in \Sigma_P$ be the simple root corresponding to the transposition of the i\textsuperscript{th} and j\textsuperscript{th} blocks of $M_P$. Then with our choices of coordinates we have
\begin{equation}
\label{eq:explicit_roots}
    \alpha=\frac{e_i^*}{n_i}-\frac{e_j^*}{n_j}, \quad \alpha^\vee=e_i-e_j.
\end{equation}

We identify the Weyl group $W$ of $\GL_n$ with $\fS_n$ the symmetric group of degree $n$. If $w \in W$ corresponds to the permutation $\sigma$, we define the representative $\dot{w} \in \GL_n(F)$ by 
\begin{equation*}
    \dot{w}\mathbf{e}_i=(-1)^{|\{ j<i \; | \; \sigma(i)<\sigma(j) \} |} \mathbf{e}_{\sigma(i)}, \quad 1 \leq i \leq n,
\end{equation*}
where $\mathbf{e}_1, \hdots, \mathbf{e}_n$ are the vectors in the canonical basis of $F^n$. This is the choice made in \cite[Section~2]{KS} and it plays a role in the normalization of local intertwining operators (see \S\ref{subsec:local_operator}). We will often write again $w$ for the lift $\dot{w} \in \GL_n(F)$. 

\subsubsection{Representatives of double quotients}
\label{subsec:rep_double_quotients}

If $P=M_P N_P$ is a standard parabolic subgroup of $\GL_n$, we have an embedding of $W(M_P)$ the Weyl group of $M_P$ inside $W$. Write $M_P=\GL_{n_1} \times \hdots \times \GL_{n_m}$. We have an identification (of sets) $W(P) \simeq \fS_m$ such that, if $\sigma \in \fS_m$, we have
\begin{equation*}
    \sigma M_P \sigma^{-1}=M_{n_{\sigma^{-1}(1)}} \times \hdots \times \hdots M_{n_{\sigma^{-1}(m)}}.
\end{equation*}
We will often identify a $w \in W(P)$ with an element in $\fS_m$. We will write $w.P$ for the standard parabolic subgroup of $\GL_n$ with standard Levi factor $wM_Pw^{-1}$. We say that $w \in W$ acts by permutation on the blocks of $M_P$ (or simply acts by blocks on $P$) if it belongs to $W(P)$. 

Let $Q$ be another standard parabolic subgroup of $\GL_n$. Let $w \in {}_Q W _P$. Then $w$ acts by blocks on $P_w$, and $w^{-1}$ by blocks on $Q_w$. Write $M_{P_w}=\prod_{i=1}^{m} \prod_{j=1}^{m_i} \GL_{n_{i,j}}$, where for each $i$ we have  $\sum_{j=1}^{m_i} n_{i,j}=n_i$. If $M_Q=\prod_i \GL_{n_i'}$, we have a similar decomposition $M_{Q_w}=\prod_i \prod_j \GL_{n_{i,j}'}$. The condition $w \in {}_Q W_P$ means that for each $i$, $w$ preserves the order of the blocks $\GL_{n_{i,j}}$, $j=1, \hdots, m_i$, and that $w^{-1}$ preserves the order of the blocks $\GL_{n_{i,j}'}$, $j=1, \hdots, m_i'$. Conversely, let $P$ and $Q$ be standard parabolic subgroups of $\GL_n$. Let $w \in W$. Then if $w$ and $w^{-1}$ preserve the order of the blocks of $P$ and $Q$ as explained above, then $w \in {}_Q W_P$. 

\subsection{Discrete automorphic forms for \texorpdfstring{$\GL_n$}{GLn}}
\label{subsec:residual}

\subsubsection{The classification of~\cite{MW89}}
\label{subsubsec:disc_gln}
Let $\pi \in \Pi_{\mathrm{disc}}(\GL_n)$. There exist integers $r, d \geq 1$ with $n=rd$ and $\sigma \in \Pi_{\mathrm{cusp}}(\GL_r)$ such that any $\varphi \in \cA_\pi(\GL_n)$ is obtained as the residue of an Eisenstein series built from a $\phi \in \cA_{P_{\pi},\sigma^{\boxtimes d}}(\GL_n)$ where $P_{\pi} \subset \GL_n$ is the standard parabolic subgroup of Levi factor $\GL_{r}^{d}$. More precisely, define
\begin{equation}
\label{eq:nu_pi_defi}
    \nu_{\pi}=-\rho_{P_\pi}/r, \quad \sigma_\pi=\sigma^{\boxtimes d} \in \Pi_{\cusp}(M_{P_\pi}),
\end{equation}
and set for $\lambda \in \fa_{P_\pi,\cc}^*$
\begin{equation}
\label{eq:L_pi_res_defi}
    L_{\pi,\mathrm{res}}(\lambda)=\prod_{\alpha \in \Delta_{P_\pi}} \left( \langle \lambda, \alpha^\vee \rangle -1 \right).
\end{equation}
Note that $ L_{\pi,\mathrm{res}}(-\nu_{\pi})=0$. We introduce a minus sign in \eqref{eq:nu_pi_defi} to follow the convention of \cite{Ch}. For every $g \in \GL_n(\bA)$, denote by $E^*(g,\phi,\cdot)$ the map $\lambda \mapsto  L_{\pi,\mathrm{res}}(\lambda) E(g,\phi,\lambda)$. It is holomorphic in a neighborhood of $-\nu_\pi$. By \cite{MW89} we have
\begin{equation}
\label{eq:residual}
    \varphi(g)=E^*(g,\phi,-\nu_{\pi}).
\end{equation}
As $\pi$ is the unique irreducible quotient of $\cA_{P_\pi,\sigma_\pi,-\nu_\pi}(\GL_n)$ (see Corollary~\ref{cor:same_quotient}), it deserves to be called a Speh representation and we write $ \pi=\Speh(\sigma,d)$.

\subsubsection{The induced case} 
\label{subsubsec:residual_blocks}
We now deal with representations induced from the discrete spectrum of a Levi subgroup, i.e. with representations of Arthur type of $\GL_n$. Let $P=M_P N_P$ be a standard parabolic subgroup. Let $\pi \in \Pi_{\mathrm{disc}}(M_P)$. Write $M_P=\GL_{n_1} \times \hdots \times \GL_{n_m}$ and $\pi = \pi_1 \boxtimes \hdots \boxtimes \pi_m$ accordingly. By \S\ref{subsubsec:disc_gln}, there exist integers $r_i, d_i \geq 1$ with $n_i=r_i d_i$ and some representations $\sigma_i \in \Pi_{\mathrm{cusp}}(\GL_{r_i})$ such that any $\varphi_i \in \cA_{\pi_i}(\GL_{n_i})$ is obtained as the residue of an Eisenstein series built from a $\phi_i \in \cA_{P_{\pi_i},\sigma_i^{\boxtimes d_i}}(\GL_{n_i})$ where $P_{\pi_i} \subset \GL_{n_i}$ is the standard parabolic subgroup of Levi factor $\GL_{r_i}^{d_i}$. Set $P_{\pi}=(P_{\pi_1} \times \hdots \times P_{\pi_m})N_P$ and $\sigma_\pi=\sigma_1^{\boxtimes d_1} \boxtimes \hdots \boxtimes \sigma_m^{\boxtimes d_m}$ which is a cuspidal automorphic representation of $M_{P_\pi}$. 

Note that any $\varphi \in \cA_{P,\pi}(\GL_n)$ is residual unless $\pi$ is cuspidal (i.e. $P=P_\pi$). Indeed, this follows from the fact that Eisenstein series are orthogonal to cusp forms (in the sense of \S\ref{subsubsec:Eisenstein}) and that we may compute the residue under the Petersson inner-product by \cite[Theorem~2.2]{Lap} (see also \cite[Lemma~9.4.2.1]{BoiPhD} for a closely related argument).

Let $Q$ and $R$ be a parabolic of $\GL_n$ such that $P_{\pi} \subset Q \subset P$. Write $Q \cap M_P=Q_1 \times \hdots \times Q_m$. We have a decomposition $\fa_Q^{P,*}=\oplus_{i=1}^m \fa_{Q_i}^{\GL_{n_i},*}$. Set
\begin{equation}
\label{eq:nu_exp_defi}
    \nu_{Q,\pi}=\left(-\rho_{Q_i}^{\GL_{n_i}}/r_i \right)_{1 \leq i \leq r}
\end{equation}
written accordingly. If the context is clear, we omit the subscript $\pi$. If $Q=P_\pi$, we will write $\nu_\pi$. 

By construction, if $P_\pi \subset Q \subset P$, for any $\varphi \in \cA_{P,\pi}(\GL_n)$ we have $\varphi_{Q,-\nu_Q} \in \cA_Q^0(\GL_n)$ (see Lemma~\ref{lem:contant_term}). Moreover, note that $\varphi_Q=0$ unless $P_\pi \subset Q$. If $w \in {}_Q W_P$ such that $P_\pi \subset P_w$, then we set $\varphi_w:=\varphi_{P_w}$ and $\nu_w:=\nu_{P_w}$.

Set
\begin{equation}
\label{eq:residual_L_pi}
     L_{\pi,\mathrm{res}}(\lambda)=\prod_{\alpha \in \Delta_{P_\pi}^P} \left( \langle \lambda, \alpha^\vee \rangle -1 \right).
\end{equation}
For every $g \in \GL_n(\bA)$ and $\phi \in \cA_{P_\pi,\sigma_\pi}(\GL_n)$ denote by $E^{P,*}(g,\phi,\cdot)$ the partial residues of Eisenstein series $\lambda \mapsto  L_{\pi,\mathrm{res}}(\lambda)E^P(g,\phi,\lambda)$. It is holomorphic in a neighborhood of $-\nu_\pi$. By exactness of induction, for every $\varphi \in \cA_{P,\pi}(\GL_n)$ there exists $\phi \in \cA_{P_\pi,\sigma_\pi}(\GL_n)$ such that $\varphi=E^{P,*}(\phi,-\nu_\pi)$. We write $\pi=\boxtimes_{i=1}^m \Speh(\sigma_i,d_i)$.

Finally, let $w^*_\pi$ be the longest element in $W^P(P_\pi,P_\pi)$, i.e. the unique element $w$ in this set such that $w(P_\pi \cap M_P)w^{-1}$ is opposed to $P_\pi \cap M_P$. Note that it acts by identity on $\fa_{P}^*$ and that $w_\pi^* \sigma_\pi=\sigma_\pi$ and $w_\pi^* \nu_{\pi}=-\nu_{\pi}$.

\subsection{Normalization of intertwining operators}
\label{sec:norm_operator}

In this section, let $P$ be a standard parabolic subgroup of $\GL_n$ and let $\pi \in \Pi_{\disc}(M_P)$. Let $Q$ be a standard parabolic subgroup of $\GL_n$. Let $w \in W(P,Q)$ and take $\lambda \in \fa_{P,\cc}^*$ in general position. Denote by $M_\pi(w,\lambda)$ the restriction of $M(w,\lambda)$ (defined in \eqref{eq:global_operator}) to the subspace $\cA_{P,\pi}(\GL_n) \subset \cA_P(\GL_n)$. Following \cite{Art82}, we normalize $M_\pi(w,\lambda)$ as
\begin{equation}
    \label{eq:normalization_non_twisted} M_\pi(w,\lambda)=n_\pi(w,\lambda)N_\pi(w,\lambda).
\end{equation}
Here $n_\pi(w,\lambda)$ is a meromorphic function in $\lambda$ referred to as "the scalar factor", and $N_\pi(w,\lambda)$ is the so-called "normalized operator". If the context is clear, we will remove the subscript $\pi$. We describe these objects below.

\subsubsection{The scalar factor} 
\label{subsubsec:scalar_factor}
Write $M_P=\GL_{n_1} \times \hdots \times \GL_{n_m}$ and $\pi=\pi_1 \boxtimes \hdots \boxtimes \pi_m$. Let $\beta \in \Sigma_{P}$ be the positive root of $P$ associated to the two blocks $\GL_{n_i}$ and $\GL_{n_j}$, with $1 \leq i < j \leq m$. Set
\begin{equation}
    \label{eq:n_pi_formula}
    n_\pi(\beta,s)=\frac{L(1-s,\pi_i^\vee \times \pi_j)}{L(1+s,\pi_i \times \pi_j^\vee)}, \quad \text{and} \quad  n_\pi(w,\lambda)=\prod_{\substack{\beta \in \Sigma_{P} \\ w \beta <0}} n_\pi(\beta,\langle \lambda, \beta^\vee \rangle).
\end{equation}
Here $L$ is the completed Rankin--Selberg $L$ function from \cite{JPSS83}.

Let $\tau$ and $\tau'$ be discrete automorphic representations of some $\GL_k$ and $\GL_{k'}$. Write $\tau=\Speh(\sigma,d)$ and $\tau'=\Speh(\sigma',d')$. Then we have the formula 
\begin{equation}
\label{eq:discrete_L}
    L(s,\tau \times \tau')=\prod_{i=1}^d \prod_{j=1}^{d'} L\left(s+\frac{d-2i+1}{2}+\frac{d'-2j+1}{2},\sigma \times \sigma'\right).
\end{equation}
Moreover, the $L$-function $L(s,\sigma \times \sigma')$ is meromorphic and non-zero if $\Re(s) > 1$. It is holomorphic unless $\sigma^\vee \simeq \sigma'$ in which case it has a simple pole at $s=0$ and $s=1$.

\subsubsection{Local normalized intertwining operators} 
\label{subsubsec:local_normalized_operator}

Let $\phi \in \cA_{P,\pi}(\GL_n)$. Assume that $\phi=\otimes'_v \phi_v$ is factorizable, so that for all place $v$ we have $\phi_v \in I_P^{\GL_n} \pi_v$. Let $\tS \subset V_F$ be a finite set of places such that $\phi_v$ is unramified if $v \notin \tS$. By \cite[Theorem~2.1]{Art89}, we have a factorization
\begin{equation}
\label{eq:unram_facto_local_intertwin}
    N_\pi(w,\lambda)\phi=\prod_{v \in \tS} N_{\pi_v}(w,\lambda)\phi_v
\end{equation}
where the $N_{\pi_v}(w,\lambda)$ are meromorphic local intertwining operators $I_{P}^{\GL_n} \pi_{v,\lambda} \to I_{Q}^{\GL_n} w\pi_{v,\lambda}$. We will define more closely these operators in Section~\ref{sec:ggp_conj_non_tempered}. For now, we only recall their main properties which follow from \cite[Theorem~2.1]{Art89} and \cite[Proposition~I.10]{MW89}.

\begin{theorem}
    \label{thm:N}
    Let $\pi_v$ be a smooth irreducible and unitary representation of $M_P(F_v)$.
    \begin{enumerate}
        \item For each $\phi_v \in I_{P}^{\GL_n} \pi_{v}$ the vector $N_{\pi_v}(w,\lambda)\phi_v$ is a rational function in the variables $\langle \lambda, \alpha \rangle$ (resp. $q_v^{-\langle \lambda, \alpha \rangle}$) for the $\alpha \in \Delta_P$ such that $w \alpha <0$ in the Archimedean case (resp. the non-Archimedean case). It is holomorphic in the region
         \begin{equation*}
           \bigcap_{\substack{\alpha \in \Sigma_{P} \\ w \alpha <0}} \left\{ \lambda \in \fa_{P,\cc}^* \; | \; \langle \Re(\lambda),\alpha^\vee \rangle \geq 0 \right\}.
        \end{equation*}
        \item If $w_1 \in W(P,Q)$ and $w_2 \in W(Q,R)$ we have
        \begin{equation*}
            N_{w_1 \pi_v}(w_2,w_1\lambda)N_{\pi_v}(w_1,\lambda)=N_{\pi_v}(w_2w_1,\lambda).
        \end{equation*}
    \end{enumerate}
\end{theorem}

\subsection{Constant terms of discrete automorphic forms}

Throughout all this section, we fix a standard parabolic subgroup $P$ of $\GL_n$ and $\pi \in \Pi_\disc(M_P)$. We have $P_\pi$ and $\sigma_\pi \in \Pi_{\cusp}(M_{P_\pi})$ as in \S\ref{subsubsec:residual_blocks}. 

\subsubsection{Regularized intertwining operator}
\label{subsubsec:regularized_int}
For any $\phi \in \cA_{P_\pi,\sigma_\pi}(\GL_n)$ consider the regularized operator 
\begin{equation}
\label{eq:reg_operator}
    M^*(w_\pi^*,\lambda) \phi:= L_{\pi,\mathrm{res}}(\lambda)M(w_\pi^*,\lambda) \phi.
\end{equation}
Note that for every $\alpha \in \Sigma_{P_\pi}$ such that $w_\pi^* \alpha <0$ we actually have $\alpha \in \Sigma_{P_\pi}^P$. In particular, for any $\lambda \in \fa_{P,\cc}^*$ we have for such $\alpha$ the lower bound $\langle \Re(-\nu_\pi+\lambda),\alpha^\vee \rangle=\langle \Re(-\nu_\pi),\alpha^\vee \rangle >0$. It follows from the properties of global $L$-functions recalled in \S\ref{subsubsec:scalar_factor} that there exists a constant $c_\pi$ such that for any $\lambda \in \fa_{P,\cc}^*$ we have
\begin{equation}
\label{eq:regular_is_local}
    M^*(w_\pi^*,-\nu_\pi+\lambda)\phi=c_\pi N_{\sigma_\pi}(w_\pi^*,-\nu_\pi+\lambda)\phi.
\end{equation}
By Theorem~\ref{thm:N}, we conclude that $M^*(w_\pi^*,-\nu_\pi+\lambda)\phi$ is regular on $\fa_{P,\cc}^*$.

\subsubsection{Computation of the constant terms}
We use the notation of \S\ref{subsubsec:residual_blocks}. Let $P_\pi \subset Q \subset P$. For each $1 \leq i \leq m$ there exist integers $d_{i,1}, \hdots, d_{i,m_i}$ such that $\sum_{j=1}^{m_i} d_{i,j}=d_i$ and $M_Q=\prod_{i=1}^m \prod_{j=1}^{m_i} \GL_{r_i d_{i,j}}$. For each pair $(i,j)$, set $n_{i,j}=r_i d_{i,j}$ and $P_{i,j}=P_{\pi_i} \cap \GL_{n_{i,j}}$. Let $\pi_{i,j}$ be the discrete representation of $\GL_{n_{i,j}}$ spanned by the residues of Eisenstein series built from $\phi_{i,j} \in \cA_{P_{i,j},\sigma_i^{\boxtimes d_{i,j}}}(\GL_{n_{i,j}})$. Set $\pi_Q=\boxtimes_i \boxtimes_j \pi_{i,j}$ which is a discrete representation of $M_Q$. 

\begin{lem}
\label{lem:contant_term}
    For every $\varphi \in \cA_{P,\pi}(\GL_n)$ we have $\varphi_{Q,-\nu_Q} \in \cA_{Q,\pi_Q}(\GL_n)$. In particular, $\varphi_{Q,-\nu_Q}$ is residual unless $Q=P_\pi$. Moreover, in that case, assume that $\varphi=E^{P,*}(\phi,-\nu_\pi)$ for $\phi \in \cA_{P_\pi,\sigma_\pi}(\GL_n)$. Then we have
    \begin{equation}
    \label{eq:constant_P_pi}
        \varphi_{P_\pi}=M^*(w_\pi^*,-\nu_{\pi})\phi.
    \end{equation}
\end{lem}

\begin{proof}
    This property only depends on the restriction of $\varphi_{Q,-\nu_Q}$ to $M_Q$, so that it is enough to prove it if $P=\GL_n$. Then the special case \eqref{eq:constant_P_pi} is proved in \cite[Proposition~2.3]{Jac}, and the general case follows by a similar argument. We refer the reader to \cite[Lemma~9.4.2.1]{BoiPhD} where we carefully write the details.
\end{proof}

\begin{cor}
\label{cor:same_quotient}
    We have
    \begin{equation*}
        \Ker E^{P,*}(\cdot,-\nu_\pi)=\Ker M^*(w_\pi^*,-\nu_\pi) = \Ker N_{\sigma_\pi}(w_\pi^*,-\nu_\pi),
    \end{equation*}
    where all three maps are defined on $\cA_{P_\pi,\sigma_\pi,-\nu_\pi}(\GL_n)$.
\end{cor}

\begin{proof}
    For the first equality, the inclusion $\subset$ comes for \eqref{eq:constant_P_pi}. For the converse, let $\phi \in \cA_{P_\pi,\sigma_\pi,-\nu_\pi}(\GL_n)$. By Lemma~\ref{lem:contant_term}, the only cuspidal component of $E^{P,*}(\phi,-\nu_\pi)$ that is not automatically zero is $E^{P,*}_{P_\pi}(\phi,-\nu_\pi)=M^*(w_\pi^*,-\nu_\pi)\phi$. Therefore, if we assume $\phi \in \Ker M^*(w_\pi^*,-\nu_\pi)$ then all the cuspidal components of $E^{P,*}(\phi,-\nu_\pi)$ are zero, and thus $E^{P,*}(\phi,-\nu_\pi)$ itself is zero by Lemma~\ref{lem:vanishing_cusp_comp}. The last equality is a direct consequence of \eqref{eq:regular_is_local}. 
\end{proof}

\section{Ichino--Yamana--Zydor regularized periods}
\label{sec:IYZ_periods}
In this section, $G$ is $\GL_n \times \GL_{n+1}$. We embed $\GL_n$ in $\GL_{n+1}$ by
\begin{equation}
\label{eq:n_embedding}
    g \mapsto  \begin{pmatrix} g & \\ & 1 \end{pmatrix}.
\end{equation}
Let $H \simeq \GL_n$ be the diagonal subgroup in $G$. The goal of this section is to study the regularized period $\cP$ from \cite{IY}, and to connect it to the construction of \cite{IY}.

\subsection{Rankin--Selberg parabolic subgroups}
\label{sec:RS_subgroup}

We start by defining a subset $\cF_{\RS}$ of the set of semi-standard parabolic subgroups of $G$ which appears in the definition of the regularized period $\cP$.

\subsubsection{Preliminary notation}

Let $P_0 \subset G$ be the product of the subgroups of upper triangular matrices in $\GL_n$ and $\GL_{n+1}$, and $T_0 \subset G$ to be the product of the standard diagonal tori. If $J$ is a subgroup of $G$, we write $J=J_n \times J_{n+1}$. Similarly we have $\fa_0=\fa_{0,n} \oplus \fa_{0,n+1}$. 

All the constructions done relatively to $H$ will be decorated by a subscript ${}_H$. In particular, a subgroup of $H$ will typically be denoted $J_H$, and if $J$ is a subgroup of $G$ we set $J_H=J \cap H$. The pair $(T_{0,H},P_{0,H})$ is standard in $H$. Set $\fa_{0,H}=\fa_{T_{0,H}}$. It embeds into $\fa_0$. We will often identify the group $\GL_n$ either with $H$, either with its two copies in $G$ from the left and right coordinates, using the embedding \eqref{eq:n_embedding} for the latter.

\subsubsection{Definition of the set of Rankin--Selberg parabolic subgroups}
\label{subsubsec:RS_defi}
We define $\cF_{\RS}$ the set of \emph{Rankin--Selberg parabolic subgroups} of $G$. This is the set of semi-standard parabolic subgroups $P=P_n \times P_{n+1}$ of $G$ such that $P_n$ is standard and, with respect to the embedding \eqref{eq:n_embedding}, $P_n=P_{n+1} \cap \GL_n$. In particular, if $P \in \cF_{\RS}$ then $P_H$ is a standard parabolic subgroup of $H$ isomorphic to $P_n$. Conversely, if $P_H$ is a standard parabolic subgroup of $H$, set
\begin{equation*}
    \cP_{\RS}(P_H)=\{ Q \in \cF_{\RS} \; | \; Q_H=P_H \}.
\end{equation*}
Then we have
\begin{equation}
\label{eq:F_decompo}
    \cF_{\RS}= \bigsqcup_{P_H \subset H \; \mathrm{standard}} \cP_{\RS}(P_H).
\end{equation}

Recall that in \S\ref{subsubsec:blocks} we have associated to any standard parabolic subgroup $P$ of $\GL_n$ or $\GL_{n+1}$ a tuple $\underline{n}(P)$. Let $P_H$ be a standard parabolic subgroup of $H$. Write $\underline{n}(P_H)=(n_1,\hdots,n_m)$ with $\sum n_i=n$. Let $\cP_n^{n+1}(P_H)$ be the set of couples $(P^{\std}_{n+1},i_0)$ where $P^{\std}_{n+1}$ is a standard parabolic subgroup of $\GL_{n+1}$ and $i_0$ is an integer such that one of the two following alternatives is satisfied:
\begin{enumerate}
    \item $\underline{n}(P^{\std}_{n+1})=(n_1, \hdots, n_{i_0-1}, n_{i_0}+1,n_{i_0+1}, \hdots ,n_m)$ and $1 \leq i_0 \leq m$;
    \item $\underline{n}(P^{\std}_{n+1})=(n_1, \hdots, n_{i_0-1}, 1,n_{i_0}, \hdots ,n_m)$ and $1 \leq i_0 \leq m+1$.
\end{enumerate}
In the first case, set $N=\sum_{i=1}^{i_0} n_i$, and in the second set $N=\sum_{i=1}^{i_0-1} n_i$. Let $w(P^{\std}_{n+1},i_0) \in \fS_{n+1}$ be the cycle 
\begin{equation}
\label{eq:w_defi}
    w(P^{\std}_{n+1},i_0)=( N+2 \quad \hdots \quad n \quad n+1 \quad N+1).
\end{equation}
We identify it with an element in $W_{n+1}$ the Weyl group of $\GL_{n+1}$.

\begin{lem}
\label{lem:messy_iso}
    There is a bijection 
    \begin{equation}
    \label{eq:messy_iso}
        (P^{\std}_{n+1},i_0) \in \cP_n^{n+1}(P_H) \mapsto P_H \times (w(P^{\std}_{n+1},i_0) .P^{\std}_{n+1})  \in \cP_{\RS}(P_H).
    \end{equation}
\end{lem}

\begin{proof}
    Let $(P^{\std}_{n+1},i_0) \in \cP_n^{n+1}(P_H)$. Then $w_{P^{\std}_{n+1},i_0} .P^{\std}_{n+1} \cap \GL_n$ is a semi-standard parabolic subgroup of $\GL_n$ which has same semi-standard Levi factor as $P_H$. It is standard because $w(P^{\std}_{n+1},i_0)$ preserve the order of the blocks. It is therefore equal to $P_H$. Thus \eqref{eq:messy_iso} is well-defined. 

    We build its inverse. Let $P_n \times P_{n+1} \in \cP_{\RS}(P_H)$. We have $P_n=P_H$. There are two cases for $P_{n+1}$.
    \begin{enumerate}
            \item Either $M_{P_{n+1}}=\prod_{i=1}^m \GL_{n'_i}$ for some $n'_i$ which are all equal to $n_i$ except a $n'_{i_0}$ which is $n_{i_0}+1$. We define $\underline{n}(P^{\std}_{n+1})$ as in case $1$ and we have $w(P^{\std}_{n+1},i_0)^{-1} .P_{n+1}=P^{\std}_{n+1}$.
        \item Either $M_{P_{n+1}}=\prod_{i=1}^m \GL_{n_i} \times \GL_1$. In that case, there is a unique $w$ which acts by blocks on $M_{P_{n+1}}$ and such that $w^{-1}.P_{n+1} =P^{\std}_{n+1}$ is standard. This determines $\underline{n}(P^{\std}_{n+1})$ and $i_0$, and we have $w(P^{\std}_{n+1},i_0)=w$.
    \end{enumerate}
    The two constructions are inverse of each other and this concludes the proof.
\end{proof}

\begin{cor}
\label{cor:param_RS}
    We have a bijection
    \begin{equation*}
        \cF_{\RS}\simeq \left\{ (P_{n+1}^{\std},i_0) \; \middle| \; \begin{array}{l}
            P^{\std}_{n+1} \subset \GL_{n+1} \; \text{standard}, \\ M_{P^{\std}_{n+1}}=\prod_{i=1}^m \GL_{n_i}, \\
             1 \leq i_0 \leq m.
        \end{array} \right\}=\bigsqcup_{P_H\subset H \; \mathrm{standard}} \cP_n^{n+1}(P_H).
    \end{equation*}
\end{cor}

\begin{proof}
    By Lemma~\ref{lem:messy_iso}, it is enough to prove the last equality. But this follows by construction of the couples $(P^{\std}_{n+1},i_0)$.
\end{proof}

\subsubsection{Standardization}

Let $P \in \cF_{\RS}$. Let $(P_{n+1}^{\std},i_P)$ be the inverse image of $P$ under the isomorphism of Corollary~\ref{cor:param_RS}. Let $w_P^\std=w(P^{\std}_{n+1},i_P)$ be the element defined in \eqref{eq:w_defi}. We have
\begin{equation}
\label{eq:standardized_P}
    P_{n+1}=w_P^\std.P^{\std}_{n+1}.
\end{equation}
Because $P^{\std}_{n+1}$ is standard, we identify  $\fa_{P^{\std}_{n+1}}$ with $\rr^m$ for some $m \geq 1$ by using the same basis as in \S\ref{subsubsec:blocks}. By composing with $w_{P}^\std$, we obtain $\fa_{P_{n+1}} \simeq \rr^{m}$. In this choice of coordinates we have 
\begin{equation*}
    \fa_{P_{n+1}}^+=\left\{ (H_1, \hdots, H_{m}) \in \rr^{m} \; | \; H_1/n_1 > \hdots > H_{m}/n_m \right\}.
\end{equation*}
In order to keep the notation consistent, we will also identify $\fa_{P_n}$ with a subspace of $\rr^m$. It is $\rr^m$ if $\underline{n}(P^{\std}_{n+1})$ is of type 1, and it is the subspace $H_{i_P}=0$ if it is of type 2. 

Set $P^{\std}=P_n \times P^{\std}_{n+1}$, which is a standard parabolic subgroup of $G$. This is the \emph{standardization} of $P$. Write $(n_i)=\underline{n}(P^{\std}_{n+1})$. We may decompose the standard Levi factor $M_P^{\std}$ of $P^\std$ as
\begin{equation}
\label{eq:standard_decompo}
    M_{P}^{\std}=M_{P,n}^{\std}\times M_{P,n+1}^{\std}=\left( \bfM_{P,+} \times \cM_{P,n} \times \bfM_{P,-} \right) \times \left( \bfM_{P,+} \times \cM^{\std}_{P,n+1} \times \bfM^{\std}_{P,-} \right),
\end{equation}
where
\begin{equation}
\label{eq:ss_decompo}
   \bfM_{P,+} =   
            \prod_{i < i_P} \GL_{n_i}, \quad \bfM_{P,-}\simeq \bfM^{\std}_{P,-} =   
            \prod_{i > i_P} \GL_{n_i}, \quad \cM_{P,n} =  
            \GL_{n_{i_P}-1}, \quad \cM^{\std}_{P,n+1} =   
            \GL_{n_{i_P}}.
\end{equation}
We add a ${}^{\std}$ on $\bfM^{\std}_{P,-}$ to emphasize that, although they are isomorphic, the groups $\bfM_{P,-}$ and $\bfM^{\std}_{P,-}$ are not identified by the embedding \eqref{eq:n_embedding}. By composing with $w_P^\std$, we get a decomposition
\begin{equation}
\label{eq:M_P_levi}
     M_{P} \simeq \left( \bfM_{P,+} \times \cM_{P,n} \times \bfM_{P,-} \right) \times \left( \bfM_{P,+} \times \cM_{P,n+1} \times \bfM_{P,-} \right).
\end{equation}
The two copies of $\bfM_{P,-}$ in \eqref{eq:M_P_levi} are now identified by the embedding \eqref{eq:n_embedding}. The groups $\cM^{\std}_{P,n+1}$ and $\cM_{P,n+1}$ are isomorphic but in general embedded in two different ways in $\GL_{n+1}$. Set
\begin{equation*}
    \bfM_P=\bfM_{P,+} \times \bfM_{P,-}, \quad \bfM_P^2=\bfM_P \times \bfM_P, \quad \cM_P=\cM_{P,n} \times \cM_{P,n+1}.
\end{equation*}
The restriction of the diagonal embedding $H \subset G$ gives $\cM_{P,H} \subset \cM_P$ and we have $\cM_{P,H} \simeq \cM_{P,n}$. Note that $M_{P_H}=M_{P,H}=\bfM_P \times \cM_{P,H}$. The group $\cM_P$ is isomorphic to $\GL_{n_{i_P}-1} \times \GL_{n_{i_P}}$, and we can define embedding of $\GL_{n_{i_P}-1}$ in $\cM_P$ as in \eqref{eq:n_embedding}. This is compatible with the inclusion $H \subset G$ in the sense that $\cM_{P,H}=\cM_P \cap \GL_{n_{i_P}-1}$. In particular, we have $\cF^{\cM_P}_{\RS}$ the set of Rankin--Selberg parabolic subgroups of $\cM_P$. We describe the relative version of the set $\cF_{\RS}$. 

\begin{lem}
\label{lem:relative_bijection}
    Let $Q \in \cF_{\RS}$. We have a bijection
    \begin{equation}
    \label{eq:relative_RS}
        P \in \left\{ P \in \cF_{\RS} \; | \; P \subset Q \right\} \mapsto (P \cap \bfM_Q,P \cap \cM_Q) \in \{ \text{standard parabolic subgroups of } \bfM_Q \} \times \cF^{\cM_Q}_{\RS}.
    \end{equation}
    Moreover, for any $P \in \cF_{\RS}$ with $P \subset Q$ we have
    \begin{equation*}
        \cM_{P} \subset \cM_Q, \quad w_P^\std=w_{\cP}^{\cM_Q,\std} w_Q^\std,
    \end{equation*}
    where $w_{\cP}^{\cM_Q,\std}$ is $w_{P \cap \cM_Q}^\std$ seen as an element in $G$ via the embedding $\cM_Q \subset G$.
\end{lem}

\begin{proof}
    This follows from the classification of Corollary~\ref{cor:param_RS}, noting that if $P \in \cF_{\RS}$ with $P \subset Q$, then $P_H \subset Q_H$ and therefore $\underline{n}(P_H)$ is a subdivision of $\underline{n}(Q_H)$.
\end{proof}

\subsubsection{Connection to \cite{Zydor}} \label{subsubsec:RS_coord}

We now connect the Rankin--Selberg parabolic subgroups to the constructions of \cite[Sections~3.1 and 3.2]{Zydor}. For any semi-standard parabolic subgroup $P$ of $G$, set 
\begin{equation*}
    \fz_P=\fa_P \cap \fa_{0,H}, \quad \text{ and } \quad \fz_P^+=\fa_P^+ \cap \fa_{0,H}.
\end{equation*}
If $P \in \cF_{\RS}$, write $M_{P,n+1}^{\std}=\prod_{i=1}^{m} \GL_{n_i}$. We have the descriptions
\begin{equation}
\label{eq:zp_descri}
    \fz_P=\left\{ H=(H_n,H_{n+1}) \in \fa_P \; \Big{|} \; 
    \begin{array}{l}
        H_{n,i}=H_{n+1,i} \text{ if } i \neq i_P, \\ 
        H_{n,i_P}=H_{n+1,i_P}=0 
    \end{array} 
    \right\}=\fa_{\bfM_P} =\oplus_{i \neq i_P} \fa_{\GL_{n_i}},
\end{equation}
and 
\begin{equation}
\label{eq:zp+_descri}
    \fz_P^+=\fa_P^+ \cap \fa_{P_H}^+=\left\{ H \in \fz_P \; | \; H_1/n_1 > \hdots > H_{i_P-1}/{n_{i_P-1}} > 0 > H_{i_P+1}/{n_{i_P+1}} > \hdots > H_{m}/n_m \right\}
\end{equation}
where we write $H_i$ for $H_{n,i}=H_{n+1,i}$.

\begin{lem}
    Let $P_H$ be a standard parabolic subgroup of $H$. We have
    \begin{equation}
    \label{eq:PG_defi}
        \cP_{\RS}(P_H)=\left\{ Q \subset G \; \text{semi-standard} \; | \; \fa_{Q}^+ \cap \fa_{P_H}^+ \neq \emptyset \right\}.
    \end{equation}
\end{lem}

\begin{rem}
    The set in the RHS of \eqref{eq:PG_defi} is defined in \cite[23]{Zydor} by asking that $Q$ the split center of the centralizer of $T_{0,H}$ in $G$ instead of $T_0$. As $Z_{\GL_{n+1}} T_{0,n}=T_{0,n+1}$, this is $T_0$.
\end{rem}

\begin{proof}
    Let $Q$ belong to the RHS of \eqref{eq:PG_defi}. By \cite[Proposition~3.1]{Zydor}, we have $Q_H=P_H$, $M_{Q,H}=M_{P,H}$ and $N_{Q,H}=N_{P,H}$. As subgroups of $\GL_n$, this implies that $P_H \subset Q_n$, $M_{P,H} \subset M_{Q,n}$ and $N_{P,H} \subset N_{Q,n}$, so that $P_H=Q_n$. By playing the same game in the $\GL_{n+1}$-coordinate, we get $P_H=Q_{n+1} \cap \GL_n$ and therefore $Q \in \cP_{\RS}(P_H)$. Conversely, we have described $\fa_P^+ \cap \fa_{P_H}^+$ in \eqref{eq:zp+_descri} in the case $P \in \cP_{\RS}(P_H)$ and it is non-empty.
\end{proof}

Let $P \in \cF_{\RS}$. We have a basis $(e_{n+1,i}^*)_i$ of $\fa_{P_{n+1}}^*$ (see \S\ref{subsubsec:blocks}). If $\lambda \in \fa_P^*$, denote by $\lambda_H$ its restriction to $\fz_P$. Set $e_i^*=( e_{n+1,i}^*)_H$. This is zero if $i = i_P$. Then the family $(e_i^*)_{i \neq i_P}$ is a basis of $\fz_P^*$ the dual space of $\fz_P$. Let $(e_i)_i$ be the basis of $\fz_P$ dual to $(e_i^*)_{i \neq i_P}$ (with the convention $e_{i_P}=0$).

Denote by $\Delta_{P_{n+1}}$ the set $w_P^\std \Delta_{P^{\std}_{n+1}} \subset \fa_{P_{n+1}}^*$, and by $\Delta_{P,H}$ its restriction to $\fz_P$. Note that $\Delta_{P,H}$ is a basis of $\fz_P^*$. Let $\hat{\Delta}_{P,H}^\vee$ be the basis of $\fz_P$ dual to $\Delta_{P,H}$. In coordinates, we have
\begin{equation}
\label{eq:root_defi}
    \Delta_{P,H}=\left\{ \frac{e_{i}^*}{n_i}-\frac{e_{i+1}^*}{n_{i+1}} \; \middle| \; 1 \leq i \leq m-1 \right\},
\end{equation}
where we recall that $e_{i_P}^*=0$, and 
\begin{equation}
\label{eq:coweight_defi}
    \hat{\Delta}_{P,H}^\vee=\left\{ \sum_{j=1}^i n_j e_j \; \middle| \; i < i_P \right\} \bigcup \left\{ -\sum_{j=i}^{m} n_j e_j \; \middle| \; i > i_P \right\}.
\end{equation}
Then 
\begin{equation*}
    \fz_P^+ = \bigoplus_{\varpi^\vee \in \hat{\Delta}_{P,H}^\vee} \rr_{>0} \varpi^\vee.
\end{equation*}
Recall that in \S\ref{subsubsec:blocks} we have given $\fa_P$ an inner product $(\cdot,\cdot)$. By restricting, we obtain an inner product on $\fz_P$. Let $\Delta_{P,H}^\vee$ be the basis of $\fz_P$ dual to $\hat{\Delta}_{P,H}^\vee$. Then we have
\begin{equation}
\label{eq:coroot_defi}
    \Delta_{P,H}^\vee=\left\{ e_{i}-e_{i+1} \; \middle| \; 1 \leq i \leq m-1 \right\},
\end{equation}
where again $e_{i_P}=0$. The inconsistency in the normalizations between \eqref{eq:root_defi} and \eqref{eq:coroot_defi} is due to the fact that $(e_i,e_i)=n_i^{-1}$. Set
\begin{equation*}
    \fz_P^{*,+}=\left\{ \lambda \in \fa_P^* \; \middle | \; \langle \lambda, \alpha^\vee_H \rangle >0, \; \forall \alpha^\vee_H \in \Delta_{P,H}^\vee \right\}.
\end{equation*}
Note that $\fz_P^+ \subset \bigoplus_{\alpha^\vee_H \in \Delta_{P,H}^\vee} \rr_{\geq 0} \alpha^\vee_H$, so that if $\lambda \in \fz_P^{*,+}$ we have $\langle \lambda,H \rangle > 0$ for every $H \in \fz_P^+$.

\subsubsection{Relative constructions} \label{subsubsec:relative_case}
Let $P \subset Q \in \cF_{\RS}$. Let $\fz_P^Q$ be the orthogonal of $\fz_Q$ in $\fz_P$. Note that this is consistent with our previous notation as $\fz_G=\{0\}$. If $T \in \fz_{P}$, let $T_Q$ be its projection on $\fz_Q$ and $T_P^Q$ be its projection on $\fz_{P}^Q$, both with respect to $\fz_P=\fz_{P}^Q \oplus \fz_Q$. This applies in particular to $T \in \fa_{0,H}=\fz_{P_0}$ and we simply write $T^Q$ for $T_{P_0}^Q$.

Let $\Delta_{P_{n+1}}^{Q_{n+1}}$ be the set $w_P^\std \Delta_{P^{\std}_{n+1}}^{Q^{\std}_{n+1}} \subset \fa_{P_{n+1}}^*$ and let $\Delta_{P,H}^Q$ be its restriction to $\fz_P^Q$. Then $\Delta_{P,H}^Q$ is a basis $\fz_P^{Q,*}$. Let $\hat{\Delta}_{P,H}^{Q,\vee}$ be the basis of $\fz_P^Q$ dual to $\Delta_{P,H}^Q$, and let $\Delta_{P,H}^{Q,\vee}$ be the basis of $\fz_P^Q$ dual to $\hat{\Delta}_{P,H}^{Q,\vee}$ under the restriction of the inner product. These sets are easy to describe. Indeed, using Lemma~\ref{lem:relative_bijection} we may write
\begin{equation*}
    \fz_P^Q=\fa_{\bfP}^{\bfM_Q} \oplus^\perp \fz_{\cP}^{\cM_Q},
\end{equation*}
where we set $\bfP=P \cap \bfM_Q$ (seen as a subgroup of $H)$ and $\cP=P \cap \cM_Q$. It follows that we have decompositions $\Delta_{P,H}^Q=\Delta_{\bfP,H}^{\bfM_Q} \sqcup \Delta_{\cP,H}^{\cM_Q}$, $\hat{\Delta}_{P,H}^{Q,\vee}=\hat{\Delta}_{\bfP,H}^{\bfM_Q,\vee} \sqcup \hat{\Delta}_{\cP,H}^{\cM_Q,\vee}$ and $\Delta_{P,H}^{Q,\vee}=\Delta_{\bfP,H}^{\bfM_Q,\vee} \sqcup \Delta_{\cP,H}^{\cM_Q,\vee}$ where $\Delta_{\bfP,H}^{\bfM_Q}$, $\hat{\Delta}_{\bfP,H}^{\bfM_Q,\vee}$ and $\Delta_{\bfP,H}^{\bfM_Q,\vee} $ can be respectively identified with the usual sets of roots, coweights and coroots of the standard parabolic subgroup $\bfP$ of $\bfM_Q$, and $\Delta_{\cP,H}^{\cM_Q}$, $\hat{\Delta}_{\cP,H}^{\cM_Q,\vee}$ and $\Delta_{\cP,H}^{\cM_Q,\vee}$ are described by \eqref{eq:root_defi}, \eqref{eq:coweight_defi} and \eqref{eq:coroot_defi} respectively. 

\subsubsection{The element $\underline{\rho}_P$} \label{subsubsec:rho_defi} Let $P \in \cF_{\RS}$. Set
\begin{equation*}
    \underline{\rho}_P=(\rho_P-2 \rho_{P_H})_{| \fz_{P}} \in \fz_P^*.
\end{equation*}
In coordinates, we have $(\underline{\rho}_P)_i=\frac{1}{2}$ if $i<i_P$, and $(\underline{\rho}_P)_i=-\frac{1}{2}$ if $i>i_P$.

Note that we have a surjection $\fa_{P^\std}^* \to \fz_P^*$ by composing with $w_P^{\std,-1}$ and then restricting. We now choose a lift $\underline{\rho}_{P}^\std$ of $\underline{\rho}_P$. Under the decomposition of $M_{P}^\std$ given in \eqref{eq:standard_decompo} and with the coordinates of \S\ref{subsubsec:blocks}, we set
\begin{equation}
\label{eq:rho_P'}
    \underline{\rho}_{P}^\std = \left( (1/4,\hdots,1/4,0,-1/4,\hdots,-1/4),(1/4,\hdots,1/4,0,-1/4,\hdots,-1/4) \right),
\end{equation}
where the zeros correspond to the blocks $\cM_{P,n}$ and $\cM_{P,n+1}^\std$. Note that $\underline{\rho}_{P}^\std \in \overline{\fa_{P^\std}^{*,+}}$.

\subsubsection{Volumes of fundamental domains} \label{subsubsec:Euclidean}
Let $P \in \cF_{\RS}$. Thanks to the isomorphism $\fz_P \simeq \oplus_{i \neq i_P} \fa_{\GL_{n_i}}$ from \eqref{eq:zp_descri}, we give $\fz_P$ the product of the Haar measures described in \S\ref{subsubsec:first_measure}. If $P \subset Q$, we give $\fz_P^Q=\fz_P / \fz_Q$ the quotient measure. If $B$ is a basis of a lattice $\zz(B)$ in $\fz_P^Q$, let $\vol(\fz_P^Q/ \zz(B))$ be its covolume. If $Q=G$, we have 
\begin{equation}
\label{eq:basis_volume}
    \vol(\fz_P/\zz((e_i)))=1, \quad \text{and} \quad \vol(\fz_P/\zz(\hat{\Delta}_{P,H}^\vee))=\prod_{i \neq i_P} n_i.
\end{equation}

\begin{rem}
    In \cite[Section~4.3]{Zydor}, $\fz_P$ is given the measure associated to the Euclidean structure induced by $(\cdot,\cdot)$. One can check that it differs from our measure by a factor of $\prod_{i \neq i_P} n_i$. This normalization cancels out with the choice of measure we make in \S\ref{subsubsec:global_measure}. 
\end{rem}

\subsection[Exponential polynomials]{Exponential polynomials and Fourier transforms on cones}
\label{sec:pol_exp_Fourier}

Let $V$ be a finite dimensional vector space, let $\Omega \subset V$ be an open subset. By an exponential polynomial on $\Omega$, we mean a function $f : \Omega \to \cc$ which belongs to the vector space generated by the maps $v \mapsto P(v) \exp(\langle \lambda, v \rangle)$ where $P \in \cc[V]$ and $\lambda \in V^*$. Any exponential polynomial can be uniquely written as $f(v)=\sum_{\lambda \in V^*} P_\lambda(v) \exp(\langle \lambda,v \rangle)$ for a unique map $\lambda \mapsto P_\lambda$ with finite support. The polynomial $P_0$ is called the polynomial part of $f$. 

Let $P \subset Q \in \cF_{\RS}$. Set
\begin{equation*}
    \varepsilon_{P}^Q=(-1)^{\dim \fz_P^Q}=(-1)^{\dim \fa_{P_{n+1}}^{Q_{n+1}}}.
\end{equation*}
For $\lambda \in \fa_0^*$, set
\begin{equation*}
    \hat{\theta}_P^Q(\lambda)=\vol(\fz_P^Q/ \zz(\hat{\Delta}_{P,H}^{Q,\vee}))^{-1} \prod_{\varpi^\vee \in \hat{\Delta}_{P,H}^{Q,\vee}} \langle \lambda, \varpi^\vee \rangle,
\end{equation*}
and
\begin{equation*}
    \theta_P^Q(\lambda)=\vol(\fz_P^Q/ \zz(\Delta_{P,H}^{Q,\vee}))^{-1} \prod_{\alpha^\vee_H \in \Delta_{P,H}^{Q,\vee}} \langle \lambda, \alpha^\vee_H \rangle.
\end{equation*}
These polynomials are not identically zero on $\fa_{P,\cc}^*$. In the special case where $P=P_0$, we omit the subscript $P$. We set 
\begin{equation}
    \label{eq:reg_space}
    \fa_{0,\cc}^{*,Q-\mathrm{reg}}=\bigcap_{\substack{ P \in \cF_{\RS} \\ P \subset Q}}\left\{ \lambda \in \fa_{0,\cc}^* \; \middle| \; \hat{\theta}_P^Q(\lambda+\underline{\rho}_P) \neq 0 \right\}.
\end{equation}
If $R$ is a semi-standard parabolic subgroup of $G$, we set $\fa_{R,\cc}^{*,Q-\mathrm{reg}}=\fa_{R,\cc}^* \cap \fa_{0,\cc}^{*,Q-\mathrm{reg}}$.

\begin{rem}
    As noted in \cite[Section~4.5]{Zydor}, it is enough to take the intersection in \eqref{eq:reg_space} along maximal proper Rankin--Selberg parabolic subgroups $P$ of $Q$. In that case, the condition amounts to asking that $\lambda+\underline{\rho}_P$ is non-zero when restricted to $\fz_P^Q$.
\end{rem}

We define as in \cite[7]{Zydor}
\begin{align*}
   & A(\overline{\fz_Q^+},\overline{\fz_P^+})=\left\{ r(H-Z) \; | \; H \in \overline{\fz_P^+}, \; r>0 \right\}, \\
    & A(\overline{\fz_Q^+},\overline{\fz_P^+})^{\vee}=\{ H \in \fz_P \; | \; (H,H') \geq 0, \; H' \in A(\overline{\fz_Q^+},\overline{\fz_P^+}) \},
\end{align*}
where $Z$ is any $Z \in \fz_Q^+$. We have
\begin{align}
    &A(\overline{\fz_Q^+},\overline{\fz_P^+})=\left\{ H \in \fz_P \; | \; \forall \alpha \in \Delta_{P,H}^Q, \; \langle \alpha,H \rangle \geq 0 \right\}=\bigoplus_{\varpi^\vee \in \hat{\Delta}_{P,H}^{Q,\vee}} \rr_{\geq 0} \varpi^\vee \bigoplus \fz_Q, \label{eq:A_defi} \\
    &A(\overline{\fz_Q^+},\overline{\fz_P^+})^{\vee}=\bigoplus_{\alpha_H^\vee \in \Delta_{P,H}^{Q,\vee}} \rr_{\geq 0} \alpha^\vee_H. \label{eq:A_vee_defi}
\end{align}
Let $\tau_P^Q$ be the characteristic function of $A(\overline{\fz_Q^+},\overline{\fz_P^+})$, and let $\hat{\tau}_P^Q$ be the characteristic function of $A(\overline{\fz_Q^+},\overline{\fz_P^+})^\vee$. These two functions are defined on $\fa_{0,H}$. Set
\begin{equation*}
    \Gamma_P^Q(H,T)=\sum_{\substack{ R \in \cF_{\RS} \\ P \subset R \subset Q}} \varepsilon_R^Q \tau_P^R(H)\hat{\tau}_R^Q(H-T), \quad H,T \in \fa_{0,H}.
\end{equation*}
We now compute the Fourier transforms of these functions.

\begin{lem}
\label{lem:Fourier_transform} 
\begin{enumerate}
    \item For every polynomial $q \in \cc[\fz_P]$, the Fourier transform 
    \begin{equation}
        \label{eq:q_FT}
        \mathrm{FT}_P^Q(q,\lambda):=\int_{\fz_P^Q} \tau_P^Q(H) q(H)\exp(\langle \lambda, H \rangle) dh,
    \end{equation}
    is absolutely convergent in the open subset $\cap_{\varpi^\vee \in \hat{\Delta}_{P,H}^{Q,\vee}} \{\langle \Re(\lambda), \varpi^\vee \rangle < 0\}$ and extends to a meromorphic function on $\{ \lambda \; | \; \hat{\theta}_P^Q(\lambda) \neq 0\}$. 
    
    \item For every $T \in \fa_{0,H}$ the map $H \mapsto \Gamma_P^Q(H,T)$ is compactly supported with support contained in the open ball $B(T/2,\norm{T}/2) \subset \fa_{0,H}$. Let $q \in \cc[\fz_P]$. For fixed $\lambda \in \fa_{P,\cc}^*$ the Fourier transform
    \begin{equation*}
        \mathrm{FT}_P^Q(\Gamma,T,q,\lambda):=\int_{\fz_P^Q} \Gamma_P^Q(H,T)q(H) \exp(\langle \lambda, H \rangle) dH
    \end{equation*}
    is an exponential polynomial in $T$, and for $\lambda \in \fa_{P,\cc}^{Q,*}$ such that $\hat{\theta}_P^Q(\lambda) \neq 0$ its purely polynomial term is constant equal to $\mathrm{FT}_P^Q(q,\lambda)$.
    \item If $q=1$, we have
    \begin{equation}
    \label{eq:tau_FT}
        \mathrm{FT}_P^Q(1,\lambda)=\varepsilon_P^Q \hat{\theta}_P^Q(\lambda)^{-1}.
    \end{equation}
    and for $\lambda \in \fa_{P,\cc}^{Q,*}$ in general position
     \begin{equation}
    \label{eq:Gamma_FT}
        \mathrm{FT}_P^Q(\Gamma,T,1,\lambda)= \sum_{\substack{ R \in \cF_{\RS} \\ P \subset R \subset Q}} \varepsilon_P^R \hat{\theta}_P^R(\lambda)^{-1} \theta_R^Q(\lambda)^{-1} \exp(\langle \lambda, T_R^Q \rangle).
    \end{equation}
\end{enumerate}
    
\end{lem}

\begin{proof}
    The first two points are contained in \cite[Section~1.4]{Zydor} and \cite[Corollary~2.6]{ZZ}. For the last, note that the set of $\lambda \in \fa_{P,\cc}^*$ such that, for all $P \subset R \subset Q$, all $\varpi^\vee \in \hat{\Delta}_{P,H}^{R,\vee}$, $\langle \Re(\lambda), \varpi^\vee \rangle <0 $ and all $\alpha^\vee_H \in \Delta_{R,H}^{Q,\vee}$, $\langle \Re(\lambda), \alpha^\vee_H \rangle <0$, is non-empty. Indeed, it contains the $\lambda$'s such that the restriction of $\Re(\lambda)$ to $\fz_P$ belongs to $-\fz_P^{*,+}$. For such $\lambda$, we see using \eqref{eq:A_defi} and \eqref{eq:A_vee_defi} that \eqref{eq:tau_FT} and \eqref{eq:Gamma_FT} hold. We conclude that they hold for $\lambda$ in general position by analytic continuation.
\end{proof}

\subsection{Regularized Rankin--Selberg periods à la Zydor}
\label{sec:reg_RS_periods_Z}

In this section, we fix $Q \in \cF_{\RS}$. We write $M^{\std}_{Q,n+1}=\prod_{i=1}^m \GL_{n_i}$. Let $T_{H,\mathrm{reg}} \in \fa_{0,H}$ be the element defined in \cite[Lemma~2.7]{Zydor}. We say that $T \in \fa_{0,H}$ is sufficiently positive if $T \in T_{H,\mathrm{reg}} + \fa_{0,H}^+$.  We now define the truncated period $\cP^T$ and its regularization $\cP$.

\subsubsection{Iwasawa decomposition and measures} \label{subsubsec:global_measure}
Set
\begin{equation*}
    M_{Q_H}(\bA)^{Q,1}=\left\{ m \in M_{Q_H}(\bA) \; \middle| \; \left(H_{Q_H}(m)\right)_Q=0 \right\}, \quad Z_Q^\infty=A_Q^\infty \cap A_{Q_H}^\infty.
\end{equation*}
The restriction of $H_{Q_H}$ to $Z_Q^\infty$ is an isomorphism with image $\fz_Q$. This gives $Z_Q^\infty$ a Haar measure. We have a direct product decomposition of commuting groups $M_{Q_H}(\bA)=Z_Q^{\infty} M_{Q_H}(\bA)^{Q,1}$. By the Iwasawa decomposition, there is a unique Haar measure on $M_{Q_H}(\bA)^{Q,1}$ such that for all $f \in C_c(H(\bA))$ we have
\begin{equation}
\label{eq:facto_measure}
    \int_{H(\bA)} f(h)dh=\int_{K_H} \int_{N_{Q_H}(\bA)} \int_{M_{Q_H}(\bA)^{Q,1}} \int_{Z_Q^\infty} \exp(\langle -2 \rho_{Q_H},H_{Q_H}(am) \rangle) f(nmak) da dm dn dk.
\end{equation}
We set
\begin{equation*}
    [M_{Q_H}]^{Q,1}=M_{Q_H}(F) \backslash M_{Q_H}(\bA)^{Q,1}
\end{equation*}
which is given the quotient by the counting measure. Note that we have
\begin{equation*}
    [M_{Q,H}]^{Q,1} \simeq \left(\prod_{i \neq i_Q} [\GL_{n_i}]^1\right) \times [\GL_{n_{i_Q}-1}]=[\bfM_{Q}]^1 \times [\cM_{Q,H}].
\end{equation*}
Then the measure $dh$ coincides the product of the ones on each $[\GL_{n_i}]^1$ and on $[\GL_{n_{i_Q}-1}]$ described in \S\ref{subsubsec:first_measure}. Moreover, if $Q=G$ we have $M_{G,H}(\bA)^{G,1}=H(\bA)$ and $Z_G^\infty=\{1\}$. 

\subsubsection{Truncation operators} Let $T \in \fa_{0,H}$. For $\phi$ a locally integrable function on $Q(F) \backslash G(\bA)$, we define a relative truncation operator as
\begin{equation}
\label{eq:truncation_operator}
    \Lambda^{T,Q} \phi(h)=\sum_{\substack{P \in \cF_{\RS} \\ P \subset Q}} \varepsilon_P^Q \sum_{\delta \in P_H(F) \backslash Q_H(F)} \hat{\tau}_P^Q(H_{0,H}(\delta h)^Q_P -T^Q_P) \phi_P(\delta h), \quad h \in Q_H(F) \backslash H(\bA),
\end{equation}
where the sum over $\delta$ is actually over a finite set which only depends on $h$ and $T$ by \cite[Lemma~2.8]{Zydor}. If $Q=G$ we simply write $\Lambda^T$. The main theorem on this operator is the following. 

\begin{theorem}
\label{thm:truncation_operator}
    Let $T \in \fa_{0,H}$ be sufficiently positive. Let $J$ be a level of $G$. For any $N,N' >0 $ there exists a finite family $(X_i)_{i \in I}$ of elements in $\cU(\fg_\infty)$ such that for any smooth and right $J$-invariant function $\phi$ on $[G]_Q$, the function $\Lambda^{T,Q} \phi$ is a function on $[H]_{Q_H}$ and we have
    \begin{equation*}
        \sup_{m \in M_{Q,H}(\bA)^{Q,1}} \norm{m}^N_{M_{Q,H}} \Val{\Lambda^{T,Q} \phi(mk)} \leq \sum_{i \in I} \norm{\phi}_{-N',X_i}.
    \end{equation*}
\end{theorem}

\begin{proof}
    This is \cite[Theorem~3.9]{Zydor}. Note that the statement in ibid. is weaker, but our version can easily be extracted from the proof.
\end{proof}

The operator $\Lambda^{T,Q}$ also satisfies two unfolding properties.

\begin{lem}
\label{lem:unfolding_lambda}
    Let $\phi$ be a locally integrable function on $Q(F) \backslash G(\bA)$. For all $T,T' \in \fa_{0,H}$ and all $h \in H(\bA)$ we have
    \begin{equation}
    \label{eq:inversion_formula}
        \phi(h)=\sum_{\substack{P \in \cF_{\RS} \\ P \subset Q}} \sum_{\delta \in P_H(F) \backslash Q_H(F)} \tau_P^Q(H_{0,H}(\delta h)_P^Q-T_P^Q) \Lambda^{T,P} \phi(\delta h),
    \end{equation}
    and 
     \begin{equation}
        \Lambda^{T+T',Q} \phi(h) =\sum_{\substack{P \in \cF_{\RS} \\ P \subset Q}} \sum_{\delta \in P_H(F) \backslash Q_H(F)} \Gamma_P^Q(H_{0,H}(\delta h)_P^Q-T_P^Q,(T')_P^Q) \Lambda^{T,P} \phi(\delta h).
    \end{equation}
\end{lem}

\begin{proof}
    These formulae are \cite[Lemma~3.7]{Zydor} and \cite[Lemma~3.8]{Zydor}. Note that they are only stated there for $Q=G$. However, the proof ultimately relies on combinatorics of cones and therefore goes through in the general case.
\end{proof}

\subsubsection{Truncated periods}

Let $\phi \in \cT([G])$. We have $\phi_Q \in \cT([G]_Q)$. For $T$ sufficiently positive, we define the truncated period of $\phi$ relative to $Q$ to be
\begin{equation*}
    \cP^{T,Q} (\phi)=\int_{K_H} \int_{[M_{Q,H}]^{Q,1}} \exp(-\langle 2 \rho_{P_H}, H_{P_H}(m) \rangle) \Lambda^{T,Q} \phi(mk) dmdk.
\end{equation*}
This integral is absolutely convergent by Theorem~\ref{thm:truncation_operator} and \cite[Proposition~A.1.1]{Beu}. Note that if $Q=G$ it reduces to  
\begin{equation*}
    \cP^{T} (\phi)= \int_{[H]} \Lambda^{T} \phi(h) dh.
\end{equation*}

\subsubsection{Regularized periods}
Let $\varphi \in \cA_Q(G)$ be an automorphic form and let $P \in \cF_{\RS}$ which is contained in $Q$. By \S\ref{subsubsec:exponents}, we have a decomposition of $\varphi_P$ as sum of normalized automorphic forms in $\cA_P^0(G)$ multiplied by exponents. By reorganizing, we see that we also have a decomposition 
(where all terms are non-zero) such that
    \begin{equation}
    \label{eq:normalized_decompo_Z}
        \varphi_P(ah)=\sum_i q_{P,i}(H_{P_H}(a))  \exp(\langle \lambda_{P,i}, H_{P_H}(a) \rangle)\varphi_{P,i}(h), \quad a \in Z_P^\infty, \quad h \in N_{P_H}(\bA) M_{P,H}(\bA)^{P,1}K_H,
    \end{equation}
    with $q_{P,i} \in \cc[\fz_P]$, $\lambda_{P,i} \in \fz_P^*$, and $\varphi_{P,i} \in \cA_P(G)$ that are left $Z_P^\infty$-invariant. Let $\cE_{P,H}(\varphi)=\{\lambda_{P,i}\}$. Note that this is the restriction to $\fz_P$ of the elements in the set of exponents $\cE_P(\varphi)$ from \S\ref{subsubsec:exponents}. Let $\cA_Q(G)^{\mathrm{reg}}$ be the subspace of $\varphi \in \cA_Q(G)$ such that for all $P \in \cF_{\RS}$ with $P \subset Q$ we have
\begin{equation}
\label{eq:regular_conditions}
        \cE_{P}(\varphi) \subset \fa_{P,\cc}^{*,Q-\mathrm{reg}}.
\end{equation}
The subspace $\cA_Q(G)^{\mathrm{reg}}$ is stable under right-translations by $G(\bA)$.

\begin{theorem}
\label{thm:reg}
    For $\varphi \in \cA_Q(G)^{\mathrm{reg}}$ there exists a unique exponential polynomial on $\fa_{0,H}$ that coincides with $T \mapsto \cP^{T,Q}(\varphi)$ for $T$ sufficiently positive. Its purely polynomial part is constant and equal to
     \begin{equation}
    \label{eq:general_period}
        \cP^Q(\varphi)=\sum_{\substack{P \in \cF_{\RS} \\ P \subset Q}} \sum_i  \exp( \langle \lambda_{P,i}+\underline{\rho}_P,T_P^Q \rangle) \mathrm{FT}_P^Q((q_{P,i})_{T},\lambda_{P,i}+\underline{\rho}_P)\cP^{T,P}(\varphi_{P,i}),
    \end{equation}
    where $(q_{P,i})_{T}(H)=q_{P,i}(H+T_P^Q)$. Moreover, $\cP^G$ is $H(\bA)$-invariant.
\end{theorem}

\begin{proof}
    The $Q=G$ case is \cite[Theorem~4.1]{Zydor}. For the general case, using Lemma~\ref{lem:Fourier_transform}, the Iwasawa decomposition \eqref{eq:facto_measure} and Lemma~\ref{lem:unfolding_lambda} we have for any for $T$ sufficiently positive and $T' \in \overline{\fa_{0,H}^+}$ the expression
    \begin{equation}
    \label{eq:exact_dvlp}
        \cP^{T+T',Q}(\varphi)=\sum_{\substack{P \in \cF_{\RS} \\ P \subset Q}} \sum_i\exp(\langle \lambda_{P,i}+\underline{\rho}_P,T_P^Q \rangle) \mathrm{FT}_P^Q(\Gamma,T',(q_{P,i})_{T},\lambda_{P,i}+\underline{\rho}_P)\cP^{T,P}(\varphi_{P,i}).
    \end{equation}
    We now conclude by Lemma~\ref{lem:Fourier_transform}.
\end{proof}

The number $\cP^Q(\varphi)$ is the \emph{regularized period à la Zydor} along $Z_Q^\infty M_{Q_H}(F) N_{Q_H}(\bA) \backslash H(\bA)$.

\subsubsection{Parabolic descent} We describe $\Lambda^{T,Q}$ by parabolic descent. Set $P_{0,\bfQ}:=P_0 \cap \bfM_Q$, and $\fa_{0,\bfQ}:=\fa_{P_{0,\bfQ}}$. For every standard parabolic subgroup $\bfP$ of $\bfM_Q$ (with respect to $P_{0,\bfQ}$), let $\hat{\tau}^{\bfM_Q}_\bfP$ be the characteristic function of 
\begin{equation*}
    \left\{ H \in \fa_{0,\bfQ} \; \middle| \; \langle \varpi,H \rangle \geq 0, \; \forall \varpi \in \hat{\Delta}_{\bfP} \right\}.
\end{equation*}
Let $\bfT \in \fa_{0,\bfQ}^{\bfM_Q}$. For $\phi$ a locally integrable function on $[\bfM_Q^2]$, define the diagonal Arthur truncation operator to be 
\begin{equation}
\label{eq:Arthur_diagonal}
    \Lambda^{T,\bfM_Q^2}\phi(m)=\sum_{\bfP \in \cF^{\bfM_Q}(P_0)} (-1)^{\dim \fa_{\bfP}^{\bfM_Q}} \sum_{\delta \in \bfP(F) \backslash \bfM_Q(F)} \hat{\tau}^{\bfM_Q}_{\bfP}(H_{0}(\delta m) -\bfT) \phi_{\bfP^2}(\delta m,\delta m), \; m \in [\bfM_Q],
\end{equation}
where $\bfP^2$ is the standard parabolic subgroup of $\bfM_Q^2$ equal to $\bfP \times \bfP$. Note that $\Lambda^{T,\bfM_Q^2}$ is the truncation operator built in \cite[Section~3.7]{Zydor} relatively to the diagonal embedding $\bfM_Q \subset \bfM_Q^2$. In particular it takes smooth functions of uniform moderate growth on $[\bfM_Q^2]$ to functions of rapid decay on $[\bfM_Q]^1$ as in Theorem~\ref{thm:truncation_operator}. We set 
\begin{equation*}
    \cP^{\bfT,\bfM_Q^2}(\phi)=\int_{[\bfM_Q]^1} \Lambda^{\bfT,\bfM_Q^2} \phi(m)dm,
\end{equation*}
The operator $\Lambda^{T,\bfM_Q^2}$ differs from the usual Arthur truncation operator of \cite{Art80}. Indeed, in the formalism of \cite{Zydor}, the latter (say for some $\GL_k$) is associated to $\GL_k \subset \GL_k$ and not $\GL_k \subset \GL_k^2$. 

By \cite[Theorem~4.1]{Zydor}, $\cP^{\bfT,\bfM_Q^2}$ gives rise to a regularized period $\cP^{\bfM_Q^2}$ when restricted to $\cA(\bfM_Q^2)^{\mathrm{reg}}$ the subset of $\varphi \in\cA(\bfM_Q^2)$ such that for all standard parabolic subgroup $\bfP$ of $\bfM_Q$
\begin{equation}
\label{eq:regular_diagonal}
    \forall \lambda \in \cE_{\bfP^2}(\varphi), \; \forall \varpi^\vee \in \hat{\Delta}_\bfP^{\vee}, \; \langle \lambda, \varpi^\vee \rangle \neq 0.
\end{equation}
If $\phi \in \cA(\bfM_Q^2)^{\mathrm{reg}}$, then the regularized period $\cP^{\bfM_Q^2}(\phi)$ satisfies a relation of the form
\begin{equation}
\label{eq:general_period_partial}
     \cP^{\bfM_Q^2}(\phi)=\sum_{\bfP \subset \bfM_Q \; \mathrm{standard}} \sum_i  \exp( \langle \lambda_{\bfP,i},\bfT_{\bfP} \rangle) \mathrm{FT}_{\bfP}^{\bfM_Q^2}((q_{\bfP,i})_{\bfT},\lambda_{\bfP,i})\cP^{\bfT,\bfP}(\phi_{\bfP,i}),
\end{equation}
where the $\lambda_{\bfP,i}$, $q_{\bfP,i}$ and $\phi_{\bfP,i}$ are respectively the exponents, polynomials and automorphic forms appearing the normalized decomposition of $\phi_{\bfP^2}$, and $\mathrm{FT}_{\bfP}^{\bfM_Q^2}$ is a Fourier transform defined as in Lemma~\ref{lem:Fourier_transform} (see \cite[Section~4.4]{Zydor}).

Set $P_{0,\cQ}:=P_0 \cap \cM_{Q,H}$, and $\fa_{0,\cQ}:=\fa_{P_{0,\cQ}}$. Let $\cT \in \fa_{0,\cQ}$. For $\phi$ a locally integrable function on $[\cM_Q]$, we define $\Lambda^{\cT,\cM_Q} \phi$ to be the operator $\Lambda^{\cT}$ of \eqref{eq:truncation_operator} built with respect to $\cM_{Q,H} \subset \cM_Q$. We then have a truncated period $\cP^{\cT,\cM_Q}(\phi)$ and a regularized period $\cP^{\cM_Q}(\phi)$, the latter being well defined if $\phi \in \cA(\cM_Q)^{\mathrm{reg}}$ where this set is described as in \eqref{eq:regular_conditions}. 

The behavior of the regularized period $\cP^Q$ is summarized in the following lemma.

\begin{lem}
\label{lem:parabolic_descent}
    Let $\varphi \in \cA_Q(G)^{\mathrm{reg}}$. Set
    \begin{equation*}
        \varphi_{M_Q}=\left(R(e_{K_H})\varphi\right)_{|M_Q(\bA),-\rho_Q}.
    \end{equation*}
    Then we have
    \begin{equation}
    \label{eq:both_orders}
        \cP^Q(\varphi)=\cP^{\bfM_Q^2} \left(m \in [\bfM^2_Q] \mapsto \cP^{\cM_Q}(R(m)\varphi_{M_Q}) \right)=\cP^{\cM_Q} \left(m \in [\cM_Q] \mapsto \cP^{\bfM_Q^2}(R(m)\varphi_{M_Q}) \right).
    \end{equation}
\end{lem}

\begin{proof}
    For any $T \in \fa_{0,H}$ we have a unique decomposition $T^Q=\bfT+\cT$ where $\bfT \in \fa_{0,\bfQ}^{\bfM_Q}$ and $\cT \in \fa_{0,\cQ}$. By Lemma~\ref{lem:relative_bijection}, it is easily checked that we have for every $\mathbf{m} \in [\bfM_Q]$ and $m \in [\cM_{Q,H}]$
    \begin{align*}
        \Lambda^{T,Q} \varphi(\mathbf{m}m)&=\Lambda^{\bfT,\bfM_Q^2}\left( m' \in [\bfM_Q^2] \mapsto \Lambda^{\cT,\cM_Q} (R(m') \varphi) (m) \right)(\mathbf{m}) \\
        &=\Lambda^{\cT,\cM_Q}\left( m' \in [\cM_Q] \mapsto \Lambda^{\bfT,\bfM_Q} (R(m') \varphi) (\mathbf{m}) \right)(m).
    \end{align*}
    If $\phi$ is a locally integrable function on $[\bfM_Q^2]$, we know that the sum defining $\Lambda^{\bfT,\bfM_Q^2} \phi(m)$ is over a finite set depending only on $\bfT$ and $m$ and not on $\phi$ (\cite[Lemma~2.8]{Zydor}). The same is true with the other truncation operator $\Lambda^{\cT,\cM_Q}$. We may therefore switch $\Lambda^{\bfT,\bfM_Q^2}$ with the integrals over $[\cM_{Q,H}]$ and $K_H$ (or $\Lambda^{\cT,\cM_Q}$ with those on $[\bfM_Q]^1$ and $K_H$) and see that $\cP^{T,Q}(\varphi)$ is
    \begin{equation*}
        \cP^{\bfT,\bfM_Q^2} \left(m \in [\bfM^2_Q] \mapsto \cP^{\cT,\cM_Q}(R(m)\varphi_{M_Q}) \right)=\cP^{\cT,\cM_Q} \left(m \in [\cM_Q] \mapsto \cP^{\bfT,\bfM_Q^2}(R(m)\varphi_{M_Q}) \right)
    \end{equation*}
    It remains to say that the regularized periods can be expressed in terms of the truncated ones (Theorem~\ref{thm:reg} and \eqref{eq:general_period_partial}) to conclude.
\end{proof}

\begin{rem}
\label{rem:diagonal_Arthur}
     Note that if $\bfM_Q^2= \prod_{i=1}^m \GL_{n_i}^2$ and if $\phi \in \cA(\bfM_Q^2)^{\mathrm{reg}}$, we may further decompose
\begin{equation}
\label{eq:gln_product}
    \cP^{\bfM_Q^2}(\phi)=\bigotimes_{i=1}^m \cP^{\GL_{n_i}^2}(\phi),
\end{equation}
where $\cP^{\GL_{n_i}^2}$ is the regularized diagonal Arthur period along $\GL_{n_i}$. Moreover, by \cite[Theorem~4.6]{Zydor} we know that, when applied to a discrete automorphic form of some $\GL_k^2$, $\cP^{\GL_k^2}$ computes the Petersson (bilinear) inner-product $\langle \cdot,\bar{\cdot} \rangle_{\GL_k,\Pet}$.
\end{rem}

\subsubsection{Regularized periods of Eisenstein series} 
Let $P$ be a standard parabolic subgroup of $G$, let $\pi \in \Pi_\disc(M_P)$. Recall that such $\pi$ has trivial central character on $A_P^\infty$. Let $\varphi \in \cA_{P,\pi}(G)$. Let $w \in {}_{Q^{\std}} W_P$. For $\lambda \in \fa_{P_w,\cc}^*$ in general position, set
\begin{equation}
    \label{eq:truncated_period_defi}
    \cP^{T,Q}(\varphi,\lambda,w)=\cP^{T,Q}(E^{Q^{\std}}(w_Q^{\std,-1} \cdot,M(w,\lambda)\varphi_{P_w},w \lambda)),
\end{equation}

\begin{rem}
    The element $w_Q^\std$ was defined in \eqref{eq:standardized_P}. We prefer to make it appear in \eqref{eq:truncated_period_defi} to deal with the standard parabolic subgroup $Q^{\std}$ rather than $Q$.
\end{rem}

The truncated period in \eqref{eq:truncated_period_defi} is well defined for $\lambda$ in general position by \cite[Theorem~2.2]{Lap} (which ensures that generalized Eisenstein series are of uniform moderate growth), and if $P_\pi \not\subset P_w$ it is zero. Let us compute the exponents of the generalized Eisenstein series. Let $R \in \cF_{\RS}$. By \eqref{eq:constant_term} and Lemma~\ref{lem:relative_bijection} we have
    \begin{equation}
    \label{eq:partial_constant_term}
        \left(E^{Q^{\std}}(w_Q^{\std,-1} \cdot,M(w,\lambda)\varphi_{P_w},w \lambda)\right)_R=\sum_{\substack{w' \in {}_{R^{\std}} W_{Q^{\std}_w}^{Q^{\std}} w \\ P_\pi \subset P_{w'}}} E^{R^{\std}}(w_R^{\std,-1} \cdot, M(w',\lambda) \varphi_{P_{w'}},w' \lambda).
    \end{equation}
    Let $w' \in {}_{R^{\std}} W_{Q^{\std}_w}^{Q^{\std}} w$ such that $P_\pi \subset P_{w'}$. Denote by $(w'(\lambda+\nu_{w'}))_{R^{\std}}$ the projection of $w'(\lambda+\nu_{w'})$ on $\fa_{R^{\std},\cc}^*$. Then by Lemma~\ref{lem:contant_term} we know that
    \begin{equation}
    \label{eq:partial_normalization}
        \left(E^{R^{\std}}(M(w',\lambda) \varphi_{P_{w'}},w' \lambda)_{-(w'(\lambda+\nu_{w'}))_{R^{\std}}}\right)(w_R^{\std,-1} \cdot) \in \cA_R^0(G).
    \end{equation}
    Because $\fz_R \subset \fa_R$, it follows that the exponents of \eqref{eq:partial_constant_term} are the
    \begin{equation}
        \left(w_R^\std w'\left(\lambda+\nu_{w'} \right)\right)_{|\fz_R}.
    \end{equation}
    We will see in Lemma~\ref{lem:regular} that they are all contained in $\fa_{R,\cc}^{*,Q-\mathrm{reg}}$ for $\lambda \in \fa_{P,\cc}^*$ in general position. Therefore, for such $\lambda$ we set
    \begin{equation}
    \label{eq:PQE}
    \cP^Q(\varphi,\lambda,w)=\cP^Q(E^{Q^{\std}}(w_Q^{\std,-1} \cdot,M(w,\lambda)\varphi_{P_w},w \lambda)).
\end{equation}

\begin{lem}
\label{lem:regular}
    For every $R \subset Q$ and every $w' \in {}_{R^{\std}} W_{Q^{\std}_w}^{Q^{\std}}w $ such that $P_\pi \subset P_{w'}$ we have for $\lambda \in \fa_{P,\cc}^*$ in general position
    \begin{equation*}
        \hat{\theta}_R^{Q}( w_R^\std w'(\lambda+\nu_{w'})+\underline{\rho}_R)\theta_R^Q( w_R^\std w'(\lambda+\nu_{w'})+\underline{\rho}_R) \neq 0
    \end{equation*}
\end{lem}

\begin{proof}
    By Lemma~\ref{lem:W_CH} we have $w' \in {}_{R^{\std}} W_P$. Set $\bfR=R \cap \bfM_Q$ and $\cR=R \cap \cM_Q$. Write $P=P_{n} \times P_{n+1}$ and $\lambda=\lambda_n+\lambda_{n+1}$, $w=(w_n,w_{n+1})$ accordingly. Let $\varpi^\vee \in \hat{\Delta}_{R,H}^{Q,\vee}$. By \S\ref{subsubsec:relative_case} we have two cases.
    \begin{enumerate}
        \item $\varpi^\vee \in \hat{\Delta}_{\bfR,H}^{\bfM_Q,\vee}$: in that case we have $\langle \underline{\rho}_R,\varpi^\vee \rangle=0$ and there exist $\varpi_n^\vee \in \hat{\Delta}_{R_n}^{Q_n,\vee}$ and $\varpi_{n+1}^\vee \in \hat{\Delta}_{R^{\std}_{n+1}}^{Q^{\std}_{n+1},\vee}$ such that 
        \begin{equation*}
            \langle w' \lambda + w'\nu_{w'}, w_R^{\std,-1} \varpi^\vee \rangle=\langle w'_n \lambda_n + w'_n\nu_{w'_n}, \varpi_n^\vee \rangle+\langle w'_{n+1} \lambda_{n+1} + w'_{n+1}\nu_{w'_{n+1}}, \varpi^\vee_{n+1} \rangle.
        \end{equation*}
        By \cite[Lemma~3.1.5.1]{Ch} and \cite[Lemma~3.1.5.2]{Ch}, because $P_\pi \subset P_{w'}$ if this expression is constant then it is strictly negative.
        \item $\varpi^\vee \in \hat{\Delta}_{\cR,H}^{\cM_Q,\vee}$: in that case by \eqref{eq:coweight_defi} $\lambda \in \fa_{P,\cc}^* \mapsto \langle w' \lambda, w_R^{\std,-1} \varpi^\vee \rangle$ is clearly non-zero.
    \end{enumerate}
    This shows that $\hat{\theta}_R^{Q}( w_R^\std w'(\lambda+\nu_{w'})+\underline{\rho}_R) \neq 0$ for $\lambda$ in general position. The argument for $\theta_R^Q( w_R^\std w'(\lambda+\nu_{w'})+\underline{\rho}_R)$ is the same.
\end{proof}

\begin{prop}
\label{prop:unfold_periods}
    For $\lambda \in \fa_{P,\cc}^*$ in general position and $T$ sufficiently positive we have
    \begin{equation}
    \label{eq:P_unfolding}
        \cP^Q(\varphi,\lambda,w)=\sum_{\substack{ R \in \cF_{\RS} \\ R \subset Q}} \varepsilon_{R}^Q \sum_{w' \in {}_{R^{\std}} W_{Q^{\std}_w}^{Q^{\std}}w} \cP^{T,R}(\varphi,\lambda,w') \cdot \frac{\exp(\langle w_R^\std w'(\lambda+\nu_{w'})+\underline{\rho}_R,T_R^Q \rangle)}{\hat{\theta}_R^Q(w_R^\std w' (\lambda + \nu_{w'})+\underline{\rho}_R)},
    \end{equation}
    and
    \begin{equation}
    \label{eq:PT_unfolding}
         \cP^{T,Q}(\varphi,\lambda,w)=\sum_{\substack{ R \in \cF_{\RS} \\ R \subset Q}} \sum_{w' \in {}_{R^{\std}} W_{Q^{\std}_w}^{Q^{\std}}w} \cP^{R}(\varphi,\lambda,w') \cdot \frac{\exp(\langle w_R^\std w'(\lambda+\nu_{w'})+\underline{\rho}_R,T_R^Q \rangle)}{\theta_R^Q(w_R^\std w' (\lambda + \nu_{w'})+\underline{\rho}_R)},
    \end{equation}
    where in both cases we understand that the summands are zero unless $P_\pi \subset P_{w'}$ (whether or not the denominator is identically zero).
\end{prop}

\begin{proof}
    The first point follows from \eqref{eq:general_period} by the computation of the constant term in \eqref{eq:partial_constant_term} and of the exponents in \eqref{eq:partial_normalization}. Indeed, we know by Lemma~\ref{lem:regular} that for $\lambda \in \fa_{P,\cc}^*$ in general position
    \begin{equation*}
        E^{Q^{\std}}(w_Q^{\std,-1} \cdot,M(w,\lambda)\varphi_{P_w},w \lambda) \in \cA_Q(G)^{\mathrm{reg}},
    \end{equation*}
    and by Lemma~\ref{lem:Fourier_transform} we have $\mathrm{FT}_R^Q=\varepsilon_R^Q (\hat{\theta}_R^Q)^{-1}$. 

    We prove \eqref{eq:PT_unfolding}. To ease notation, we assume that $Q=G$ and $w=1$, the general case being similar. By the computation of the constant term in \eqref{eq:partial_constant_term}, we see by \eqref{eq:exact_dvlp} that for $T$ sufficiently positive, $T' \in \overline{\fa_{0,H}^+}$ and $\lambda \in \fa_{P,\cc}^*$ in general position, $ \cP^{T+T'}(\varphi,\lambda)$ is
    \begin{equation*}
       \sum_{R \in \cF_{\RS}} \sum_{w' \in {}_{R^{\std}} W_P} \exp(\langle w_R^\std w'(\lambda+\nu_{w'})+\underline{\rho}_R,T_R\rangle) \mathrm{FT}_R^G(\Gamma,T',1,w_R^\std w'(\lambda+\nu_{w'})+\underline{\rho}_R)\cP^{T,R}(\varphi,\lambda,w) 
    \end{equation*}
    By Lemma~\ref{lem:Fourier_transform} and Lemma~\ref{lem:regular}, for $\lambda \in \fa_{P,\cc}^*$ in general position we have
    \begin{equation*}
      \mathrm{FT}_R^G(\Gamma,T',1,w_R^\std w'(\lambda+\nu_{w'})+\underline{\rho}_R) =\sum_{Q \supset R} \varepsilon_R^Q \frac{\exp(\langle w_R^\std w'(\lambda+\nu_{w'})+\underline{\rho}_R,T'_Q\rangle) }{(\hat{\theta}_R^Q \theta_Q)(w_R^\std w'(\lambda+\nu_{w'})+\underline{\rho}_R)}.
    \end{equation*}
    Note that if $w' \in {}_{R^{\std}} W_{Q^{\std}_w}^{Q^{\std}}w$ for some $w \in {}_{Q^{\std}}W_P$, for any $H \in \fz_Q$ we have by Lemma~\ref{lem:relative_bijection}
    \begin{equation*}
        \langle w_R^\std w'(\lambda+\nu_{w'})+\underline{\rho}_R,H \rangle=\langle w_Q^\std w(\lambda+\nu_{w})+\underline{\rho}_Q, H \rangle.
    \end{equation*}
    By switching the sums over $Q$ and $R$ and using Lemma~\ref{lem:W_CH} we now see that 
    \begin{align*}
        \cP^{T+T'}(\varphi,\lambda)=&\sum_{Q \in \cF_{\RS}} \sum_{w \in {}_{Q^{\std}}W_P} \frac{\exp(\langle w_Q^\std(\lambda+\nu_{w})+\underline{\rho}_R,T_Q+T'_Q\rangle)}{\theta_Q(w_Q^\std w(\lambda+\nu_w)+\underline{\rho}_Q)} \times \\
        &\left( \sum_{R \subset Q} \varepsilon_R^Q \sum_{w' \in {}_{R^{\std}} W_{Q^{\std}_w}^{Q^{\std}}w} \cP^{T,R}(\varphi,\lambda,w') \cdot \frac{\exp(\langle w_R^\std w'(\lambda+\nu_{w'})+\underline{\rho}_R,T_R^Q \rangle)}{\hat{\theta}_R^Q(w_R^\std w' (\lambda + \nu_{w'})+\underline{\rho}_R)} \right).
    \end{align*}
    By the first part of the theorem, the last line is $\cP^{Q}(\varphi,\lambda,w')$. We conclude by taking $T'=0$.
\end{proof}

\begin{cor}
\label{cor:P_regular}
    The map $\lambda \in \fa_{P,\cc}^* \mapsto \cP^Q(\varphi,\lambda,w)$ is meromorphic. Moreover, let $f$ be a holomorphic function such that
    \begin{equation*}
        \lambda \mapsto M^*(w,\lambda) \varphi_{P_w}:=f(\lambda) M(w,\lambda) \varphi_{P_w}
    \end{equation*}
    is holomorphic in a neighborhood of $\lambda_0$. Then the map
    \begin{equation*}
        \lambda \mapsto \cP^{Q,*}(\varphi,\lambda,w):=f(\lambda) \cP^Q(\varphi,\lambda,w)
    \end{equation*}
    is meromorphic and we have in a neighborhood of $\lambda_0$
    \begin{equation}
    \label{eq:residue_PQ}
        \cP^{Q,*}(\varphi,\lambda,w)=\cP^Q(M^*(w,\lambda)\varphi_{P_w},\lambda,w).
    \end{equation}
\end{cor} 

\begin{proof}
    By Theorem~\ref{thm:truncation_operator} and the continuity of Eisenstein series from \cite[Theorem~2.2]{Lap}, for every $T$ sufficiently positive the map $\lambda \mapsto \cP^{T,Q}(\varphi,\lambda,w)$ is meromorphic. The first statement therefore follows from Lemma~\ref{lem:regular} and Proposition~\ref{prop:unfold_periods}. The other assertions are proved by the same argument. 
\end{proof}

\subsection{Regularized periods and Zeta functions}
\label{sec:reg_zeta}

In this section, we give a representation-theoretic interpretation of the regularized period $\cP^Q$. By Lemma~\ref{lem:parabolic_descent} and Remark~\ref{rem:diagonal_Arthur}, it is enough to understand Zydor's regularized period $\cP$. We show that $\cP$ it coincides with the Ichino--Yamana regularized period $\cP^{IY}$ built in \cite{IY}. It follows from \cite[Theorem~1.1]{IY} that it computes the Zeta integral of automorphic forms.

\subsubsection{Local Whittaker functionals}

Let $\psi$ be a non-trivial automorphic unitary character of $F \backslash \bA$. We define $\psi_0$ a generic automorphic character of $N_0(\bA)$ by the formula
\begin{equation*}
    \psi_0(u)=\psi \left(\sum_{i=1}^{n-1} -(u_n)_{i,i+1} \right)\psi \left(\sum_{i=1}^{n} (u_{n+1})_{i,i+1}\right), \quad u \in N_0(\bA).
\end{equation*}
This is a generic character of $N_0(\bA)$ which is trivial on $N_{0,H}(\bA)$. It decomposes as $\psi_0=\prod_v \psi_{0,v}$.

We now fix $v$ a place of $F$. Let $P$ be a standard parabolic subgroup of $G$, and let $\pi_v$ be a smooth irreducible representation of $M_P(F_v)$. We say that $\pi_v$ is generic if it admits a non-zero Whittaker model with respect to the restriction of $\psi_{0,v}$ to $M_P(F_v) \cap N_0(F_v)$ (in which case it is for every generic character). We denote by $\cW(\pi_v,\psi_{0,v})$ the Whittaker model of $\pi_v$, and take an isomorphism $W_{\pi_v}^{\psi_{0,v}} : \pi_v \to \cW(\pi,\psi_{0,v})$. 

Let $w$ be the element in the Weyl group of $G$ such that $w(\Delta_0^P) \subset \Delta_0$ and such that for every root $\alpha$ in $N_P$ we have $w \alpha <0$. Let $P'=M'N'$ be the standard parabolic subgroup of $G$ with standard Levi factor $wM_Pw^{-1}$. Let $w \psi_{0}$ be the generic character of $(M_P \cap N_0)(\bA)$ such that for $n' \in (M_{P'} \cap N_0)(\bA)$ we have $w \psi_{0}(\dot{w}^{-1} n' \dot{w})=\psi_{0}(n')$. Let $\lambda \in \fa_{P,\cc}^*$. For $\varphi_v \in I_{P}^G(\pi_{v,\lambda})$ with support in $P(F_v) \dot{w} P_0(F_v)$, we define the \emph{Jacquet functional} by
\begin{equation}
\label{eq:local_Jacquet}
    W^{\psi_{0,v}}_{P,\pi_v}(\varphi_v,\lambda)=\int_{N'(F_v)} W^{w\psi_{0,v}}_{\pi_v}(\varphi_{\lambda,v}(\dot{w}^{-1}n'_v)) \overline{\psi_{0,v}}(n'_v) dn'_v.
\end{equation}
Then $W^{\psi_{0,v}}_{P,\pi_v}(\cdot,\lambda)$ extends to a linear form on $I_{P}^G(\pi_{v,\lambda})$. We also write $W^{\psi_{0,v}}_{P,\pi_v}$ for the associated Whittaker function, i.e. 
\begin{equation*}
    W^{\psi_{0,v}}_{P,\pi_v}(g_v,\varphi_v,\lambda)=W^{\psi_{0,v}}_{P,\pi_v}(R(g_v)\varphi_v,\lambda).
\end{equation*}
Moreover, the map $\lambda \mapsto W^{\psi_{0,v}}_{P,\pi_{v}}(g_v,\varphi_v,\lambda)$ is holomorphic (see \cite{CS}).

\subsubsection{Local Zeta functions}

For $\lambda \in \fa_{P,\cc}^*$ with $\Re(\lambda)$ in a positive cone, we define the local Zeta function by 
 \begin{equation}
        \label{eq:local_zeta_int}
            Z_{\pi_v}(\varphi_v,\lambda)=\int_{N_{0,H}(F_v) \backslash H(F_v)}  W^{\psi_{0,v}}_{P,\pi_v}(h_v,\varphi_v,\lambda) dh_v, \quad \varphi_v \in I_{P}^G \pi_v.
    \end{equation}
It admits a meromorphic continuation to $\fa_{P,\cc}^*$.

Write $M_P=\prod_{i=1}^{m_n} \GL_{n_{n,i}} \times \prod_{i=1}^{m_{n+1}} \GL_{n_{n+1,i}}$ and $\pi_v=\boxtimes_{i=1}^{m_n} \pi_{v,n,i} \boxtimes_{i=1}^{m_{n+1}} \pi_{v,n+1,i}$ accordingly. For every place $v$, set
\begin{equation}
    \label{eq:b_v}
    b(\lambda,\pi_v)=\prod_{i<j}L(1+\lambda_{n,i}-\lambda_{n,j},\pi_{v,n,i} \times \pi_{v,n,j}^\vee)\prod_{i<j}L(1+\lambda_{n+1,i}-\lambda_{n+1,j},\pi_{v,n+1,i} \times \pi_{v,n+1,j}^\vee).
\end{equation}
Moreover, consider the Rankin--Selberg $L$-function
\begin{equation}
    \label{eq:local_RS}
    L\left(\lambda+\frac{1}{2},\pi_{v,n} \times \pi_{v,n+1}\right):=\prod_{i=1}^{m_n} \prod_{j=1}^{m_{n+1}} L(1/2+\lambda_{n,i}+\lambda_{n+1,j},\pi_{v,n,i} \times \pi_{v,n+1,j}).
\end{equation}
We now define the \emph{normalized local Zeta function} by
\begin{equation*}
    Z_{\pi_v}^\natural(\varphi_v,\lambda):=b(\lambda,\pi_v) \frac{Z_{\pi_v}(\varphi_v,\lambda)}{L\left(\lambda+1/2,\pi_{v,n} \times \pi_{v,n+1}\right)}.
\end{equation*}

\begin{lem}
    \label{lem:local_zeta_reg}
    For any $\varphi_v \in I_P^G \pi_v$, the quotient
    \begin{equation}
        \label{eq:quotient_zeta}
        \lambda \mapsto \frac{Z_{\pi_v}(\varphi_v,\lambda)}{L\left(\lambda+1/2,\pi_{v,n} \times \pi_{v,n+1}\right)}
    \end{equation}
    is holomorphic on $\fa_{P,\cc}^*$. If we assume moreover that $\pi_v$ is unitary, the regularized Zeta function $\lambda \mapsto Z_{\pi_v}^\natural(\varphi_v,\lambda)$ is a meromorphic function on $\fa_{P,\cc}^*$ which is regular in the region $\Re(\lambda) \in \overline{\fa_{P}^{*,+}}$.
\end{lem}

\begin{proof}
    Assume by contradiction that \eqref{eq:quotient_zeta} is not regular on $\fa_{P,\cc}^*$. Then there exists some $\mu$ such that there is a single polar divisor passing through $\mu$ and no zero divisor. Set $W_v=W_{P,\pi_v}^{\psi_{0,v}}(\cdot,\varphi_v,\mu)$ which belongs to the Whittaker model of $I_P^G \pi_{v,\mu}$. Assume first that $v$ is non-Archimedean. By \cite[Theorem~2.7]{JPSS83}, there exists a meromorphic function $L(s+1/2,I_P^G \pi_{v,\mu})$ such that the quotient
    \begin{equation*}
        s \mapsto \frac{Z_{\pi_v}(W_v,s)}{L(s+1/2,I_P^G \pi_{v,\mu})}
    \end{equation*}
    is entire, where $Z_{\pi_v}(W_v,s)$ is the Zeta function from \eqref{eq:local_zeta_int} associated to $W_v$. By \cite[Theorem~3.1]{JPSS83}, $L(s+1/2,I_P^G\pi_{v,\mu})$ is equal to the product of $ L\left(1/2+s+\lambda,\pi_{v,n} \times \pi_{v,n+1}\right)$ (where we identify $s$ with an element in $\fa_{\GL_n,\cc}^*$) with a polynomial in $q_v^s$ and $q_v^{-s}$. Therefore, \eqref{eq:quotient_zeta} is also entire. This is the desired contradiction as the singularity at $\mu$ cannot be compensated by a zero by assumption. In the Archimedean case, one argues similarly using \cite[Theorem~2.1]{Jac09} instead.
    
    Finally, because $\pi_v$ is unitary and generic, $b(\lambda,\pi_v)$ is regular in the region $\Re(\lambda) \in \overline{\fa_{P}^{*,+}}$. by \cite[Section~2]{BR}. This proves the last point.  
\end{proof}

\subsubsection{Global Zeta functions}
\label{subsubsec:global_zeta}

We now go back to the global setting. For any automorphic form $\Phi \in \cA(G)$, we may consider the global Whittaker function
\begin{equation}
\label{eq:global_whitt}
    W^{\psi_0}(g,\Phi)=\int_{[N_0]} \Phi(ng) \overline{\psi_0}(n) dn, \quad g \in [G].
\end{equation}
For $s \in \cc$, we consider the \emph{global Zeta integral}
\begin{equation*}
    Z(\Phi,s)=\int_{N_{0,H}(\bA) \backslash H(\bA)} W^{\psi_0}(h,\Phi) \Val{\det h}^s dh.
\end{equation*}
By \cite[Lemma~7.1.1.1]{BPCZ}, this integral is absolutely convergent for $\Re(s)$ large enough, and by \cite[Corollary~5.3]{IY}, it admits a meromorphic continuation to $\cc$. 

Let $P$ be a standard parabolic subgroup of $G$, let $\pi \in \Pi_{\disc}(M_P)$. If $\varphi \in \cA_{P,\pi}(G)$, we consider 
\begin{equation*}
    \lambda \in \fa_{P,\cc}^* \mapsto Z_\pi(\varphi,\lambda):=Z(E(\varphi,\lambda),0).
\end{equation*}
This is a meromorphic function in $\lambda$ by \cite[Corollary~5.4]{IY} and \cite[Equation~(4.2)]{IY}. Moreover, assume that $\varphi=\otimes'_v \varphi_v$ is factorizable and consider the product of global $L$ functions $L(\lambda+1/2,\pi_n \times \pi_{n+1})$ and $b(\lambda,\pi)$ which are the analogues of \eqref{eq:b_v} and \eqref{eq:local_RS}. By the unramified computations of \cite{CS} and \cite[Theorem~7.1]{Cog}, there exists $\tS$ a finite set of places of $F$ such that 
\begin{equation}
    \label{eq:global_facto_zeta}
    Z_\pi(\varphi,\lambda)=\frac{L(\lambda+\frac{1}{2},\pi_n \times \pi_{n+1})}{b(\lambda,\pi)} \times \prod_{v \in \tS} Z_{\pi_v}^\natural(\varphi_v,\lambda).
\end{equation}

The Zeta function is identically zero as soon as $\pi$ is not cuspidal as residual representations are not generic (see \cite[Theorem~2.7]{Li92} for a proof that their unramified components are not generic).

\subsubsection{Enter the Zeta integrals} We now relate the regularized periods $\cP$ and the Zeta integrals $Z_\pi$. We follows the strategy of \cite{IY}. The starting point is a weak form of the fine expansion of the Rankin--Selberg period for pseudo-Eisenstein series which vanish along singular affine hyperplanes. Its proof is nearly identical to that of \cite[Lemma~4.9]{IY} and \cite[Lemma~9.1.1]{LR}.

\begin{lem}
\label{lem:vanishing_unfolding}
    Let $P_{n+1}$ be a standard parabolic subgroup of $\GL_{n+1}$. Let $\Phi \in \cP\cW_{P_{n+1}}$ be a Paley--Wiener function (see \S\ref{subsubsec:pseudo_eisenstein}). Let $\varphi_n \in \cA(\GL_n)$. Assume that $\Phi(\lambda)$ vanishes along the affine hyperplanes of $\lambda \in \fa_{P_{n+1},\cc}^*$ such that
    \begin{equation}
    \label{eq:singular_hyperplanes}
        \langle w_Q^\std w \lambda + \lambda_n + \underline{\rho}_Q, \varpi^\vee \rangle=0,
    \end{equation}
    for any $Q \in \cF_{\RS}$, $w \in W(P_{n+1};Q^{\std}_{n+1})$, $\lambda_n \in \cE_{Q_n}(\varphi_n)$ and $\varpi^\vee \in \hat{\Delta}_{Q,H}^\vee$ (counted with multiplicities). Then for $\kappa \in \fa_{P_{n+1}}^*$ sufficiently positive we have
    \begin{equation*}
        \int_{[H]} \varphi_n(h) E(h,F_\Phi)dh = \int_{\Re(\lambda)=\kappa} \cP(\varphi_n \otimes E(\Phi(\lambda),\lambda)) d \lambda.
    \end{equation*}
    
\end{lem}

\begin{proof}
    Let $\kappa \in \fa_{P_{n+1}}^*$ be sufficiently positive. By \eqref{eq:pseudo_Eisenstein_unfold}, we have
    \begin{equation}
    \label{eq:pseudo_period}
        \int_{[H]} \varphi_n(h) E(h,F_\Phi)dh=\int_{[H]} \int_{\Re(\lambda)=\kappa} \varphi_n(h)E(h,\Phi(\lambda),\lambda) d\lambda dh.
    \end{equation}
    By Lemma~\ref{lem:unfolding_lambda} and the computation of the constant term of cuspidal Eisenstein series in \eqref{eq:constant_term_cuspidal}, we have for fixed $h$ and $\lambda$
    \begin{align*}
        \varphi_n(h)E(h,\Phi(\lambda),\lambda)=\sum_{Q \in \cF_{\RS}} &\sum_{\delta \in Q_H(F) \backslash H(F)} \tau_Q(H_{0,H}(\delta h)_Q-T_Q) \\
        \times&\sum_{w \in W(P_{n+1};Q^{\std}_{n+1})}\Lambda^{T,Q} \left((\varphi_n)_{Q_n} \otimes E^{Q_{n+1}}(M(w,\lambda)\Phi(\lambda),w\lambda)\right)(\delta h),
    \end{align*}
    where we write $E^{Q_{n+1}}(M(w,\lambda)\Phi(\lambda),w\lambda)$ for $E^{Q^{\std}_{n+1}}(w_Q^{\std,-1}\cdot,M(w,\lambda)\Phi(\lambda),w\lambda)$. Because these three sums are finite, we see that \eqref{eq:pseudo_period} is
    \begin{equation}
    \label{eq:partial_period}
        \int_{[H]} \sum_{Q} \sum_{w} \sum_{\delta} \tau_Q(H_{0,H}(\delta h)_Q-T_Q) \int_{\Re(\lambda)=\kappa} \Lambda^{T,Q}\left( (\varphi_n)_{Q_n} \otimes E^{Q_{n+1}}(M(w,\lambda)\Phi(\lambda),w\lambda)(\delta h) \right)d \lambda dh.
    \end{equation}
    Fix $Q$ and $w$. Write the normalized decomposition
    \begin{equation*}
        (\varphi_n)_{Q_n}(g)=\sum_i q_{Q,i}(H_Q(g))  \exp(\langle \lambda_{Q,i}, H_Q(g) \rangle)\varphi_{i}(g), \quad g \in G(\bA)
    \end{equation*}
    as in \eqref{eq:normalized_decompo_Z}. Let $\underline{\rho}_{Q}^{\std}$ be the element defined in \S\ref{subsubsec:rho_defi}. Because it belongs to $\overline{\fa_{Q^{\std}}^{*,+}}$, we see that for every $t>0$ it satisfies the following properties:
    \begin{enumerate}
        \item if $\alpha \in \Delta_{P_{n+1}}$ with $w \alpha <0$, then $\langle -tw^{-1}\underline{\rho}_{Q}^{\std},\alpha^\vee \rangle \geq 0$;
        \item for every $\alpha \in \Delta_{Q^{\std}_{n+1,w}}^{Q^{\std}_{n+1}}$, $\langle -t \underline{\rho}_{Q}^{\std},\alpha^\vee \rangle =0$;
        \item for every $\varpi^\vee \in \hat{\Delta}_{Q,H}^{\vee}$, $\lim \limits_{t \to \infty} \langle - tw_Q^\std \underline{\rho}_{Q}^{\std}, \varpi^\vee \rangle=\lim \limits_{t \to \infty} \langle - t\underline{\rho}_{Q}, \varpi^\vee \rangle = - \infty$.
    \end{enumerate}
    By 1 and 2, we can shift the contour of integration of the outermost integral of \eqref{eq:partial_period} to $\Re(\lambda)=\kappa-tw^{-1} \underline{\rho}_{Q}^{\std}$ without encountering any poles of $M(w,\lambda)$ or $E(\cdot,w\lambda)$. Moreover, we stay in their region of absolute convergence by \cite{MW95}, so that the integrand is of rapid decay in vertical strips by Theorem~\ref{thm:truncation_operator}. By 3, the factorization of the measures \eqref{eq:facto_measure}, Lemma~\ref{lem:Fourier_transform} and Theorem~\ref{thm:truncation_operator}, for $t$ large enough we have 
    \begin{align*}
         &\int_{[H]} \sum_{\delta} \tau_Q(H_{0,H}(\delta h)_Q-T_Q) \int_{\Re(\lambda)=\kappa} \Lambda^{T,Q} \left((\varphi_n)_{Q_n} \otimes E^{Q_{n+1}}(M(w,\lambda)\Phi(\lambda),w\lambda)\right)(\delta h) d \lambda dh \\
         =&\sum_i \int_{\Re(\lambda)=\kappa-tw^{-1}\underline{\rho}_{Q}^{\std}} \mathrm{FT}_Q(q_{Q,i},T,w_Q^\std w\lambda+\lambda_{Q,i}+\underline{\rho}_Q) \cP^{T,Q}(\varphi_i \otimes E^{Q_{n+1}}(M(w,\lambda)\Phi(\lambda),w\lambda)) d \lambda,
    \end{align*}
    where in the first line the first double integral is absolutely convergent. It follows the first quadruple integral in \eqref{eq:partial_period} is absolutely convergent, and that by shifting back the contour to $\kappa$, which is possible because of the vanishing assumptions on $\Phi$ by Lemma~\ref{lem:Fourier_transform}, it is
    \begin{equation*}
        \int_{\Re(\lambda)=\kappa} \sum_{Q,w,i} \mathrm{FT}_Q(q_{Q,i},T,w_Q^\std w\lambda+\lambda_i+\underline{\rho}_Q) \cP^{T,Q}(\varphi_{i} \otimes E^{Q_{n+1}}(M(w,\lambda)\Phi(\lambda),w\lambda)) d \lambda.
    \end{equation*}
    By \eqref{eq:general_period} the integrand is $\cP(\varphi_n \otimes E(\Phi(\lambda),\lambda))$, which concludes.
\end{proof}

\begin{prop}
\label{prop:Zydor_RS}
    Let $\varphi=\varphi_n \otimes \varphi_{n+1} \in \cA(G)$. For $s \in \cc$, write $\varphi_s=\varphi_{n,s} \otimes \varphi_{n+1}$. Then we have the equality of meromorphic functions $\cP(\varphi_s)=Z(\varphi,s)$. In particular, if $P$ is a standard parabolic subgroup of $G$ and if $\pi \in \Pi_\cusp(M_P)$, we have for $\lambda \in \fa_{P,\cc}^*$ in general position
    \begin{equation*}
        \cP(\varphi,\lambda)=Z_\pi(\varphi,\lambda).
    \end{equation*}
\end{prop}

\begin{proof}
    This is proved as in \cite[Theorem~1.1]{IY}. Let us sketch the proof. \cite{IY}, Ichino and Yamana build a regularized period $\cP^{IY}$ with similar properties as Zydor's version $\cP$. By \cite[Lemma~4.9]{IY}, it satisfies 
     \begin{equation*}
        \int_{[H]} \varphi_n(h) E(h,F_\Phi)dh = \int_{\Re(\lambda)=\kappa} \cP^{IY}(\varphi_n \otimes E(\Phi(\lambda),\lambda)) d \lambda,
    \end{equation*}
    under the same vanishing hypotheses on $\Phi$ as in Lemma~\ref{lem:vanishing_unfolding}. Moreover, it is proved in \cite[Lemma~4.7]{IY} that, still with these vanishing requirements, for $s \in \cc$ with $\Re(s)$ large enough and for any $f \in C_c^\infty(G(F_\infty))$ we have
    \begin{equation*}
        \int_{[H]} \varphi_{n,s}(h) E(h,f*F_\Phi)dh = \int_{\Re(\lambda)=\kappa} Z\left(\varphi_n \otimes E(I_P(\lambda,f)\Phi(\lambda),\lambda),s\right) d \lambda.
    \end{equation*}
    By repeating the argument of \cite[Section~4.10]{IY}, we see that for all $\lambda$ sufficiently positive, and therefore that for all regular $\lambda$ by analytic continuation we have
    \begin{equation*}
        \cP(\varphi_n \otimes E(\Phi(\lambda),\lambda))=\cP^{IY}(\varphi_n \otimes E(\Phi(\lambda),\lambda))=Z(\varphi_n \otimes E(\Phi(\lambda),\lambda),0).
    \end{equation*}
    Proposition~\ref{prop:Zydor_RS} now follows from Franke's theorem \cite{Franke} which states that any automorphic form can be obtained as a linear combination of derivatives of cuspidal Eisenstein series (see \cite[Section~4.6]{IY}).
\end{proof}

\begin{rem}
    \label{rem:differente_truncations}
    As a byproduct of the proof of Proposition~\ref{prop:Zydor_RS}, we obtain the equality $\cP(\varphi)=\cP^{IY}(\varphi)$ for any $\varphi \in \cA(G)^{\mathrm{reg}}$. One may moreover check that the space $\cA(G \times G')^*$ of \cite[Definition~3.2]{IY}, which is the space of automorphic forms on which $\cP^{IY}$ is defined, is indeed equal to our $\cA(G)^{\mathrm{reg}}$. However, the truncations $\Lambda^T$ and $\Lambda^{T,IY}$ are different. In the case $G=\GL_1 \times \GL_2$, we have $\cF_{\RS}=\{G,P_0,\overline{P}_0\}$ where $\overline{P}_0$ is the opposite Borel. Then
    \begin{equation}
    \label{eq:example_truncation}
        \Lambda^T \phi(h) = \phi(h) - 1_{H_{0,H}(h)\geq T} \cdot \phi_{P_0}(h) - 1_{H_{0,H}(h)\leq T} \cdot \phi_{\overline{P}_0}(h), \quad h \in \GL_1(\bA),
    \end{equation}
    while
    \begin{equation*}
        \Lambda^{T,IY} \phi(h) = \phi(h) - 1_{H_{0,H}(h)\geq T} \cdot \phi_{P_0}(h) - 1_{H_{0,H}(h)\leq -T} \cdot \phi_{\overline{P}_0}(h), \quad h \in \GL_1(\bA).
    \end{equation*}
  
\end{rem}

\subsection{Residues of Rankin--Selberg periods}
\label{sec:residues_RS}

In this section, we compute some residues of the regularized period $\lambda \mapsto \cP(\varphi,\lambda)$. 

\subsubsection{A naive notion of residues} \label{subsubsec:naive_residue} Let $m \geq 1$, let $f$ be a meromorphic function on $\cc^m$. Let $\Lambda$ be a non-zero affine linear form on $\cc^m$. Write $\cH_{\Lambda}$ for the affine hyperplane $\{ v \in \cc^m \; | \; \Lambda(v)=0\}$. The map $v \mapsto \Lambda(v) f(v)$ is a meromorphic function on $\cc^m$. Assume that $\cH_{\Lambda}$ is not contained in its polar divisor (i.e. $\cH_{\Lambda}$ is at most a simple polar divisor of $f$). Then its restriction to $\cH_{\Lambda}$ is a meromorphic function on $\cH$, and we set
\begin{equation*}
    \underset{\Lambda}{\Res} f:= \left( \Lambda f\right)_{| \cH_{\Lambda}}.
\end{equation*}
Let $\Lambda_1, \hdots, \Lambda_r$ be a family of affine linear forms such that the underlying family of linear forms is linearly disjoint. We consider the iterated residue
\begin{equation*}
    \underset{\Lambda_{r} \leftarrow \Lambda_{1}}{\Res} f:= \underset{\Lambda_{r}}{\Res} \hdots \underset{\Lambda_{1}}{\Res} f,
\end{equation*}
provided each residue is defined in the above sense. This is a meromorphic function on $\cH:=\bigcap \cH_{\Lambda_i}$. Note that the iterated residue a priori depends on the order of the affine linear forms. The following easy lemma gives a condition for when it does not.

\begin{lem}
\label{lem:residues_commute}
    Assume that there exists a meromorphic function $g$ on $\cc^m$ such that
    \begin{equation*}
        f(\lambda)=\frac{g(\lambda)}{\prod_{i=1}^r \Lambda_{i}(\lambda)},
    \end{equation*}
    and that moreover $\cH$ is not contained in any of the singularities of $g$. Then the residue $ \underset{\Lambda_{r} \leftarrow \Lambda_{1}}{\Res} f$ may be taken in any order and we have $\underset{\Lambda_{r} \leftarrow \Lambda_{1}}{\Res} f=g_{|\cH}$.
\end{lem}

\subsubsection{Residues as regularized periods}
Let $P$ be a standard parabolic subgroup of $G$ and $\pi \in \Pi_\cusp(M_P)$. Write the Levi factor $M_P=\left( \prod_{i=1}^{m_n} \GL_{n_{n,i}} \right) \times \left( \prod_{j=1}^{m_{n+1}} \GL_{n_{n+1,j}} \right)$ and $\pi=\boxtimes \pi_{n,i} \boxtimes \pi_{n+1,j}$ accordingly.

\begin{prop}
\label{prop:easy_residue}
Let $1 \leq i_1, \hdots, i_m \leq m_n$ (resp. $1 \leq j_1, \hdots, j_m \leq m_{n+1}$) be distinct indices such that $\pi_{n,i_l}=\pi_{n+1,j_l}^\vee$ for $1 \leq l \leq m$. Set $k=n-\sum_l n_{n,i_l}$.
\begin{enumerate}
    \item For every $l$, let $\Lambda_{l}$ be the affine linear form $\Lambda_{l}(\lambda)=\lambda_{n,i_l}+\lambda_{n+1,j_l}+\frac{1}{2}$. Let $Q_{n+1}^{\std}$ be the standard parabolic subgroup of $\GL_{n+1}$ of standard Levi factor $\left(\prod_{l=1}^m \GL_{n_{n,i_l}} \right) \times \GL_{k+1}$. Let $Q \in \cF_{\RS}$ be the element corresponding to $(Q^{\std}_{n+1},m+1)$ under the bijection of Corollary~\ref{cor:param_RS}. Let $w=(w_n,w_{n+1}) \in W(P;Q^{\std})$ be the only element such that $w_n(i_l)=w_{n+1}(j_l)=l$ for $1 \leq l \leq m$. Then for every $\varphi \in \cA_{P,\pi}(G)$ we have
    \begin{equation}
    \label{eq:residue_easy_+}
        \underset{\Lambda_m \leftarrow \Lambda_1}{\Res} \cP(\varphi,\lambda)=(-1)^m\cP^Q(\varphi,\lambda,w).
    \end{equation}
    \item For every $l$, let $\Lambda_l$ be the affine linear form $\Lambda_l(\lambda)=\lambda_{n,i_l}+\lambda_{n+1,j_l}-\frac{1}{2}$. Let $Q^{\std}_{n+1}$ be the standard parabolic subgroup of $\GL_{n+1}$ with standard Levi factor $\GL_{k+1} \times \left(\prod_{l=m}^1 \GL_{n_{n,i_l}} \right) $. Let $Q \in \cF_{\RS}$ be the element corresponding to $(Q^{\std}_{n+1},1)$ under the bijection of Corollary~\ref{cor:param_RS}. Let $w=(w_n,w_{n+1}) \in W(P;Q^{\std})$ be the only element such that $w_n(i_l)=m_n-l+1$ and $w_{n+1}(j_l)=m_{n+1}-l+1$ for $1 \leq l \leq m$. Then for every $\varphi \in \cA_{P,\pi}(G)$ we have
    \begin{equation}
    \label{eq:residue_easy_-}
        \underset{\Lambda_m \leftarrow \Lambda_1}{\Res} \cP(\varphi,\lambda)=\cP^Q(\varphi,\lambda,w).
    \end{equation}
\end{enumerate}
\end{prop}

\begin{proof}
    We begin with the first case \eqref{eq:residue_easy_+} under the assumption that $m=1$. By the functional equation of Eisenstein series \cite[Theorem~2.3]{BL}, we have for $\lambda$ in general position
    \begin{equation*}
        \cP(\varphi,\lambda)=\cP(M(w,\lambda)\varphi,w\lambda).
    \end{equation*}
    Note that $M(w,\lambda)\varphi$ is regular for $\lambda$ in general position in the affine hyperplane $\lambda_{n,i}+\lambda_{{n+1},j}+\frac{1}{2}=0$. It follows from Corollary~\ref{cor:P_regular} that it is enough to prove \eqref{eq:residue_easy_+} for $i=j=1$ and $w=1$. 
    
    We now assume that $i=j=1$ and $w=1$. We set $\Lambda(\lambda)=\lambda_{n,1}+\lambda_{n+1,1}+\frac{1}{2}$. By the first statement of Proposition~\ref{prop:unfold_periods}, we have for any $T$ sufficiently positive
    \begin{equation}
        \label{eq:reversed_MS}
        \cP(\varphi,\lambda)=\sum_{R} \varepsilon_{R} \sum_{w \in W(P;R^{\std})} \cP^{T,R}(\varphi,\lambda,w) \cdot \frac{\exp(\langle w_R^\std w\lambda+\underline{\rho}_R,T_R \rangle)}{\hat{\theta}_R(w_R^\std w \lambda+\underline{\rho}_R)}.
    \end{equation}
    By the description of $\hat{\theta}_R$ in \S\ref{subsubsec:RS_coord} and \S\ref{sec:pol_exp_Fourier}, we see that $\Lambda$ divides $\hat{\theta}_R(w_R^\std w \lambda+\underline{\rho}_R)$ if and only if $\bfM_{R,+}$ is of the form $(\GL_{n_{n,1}} \times \hdots)$ and $w_n(1)=1$, $w_{n+1}(1)=1$. In that case, denote by $\hat{\theta}_R^\natural$ the quotient of $\hat{\theta}_R(w_R^\std w \lambda+\underline{\rho}_R)$ by $\Lambda$. By \cite[Theorem~2.3]{BL}, the generalized Eisenstein series $E^{R^\std}(M(w,\lambda)\varphi,w\lambda)$ appearing in \eqref{eq:reversed_MS} are regular along our affine hyperplane. By Theorem~\ref{thm:truncation_operator} and the continuity of Eisenstein series from \cite[Theorem~2.2]{Lap}, so are their truncated periods so that
    \begin{equation}
        \label{eq:first_residue}
        \underset{\Lambda}{\Res} \; \cP(\varphi,\lambda)=\sum_{(R,w)} \varepsilon_R\cP^{T,R}(\varphi,\lambda,w) \cdot \frac{\exp(\langle w_R^\std w\lambda+\underline{\rho}_R,T_R \rangle)}{\hat{\theta}^\natural_R(w_R^\std w \lambda+\underline{\rho}_R)},
    \end{equation}
    where $(R,w)$ ranges in the couples we just described. But these are exactly the pairs with $R \subset Q$ a Rankin--Selberg parabolic subgroup of $G$ and $w \in W^{Q^{\std}}(Q^{\std};R^{\std})$. Moreover, because of the volume computation of \eqref{eq:basis_volume}, $\hat{\theta}^\natural_R(w_R^\std w \lambda+\underline{\rho}_R)$ is equal to the restriction to $\Lambda^{-1}(\{0\})$ of $\hat{\theta}^Q_R(w_R^\std w \lambda+\underline{\rho}_R)$. Because $\varepsilon_R^Q=-\varepsilon_R$, it follows from Proposition~\ref{prop:unfold_periods} that 
    \begin{equation}
    \label{eq:k=1_case}
        \underset{\Lambda}{\Res} \; \cP(\varphi,\lambda)=-\cP^Q(\varphi,\lambda).
    \end{equation}

    To prove \eqref{eq:residue_easy_+} if $m \geq 1$, it remains to do induction on $k$ thanks to \eqref{eq:k=1_case} and to use parabolic descent by Lemma~\ref{lem:parabolic_descent}.

    The second case is exactly the same, the only difference being a minus sign which appears because of the signs in~\ref{eq:coweight_defi}.
\end{proof}

\section{Residues of global Zeta integrals}
\label{sec:RS_non_tempered}

Let $\Pi$ be an automorphic representation of Arthur type of $G$, realized as a quotient of $\cA_{P_\pi,\pi,-\nu_\pi}(G)$ as in \S\ref{subsubsec:residual_blocks}. In this section, we build a $H(\bA)$-invariant linear form $\cP_\pi$ on $\cA_{P_\pi,\pi,-\nu_\pi}(G)$, and show that it factors through $\Pi$ if $\Pi$ is in general position (see Theorem~\ref{thm:non_tempered_periods}). Because it is a residue of a global Zeta integral, it admits an Euler product decomposition. We also show in \S\ref{subsubsec:residue_free} that $\cP_\pi$ is a regularized period along a degeneracy of $[H]$. 

\subsection{Periods for relevant inducing pairs}
\label{sec:relevant_discrete}

Because $\cP_\pi$ is defined on an induction, we prefer to keep track of the choice of the inducing data for $\Pi$. We translate the relevance condition of \S\ref{subsubsec:RS_for_Arthyr} for Arthur parameters in this language.

\subsubsection{The relevance condition}
\label{subsubsec:relevant_inducing}
Let $\Pi_H$ be the set of \emph{relevant inducing pairs}. This is the set of pairs $(P,\pi)$ where $P$ is a standard parabolic subgroup of $G$ whose standard Levi factor $M_P=M_{P,n} \times M_{P,n+1}$ admits a decomposition 
 \begin{align}
    \label{eq:MP_n}
        M_{P,n}&=\prod_{i=1}^{m_1} \GL_{d(1,i)r(1,i)} \times \prod_{i=1}^{m_2} \GL_{(d(2,i)-1) r(2,i)} , \\
    \label{eq:MP_n+1}
        M_{P,{n+1}}&=\prod_{i=1}^{m_1} \GL_{(d(1,i)-1)r(1,i)} \times \prod_{i=1}^{m_2} \GL_{d(2,i) r(2,i)};
    \end{align}
and $\pi \in \Pi_\disc(M_P)$ is a discrete automorphic representation (with trivial central character on $A_P^\infty$) and such that, with respect to \eqref{eq:MP_n} and \eqref{eq:MP_n+1}, $\pi=\pi_n \boxtimes \pi_{n+1}$ decomposes as
     \begin{equation}
    \label{eq:pi_n}
        \pi_n=\boxtimes_{i=1}^{m_1} \pi_{1,i} \boxtimes_{i=1}^{m_2} \pi_{2,i}^{-,\vee}, \quad \pi_{n+1}=\boxtimes_{i=1}^{m_1} \pi_{1,i}^{-,\vee} \boxtimes_{i=1}^{m_2} \pi_{2,i},
    \end{equation}    
    where we set $\Speh(\sigma,d)^-=\Speh(\sigma,d-1)$, and we have $\pi_{1,i}=\Speh(\sigma_{1,i},d(1,i))$, $\pi_{2,i}=\Speh(\sigma_{2,i},d(2,i))$ for some representations $\sigma_{1,i} \in \Pi_\cusp(\GL_{r(1,i)})$ and $\sigma_{2,i} \in \Pi_\cusp(\GL_{r(2,i)})$. By convention $\Speh(\sigma,0)$ is the trivial representation of the trivial group.

Let $(P,\pi) \in \Pi_H$. With the choices of coordinates made in \S\ref{subsubsec:blocks}, 
$\fa_{P}^*$ is realized as a subspace
\begin{equation}
\label{eq:a_P_explicit coordinates}
    \fa_{P}^* \subset \left(\rr^{m_1} \times \rr^{m_2}\right) \times \left(\times \rr^{m_1} \times \rr^{m_2} \right).
\end{equation}
A similar decomposition holds for $\fa_{P,\cc}^*$. If $\lambda \in \fa_{P}^*$, we write 
\begin{equation}
\label{eq:notation_P}
    \lambda=\left( (\lambda(1)_n, \lambda(2)_n),(\lambda(1)_{n+1}, \lambda(2)_{n+1}) \right)
\end{equation}
according to this decomposition. Note that $\lambda(2)_{n,i}=0$ if $d(2,i)=1$, and $\lambda(1)_{n+1,i}=0$ if $d(1,i)=1$.

We define the anti-diagonal subspace $\fa_\pi^* \subset \fa_P^*$ to be
\begin{equation}
\label{eq:a_pi_defi}
    \fa_{\pi}^*=\left\{ \lambda \in \fa_{P}^* \; \middle| \; \begin{array}{ll}
      
        \lambda(1)_{n,i}=-\lambda(1)_{n+1,i}, & 1 \leq i \leq m_1, \; \text{if } d(1,i) \neq 1, \\
        \lambda(2)_{n,i}=-\lambda(2)_{n+1,i}, & 1 \leq i \leq m_2, \; \text{if } d(2,i) \neq 1,
    \end{array} \right\}.
\end{equation}
We have an isomorphism 
\begin{equation}
\label{eq:coord_a_pi}
    \lambda \in \fa_{\pi}^* \mapsto \left( \lambda(1), \lambda(2) \right):=\left( \lambda(1)_n, \lambda(2)_{n+1} \right) \in \rr^{m_1} \times \rr^{m_2}.
\end{equation}
We also have an anti-diagonal subspace $\fa_{\pi,\cc}^* \subset \fa_{P,\cc}^*$ defined by the same equations. Note that $i\fa_{\pi}^{*}$ is exactly the subspace of $\lambda \in \fa_{P,\cc}^*$ such that $(P,\pi_\lambda) \in \Pi_H$ if we lift the requirement that the central character is trivial on $A_P^\infty$, and ask that it is unitary instead.

\subsubsection{Singular affine linear planes}
\label{subsubsec:linear_forms}

Let $(P,\pi) \in \Pi_H$ be a relevant inducing pair. Let $\sigma_\pi$ and $\nu_\pi$ be respectively the cuspidal automorphic representation of $M_{P_\pi}$ and the element of $\fa_{P_\pi}^*$ such that $\cA_{P,\pi}(G)$ is obtained by taking residues of Eisenstein series on the induction $\cA_{P_\pi,\sigma_\pi,-\nu_\pi}(G)$ (see \S\ref{subsubsec:residual_blocks}). For $\lambda \in \fa_{P_\pi,\cc}^*$ in general position, we have the global Zeta function $Z_{\sigma_\pi}(\cdot,\lambda)$ from \S\ref{subsubsec:global_zeta}. We can identify $\fa_{\pi,\cc}^*-\nu_\pi$ with a subspace of $\fa_{P_\pi,\cc}^*$. It is contained in a finite union of singularities of $Z_{\sigma_\pi}(\cdot,\lambda)$ which are all affine hyperplanes. We now describe the corresponding affine linear forms.

By assumption, we can write $M_{P_\pi}$ the standard Levi factor of $P_\pi$ as follows:
  \begin{align}
        &M_{P_\pi,n}=\prod_{i=1}^{m_1} (\GL_{r(1,i)})^{d(1,i)} \times \prod_{i=1}^{m_2} (\GL_{r(2,i)})^{d(2,i)-1} , \label{eq:M_P_pi_n} \\
        &M_{P_\pi,n+1}=\prod_{i=1}^{m_1} (\GL_{r(1,i)})^{d(1,i)-1} \times \prod_{i=1}^{m_2} (\GL_{r(2,i)})^{d(2,i)} . \label{eq:M_P_pi_n+1}
    \end{align}
    This yields an identification
    \begin{equation*}
        \fa_{P_\pi,\cc}^* = \left(\prod_{i=1}^{m_1} \cc^{d(1,i)} \times \prod_{i=1}^{m_2} \cc^{d(2,i)-1}\right) \times \left(\prod_{i=1}^{m_1} \cc^{d(1,i)-1} \times \prod_{i=1}^{m_2} \cc^{d(2,i)} \right).
    \end{equation*}
    We will from now on write the coordinates of any $\lambda \in \fa_{P_\pi,\cc}^*$ with respect to these identifications. More precisely, if $\lambda=(\lambda_n,\lambda_{n+1}) \in \fa_{P_\pi,\cc}^*$ we write $\lambda_n$ as
    \begin{equation}
        \label{eq:lambda_coordinates}
        ( \lambda(1,1)_n, \hdots, \lambda(1,m_1)_n,  \lambda(2,1)_n, \hdots ,\lambda(2,m_2)_n)
    \end{equation}
    where for example $\lambda(1,1)_n \in \cc^{d(1,1)}$ has coordinates $\lambda(1,1)_{n,1}, \hdots ,\lambda(1,1)_{n,d(1,1)}$. The same applies for $\lambda_{n+1}$. 

    We now build a set $L_+$ of affine linear forms on $\fa_{P_\pi,\cc}^*$. They are defined as
    \begin{equation}
        \left\{
            \begin{array}{ll}
                \Lambda_+(1,i,j)(\lambda)=-(\lambda(1,i)_{n,d(1,i)-j+1}+\lambda(1,i)_{n+1,j}+1/2), &  \quad  \left\{\begin{array}{ll}
                    1 \leq i \leq m_1,  \\
                    1 \leq j \leq d(1,i)-1,
                \end{array}\right.   \\
                \Lambda_+(2,i,j)(\lambda)=-(\lambda(2,i)_{n,j}+\lambda(2,i)_{n+1,d({2,i})-j+1}+1/2), &  \quad \left\{\begin{array}{ll}
                    1 \leq i \leq m_2,   \\
                    1 \leq j \leq d(2,i)-1.
                \end{array}\right.   
            \end{array}
        \right.
        \label{eq:Lambda}
    \end{equation}
    Set $\cH_+=\cap_{\Lambda_+ \in L_+} \cH_{\Lambda_+}$. We also have the set $L_-$ of linear forms defined by the following equations.
    \begin{equation}
        \left\{
            \begin{array}{ll}
                \Lambda_-(1,i,j)(\lambda)=\lambda(1,i)_{n,d(1,i)-j}+\lambda(1,i)_{n+1,j}-1/2, &  \quad  \left\{\begin{array}{ll}
                    1 \leq i \leq m_1,   \\
                    1 \leq j \leq d(1,i)-1,
                \end{array}\right.   \\
                \Lambda_-(2,i,j)(\lambda)=\lambda(2,i)_{n,j}+\lambda(2,i)_{n+1,d(2,i)-j}-1/2, &  \quad  \left\{\begin{array}{ll}
                    1 \leq i \leq m_2,   \\
                    1 \leq j \leq d(2,i)-1.
                \end{array}\right.  
            \end{array}
        \right.
                \label{eq:Lambda'}
    \end{equation}
    Set $\cH_-=\cap_{\Lambda_- \in L_-} \cH_{\Lambda_-}$. If we restrict to $\lambda \in \cH_+$, we have (for $i$ and $j$ in the suitable range)
    \begin{equation}
    \label{eq:L'_alternative}
        \left\{
            \begin{array}{l}
                \Lambda_-(1,i,j)(\lambda)=\lambda(1,i)_{n,d(1,i)-j}-\lambda(1,i)_{n,d(1,i)-j+1}-1=\lambda(1,i)_{n+1,j}-\lambda(1,i)_{n+1,j+1}-1,\\
                \Lambda_-(2,i,j)(\lambda)=\lambda(2,i)_{n,j}-\lambda(2,i)_{n,j+1}-1=\lambda(2,i)_{n+1,d(2,i)-j}-\lambda(2,i)_{n+1,d(2,i)-j+1}-1.
            \end{array}
        \right.
    \end{equation}
    It follows that
    
    \begin{equation}
    \label{eq:intersection_hyperplanes}
        \cH_+ \cap \cH_-=\fa_{\pi,\cc}^*-\nu_{\pi}.
    \end{equation}

    \begin{lem}
        \label{lem:iterated_residue}
        Let $\phi \in \cA_{P_\pi,\sigma_\pi,-\nu_\pi}(G)$. The residue $\underset{L_+ \cup L_-}{\Res} \; Z_{\sigma_\pi}(\phi,\lambda)$ is a well-defined meromorphic function on $\fa_{\pi,\cc}^*-\nu_{\pi}$ and can be computed with respect to any order on the set $L_+ \cup L_-$.
    \end{lem}

    \begin{proof}
        We can assume that $\phi=\otimes_v' \phi_v$ is factorizable. By \eqref{eq:global_facto_zeta}, we have the Euler product expansion for $\tS$ large enough
        \begin{equation*}
            Z_{\sigma_\pi}(\phi,\lambda)=\frac{L(\lambda+1/2,\sigma_{\pi,n} \times \sigma_{\pi,n+1})}{b(\lambda,\sigma_\pi)} \prod_{v \in \tS} Z^\sharp_{\sigma_\pi,v}(\phi_v,\lambda).
        \end{equation*}
        Because $-\nu_\pi \in\fa_{P_\pi}^{P,*,+}$, we see using the description of $\fa_\pi^*$ in \eqref{eq:a_pi_defi}, that for every $\alpha \in \Sigma_{P_\pi}$ the map $\lambda \in \fa_{\pi,\cc}^*-\nu_\pi \mapsto \langle \lambda,\alpha^\vee \rangle$ is either non-constant, either constant equal to a positive integer. By Lemma~\ref{lem:local_zeta_reg}, we conclude that $\fa_{\pi,\cc}^*-\nu_{\pi}$ is not contained in any of the singularities of the product of the local terms, which all come from $b(\lambda,\sigma_{\pi,v})$. By the properties of Rankin--Selberg $L$-functions, the affine linear forms in $L_+ \cup L_-$ direct all the singularities of the quotient of global $L$-functions containing $\fa_{\pi,\cc}^*-\nu_{\pi}$, which are all affine hyperplanes with multiplicity one coming from $L(\lambda+1/2,\sigma_{\pi,n} \times \sigma_{\pi,n+1})$. Therefore, Lemma~\ref{lem:residues_commute} concludes.
    \end{proof}

\subsubsection{Residues of Zeta functions and main result}

\label{subsection:extension_RS}

Let $(P,\pi) \in \Pi_H$ be a relevant inducing pair. We denote by $\Res \; Z_{\sigma_\pi}(\cdot,\lambda)$ the iterated residue obtained in Lemma~\ref{lem:iterated_residue}. For $\lambda \in \fa_{\pi,\cc}^*$ we set, with the coordinates of \eqref{eq:coord_a_pi},
\begin{equation*}
     \cL(\lambda,\pi)=\frac{\displaystyle \prod_{i,j} L\left(\lambda(1)_i +\lambda(2)_j + \frac{d(1,i)-d(2,j)+1}{2}, \sigma_{1,i} \times \sigma_{2,j} \right)}{\displaystyle \prod_{k=1}^2 \prod_{1 \leq i<j \leq m_k} L\left(\lambda(k)_i-\lambda(k)_j+\frac{d(k,i)+d(k,j)}{2}, \sigma_{k,i} \times \sigma_{k,j}^\vee \right)}.
\end{equation*}

\begin{theorem}
\label{thm:non_tempered_periods}
The following properties hold.
\begin{enumerate}
    \item For $\lambda \in \fa_{\pi,\cc}^*$ in general position, $\Res \; Z_{\sigma_\pi}(\cdot,\lambda-\nu_\pi)$ factors through $\cA_{P_\pi,\sigma_{\pi},\lambda-\nu_\pi}(G) \twoheadrightarrow \cA_{P,\pi,\lambda}(G)$ and yields a $H(\bA)$-invariant linear form $\cP_\pi(\cdot,\lambda)$ on this induction.
    \item There exists a regular never vanishing function $\varepsilon$ (which is a product of epsilon factors and of special values of $L$-functions) such that for $\lambda \in \fa_{\pi,\cc}^*$ in general position and $\varphi \in \cA_{P,\pi}(G)$ with $\varphi=E^{P,*}(\phi,-\nu_\pi)$ and $\phi=\otimes_v \phi_v$ we have
         \begin{equation*}
        \cP_\pi(\varphi,\lambda)=\varepsilon(\lambda)\cL(\lambda,\pi) \prod_{v \in \tS} Z_{\sigma_{\pi,v}}^\natural(\phi_v,\lambda-\nu_\pi),
    \end{equation*}
    where $\tS$ is any sufficiently large finite set of places of $F$.
\end{enumerate}
\end{theorem}

\begin{rem}
    The possible singularities of $\cP_{\pi}(\varphi,\lambda)$ come either from $\cL(\lambda,\pi)$, either from the $Z_{\sigma_\pi,v}^\natural$. It turns out that the latter are regular for $\lambda \in i \fa_\pi^*$. This will be proved in Theorem~\ref{thm:local_GGP_explicit}.
\end{rem}

\subsection{Computation of the residues}

We now compute $\Res \; Z_{\sigma_\pi}(\cdot,\lambda-\nu_\pi)$ by varying the order of the linear forms in $L_+ \cup L_-$. Because we know that $ Z_{\sigma_\pi}(\cdot,\lambda-\nu_\pi)$ is $\cP(\cdot,\lambda)$ (Proposition~\ref{prop:Zydor_RS}), we can use Proposition~\ref{prop:easy_residue}.

    \subsubsection{Residues along $L_+$}
    \label{subsubsec:residues_cL}

    Let $\phi \in \cA_{P_\pi,\sigma_\pi}(G)$. By Lemma~\ref{lem:iterated_residue}, we have the meromorphic function $\underset{L_+}{\Res} \; \cP(\phi,\lambda)$ on $\cH_+$. We compute it using Proposition~\ref{prop:easy_residue} by varying the order in $L_+$.

    We start by taking the order
    \begin{equation}
         \left( \xleftarrow[i=1]{m_2} \xleftarrow[j=1]{d(2,i)-1} \Lambda_+(2,i,j) \right)\leftarrow
         \left( \xleftarrow[i=1]{m_1} \xleftarrow[j=d(1,i)-1]{1} \Lambda_+(1,i,j) \right)\label{eq:first_order}
    \end{equation}
    where for example the notation $\xleftarrow[i=1]{m_1} \xleftarrow[j=d(1,i)-1]{1} \Lambda_+(1,i,j)$ means
    \begin{equation*}
     \Lambda_+(1,m_1,1) \leftarrow \hdots \leftarrow \Lambda_+(1,2,d(1,2)-1) \leftarrow \Lambda_+(1,1,1) \leftarrow \hdots \leftarrow \Lambda_+(1,1,d(1,1)-1)
    \end{equation*}
    By Proposition~\ref{prop:easy_residue}, we have an explicit description of $\underset{L_+}{\Res} \; \cP(\varphi,\lambda)$. Set 
    \begin{equation}
    \label{eq:defi_k}
        k=\sum_{i=1}^{m_1} r(1,i)=\sum_{i=1}^{m_2} r(2,i)-1.
    \end{equation}
    Let $P_{\res,n+1}^\std$ be the standard parabolic subgroup of $\GL_{n+1}$ with standard Levi factor
    \begin{equation}
    \label{eq:Q_cL_defi}
        M_{P_{\res}^\std,n+1}=\prod_{i=1}^{m_1} (\GL_{r(1,i)})^{d(1,i)-1}  \prod_{i=1}^{m_2} (\GL_{r(2,i)})^{d(2,i)-1} \times \GL_{k+1}.
    \end{equation}
    To ease notation, set
     \begin{equation}
    \label{eq:Ds}
        d(1,\leq i)=\sum_{j=1}^{i} (d(1,j)-1),  \quad 
            d(2,\leq i)=\sum_{j=1}^{i} (d(2,j)-1), \quad D=d(1,\leq m_1) + d(2,\leq m_2).
    \end{equation}
    Let $P_{\res}$ be the Rankin--Selberg parabolic subgroup of $G$ associated to $(P^\std_{\res,n+1},D+1)$ by the bijection of Corollary~\ref{cor:param_RS}. Note that it is standard. We simply write $\bfM_{\res}^2$ and $\cM_{\res}$ for $\bfM_{P_{\res}}^2$ and $\cM_{P_{\res}}$ respectively. We now describe the element $w$ appearing in \eqref{eq:residue_easy_+}. It belongs to $W(P_\pi;P_{\res})$. We first build an element $w_1=(w_{1,n},w_{1,n+1}) \in W(P_\pi)$. Let $m_{\pi,n}$ (resp. $m_{\pi,n+1}$) be the number of blocks of $M_{P_\pi,n}$ (resp. $M_{P_\pi,n+1}$) so that we identify $w_{1,n}$ with a permutation in $\fS(m_{\pi,n})$ (resp. $w_{1,n+1}$ with a permutation in $\fS(m_{\pi,n+1})$). Then $w_1$ is a product of cycles. More precisely, for every $1 \leq i \leq m_1$ let $w_{n,i} \in \fS(m_{\pi,n})$ be the cycle
    \begin{equation*}
        w_{n,i}=\left(d(1,\leq i-1) + 1 \quad D+m_1 \quad D+m_1-1 \quad \hdots \quad  d(1,\leq i-1) +2 \right)
    \end{equation*}
    and let for $1 \leq i \leq m_2$ let $w_{n+1,i} \in \fS(m_{\pi,n+1})$ be the cycle following the same pattern
    \begin{equation*}
        w_{n+1,i}=\left( d(1,\leq m_1) +d(2,\leq i) + 1  \quad D+m_2 \quad \hdots  \quad   d(1,\leq m_1) +d(2,\leq i) + 2 \right)
    \end{equation*}
    Then we set
    \begin{equation}
    \label{eq:w_cL_defi}
        w_{1,n}=w_{n,m_1} \hdots w_{n,1}, \quad w_{1,n+1}=w_{n+1,m_2} \hdots w_{n+1,1}, \quad w_1=(w_{1,n},w_{1,n+1}).
    \end{equation}
    More concretely, $w_{1,n}$ sends $M_{P_\pi,n}$ to 
   \begin{equation}
    \label{eq:w_1_action_1}
    \prod_{i=1}^{m_1} (\GL_{r(1,i)})^{d(1,i)-1} \times \prod_{i=1}^{m_2} (\GL_{r(2,i)})^{d(2,i)-1} \times \prod_{i=1}^{m_1} \GL_{r(1,i)},
   \end{equation}
   by sending the first $\GL_{r(1,i)}$ in each product $(\GL_{r(1,i)})^{d(1,i)}$ to the bottom-right corner. On the other hand, $w_{1,n+1}$ sends $M_{P_\pi,n+1}$ to 
    \begin{equation}
        \label{eq:w_1_action_2}
    \prod_{i=1}^{m_1} (\GL_{r(1,i)})^{d(1,i)-1} \times \prod_{i=1}^{m_2} (\GL_{r(2,i)})^{d(2,i)-1} \times \prod_{i=1}^{m_2} \GL_{r(2,i)},
   \end{equation}
   this time by sending the last block of each $(\GL_{r(2,i)})^{d(2,i)}$ to the corner. By elementary computations, the element $w$ which appears in \eqref{eq:residue_easy_+} is $w_1 w_{\pi,n+1}^*$, where $w_{\pi,n+1}^* \in \GL_{n+1}$ is the element defined in \S\ref{subsubsec:residual_blocks} such that $\cA_{P_{n+1},\pi_{n+1}}(\GL_{n+1})$ is the image of $M^*(w_{\pi,n+1},-\nu_{\pi,n+1})$. By Proposition~\ref{prop:easy_residue}, we have for $\lambda \in \cH_+$ in general position
   \begin{equation}
   \label{eq:residue_L_1}
       \underset{L_+}{\Res} \; \cP(\phi,\lambda)=\cP^{P_{\res}}(\phi,\lambda,w_1 w_{\pi,n+1}^*).
   \end{equation}

    We now compute $\underset{L_+}{\Res} \; \cP(\phi,\lambda)$ by choosing a different order on $L_+$. Consider 
     \begin{equation}
         \left( \xleftarrow[i=1]{m_2} \xleftarrow[j=d(2,i)-1]{1} \Lambda_+(2,i,j) \right)\leftarrow
         \left( \xleftarrow[i=1]{m_1} \xleftarrow[j=1]{d(1,i)-1} \Lambda_+(1,i,j) \right) \label{eq:second_order}
    \end{equation}
    For every $1 \leq i \leq m_1$, let $w_{n,i}' \in \fS(m_{\pi,n})$ be the cycle 
    \begin{equation*}
        w_{n,i}'=\left( d(1,\leq i) + 1 \quad D+m_1 \quad D+m_1-1 \quad \hdots \quad  d(1,\leq i) + 2 \right)
    \end{equation*}
    and for $1 \leq i \leq m_2$, set
    \begin{equation*}
        w_{n+1,i}'=\left(  d(1,\leq m_1) +d(2,\leq i-1) + 1  \quad D+m_2 \quad \hdots \quad  d(1,\leq m_1) +d(2,\leq i-1) + 2 \right).
    \end{equation*}
    Write  
    \begin{equation}
    \label{eq:w_cL_'_defi}
        w_{2,n}=w_{n,m_1}' \hdots w_{n,1}', \quad w_{2,n+1}=w_{n+1,m_2}' \hdots w_{n+1,1}', \quad w_2=(w_{2,n},w_{2,n+1}).
    \end{equation}
    The action of $w_2$ can be described as in \eqref{eq:w_1_action_1} and \eqref{eq:w_1_action_2}. In particular, $w_1 w_{\pi,n+1}^* .P_\pi=w_2 w_{\pi,n}^* .P_\pi$. Another application of Proposition~\ref{prop:easy_residue} gives for $\lambda \in \cH_+$ in general position
    \begin{equation}
    \label{eq:residue_L_2}
       \underset{L_+}{\Res} \; \cP(\phi,\lambda)=\cP^{P_{\res}}(\phi,\lambda,w_2 w_{\pi,n}^*).
   \end{equation}

    \subsubsection{Residues along $L_-$} 
    \label{subsubsec:residues_cL'}

    We now compute the residues along $L_-$. The key property we need is the following lemma.

   \begin{lem}
    \label{lem:w_1_action} The following assertions hold.

    \begin{itemize}
        \item For every $\alpha \in \Sigma_{P_\pi}$ such that $w_{\pi,n+1}^* \alpha <0$, the map $\lambda \mapsto \langle \lambda,\alpha^\vee \rangle$ is non-constant on $\cH_+$, and either non-constant or equal to a positive integer on $\cH_+ \cap \cH_-$. Let $\Sigma_1$ be the subset of $\alpha$'s such that this integer is $1$. Then for $\lambda \in \cH_+$
        \begin{equation*}
            \prod_{\alpha \in \Sigma_1} (\langle \lambda,\alpha^\vee \rangle -1)=\frac{\prod_{\Lambda_- \in L_-} \Lambda_-(\lambda)}{\prod_{i=1}^{m_1} \Lambda_-(1,i,d(1,i)-1)(\lambda)}.
        \end{equation*}
        \item For every $\alpha \in \Sigma_{P_\pi}$ such that $w_1 \alpha <0$, the map $\lambda \mapsto \langle w_{\pi,n+1}^* \lambda,\alpha^\vee \rangle$ is non-constant on $\cH_+$, and either non-constant or equal to a positive integer on $\cH_+ \cap \cH_-$. Let $\Sigma'_1$ be the subset of $\alpha$'s such that this integer is $1$. Then for $\lambda \in \cH_+$
    \begin{equation*}
        \prod_{\alpha \in \Sigma_1'}(\langle w_{\pi,n+1}^* \lambda,\alpha^\vee \rangle-1)=\prod_{i=1}^{m_1} \Lambda_-(1,i,d(1,i)-1)(\lambda).
    \end{equation*}
    \end{itemize}
   \end{lem}

   \begin{proof}
        For the first assertion of each point, note that the maps $\lambda \in \cH_+ \mapsto \lambda_n \in \fa_{P_\pi,n,\cc}^*$ and $\lambda \in \cH_+ \mapsto \lambda_{n+1} \in \fa_{P_\pi,n+1,\cc}^*$ are both surjective. The second assertion follows from the explicit description of $w_1$ in \eqref{eq:w_1_action_1} and \eqref{eq:w_1_action_2}, and of the restriction of the $\Lambda_-$'s to $\cH_+$ in \eqref{eq:L'_alternative}.
   \end{proof}

   We can now fully compute the residue using the first order \eqref{eq:first_order}.
    
    \begin{lem}
    \label{lem:second_residue_n+1}
        For $\lambda \in \cH_+$, define the meromorphic map 
        \begin{equation}
        \label{eq:regularized_L}
            M^*(w_{1},w_{\pi,n+1}^*\lambda)=\left(\prod_{i=1}^{m_1} \Lambda_-(1,i,d(1,i)-1)(\lambda)\right) M(w_{1},w_{\pi,n+1}^*\lambda).
        \end{equation}
        Then we have the equality of meromorphic functions on $\cH_+ \cap \cH_-$
        \begin{equation*}
             \underset{L_-}{\Res} \; \underset{L_+}{\Res} \; \cP(\phi,\lambda)= \cP^{P_{\res}}\left(M^*(w_{1},w_{\pi,n+1}^*\lambda)M^*(w_{\pi,n+1}^*,\lambda) \phi,w_1 w_{\pi,n+1}^*\lambda\right),
        \end{equation*}
        where $M^*(w_{\pi,n+1}^*,\lambda) \phi$ is the regularized global intertwining operator (on the $\GL_{n+1}$ component) defined in \eqref{eq:reg_operator}. Moreover, $\cH_+ \cap \cH_-$ is not contained in any of the singularities of $M^*(w_{1},  w_{\pi,n+1}^*\lambda)$.
    \end{lem}

    \begin{proof}
    By Lemma~\ref{lem:w_1_action}, the restriction of $M(w_1,w_{\pi,n+1}^*\lambda)$ to $\cH_+$ is well-defined, and moreover by Theorem~\ref{thm:N} all its singularities that contain $\cH_+ \cap \cH_-$ must come from poles of the global factor $n_{\sigma_\pi}(w_1,w_{\pi,n+1}^*\lambda)$. As this factor is a product of Rankin--Selberg $L$-functions of cuspidal representations by \eqref{eq:n_pi_formula}, they must lie along affine hyperplanes such that $\langle w_{\pi,n+1}^* \lambda,\alpha^\vee \rangle$ is constant equal to $1$. Therefore, $\cH_+ \cap \cH_-$ is not contained in any of the singularities of $M^*(w_{1},  w_{\pi,n+1}^*\lambda)$ by Lemma~\ref{lem:w_1_action}. By the same argument, the map
    \begin{equation*}
        \lambda \mapsto \frac{\prod_{\Lambda_- \in L_-} \Lambda_-(\lambda)}{\prod_{i=1}^{m_1} \Lambda_-(1,i,d(1,i)-1)(\lambda)} M(w^*_{\pi,n+1},\lambda) \phi
    \end{equation*}
    is $M^*(w_{\pi,n+1}^*,\lambda) \phi$ when restricted to $\cH_+ \cap \cH_-$. Therefore, we see that 
    \begin{equation*}
        M^*(w_{1}  w_{\pi,n+1}^*,\lambda)\phi:= \left( \prod_{\Lambda_- \in L_-} \Lambda_-(\lambda) \right) M(w_1  w_{\pi,n+1}^*,\lambda) \phi
    \end{equation*}
    is meromorphic on $\cH_+$ and equal to $M^*(w_{1},  w_{\pi,n+1}^*\lambda)M^*(  w_{\pi,n+1}^*,\lambda) \phi$ when restricted to $\cH_+ \cap \cH_-$. 

    We now study the possible poles coming from $\cP^{P_{\res}}$. Set $Q_\pi=w_{1}  w_{\pi,n+1}^*. P_\pi=w_1 .P_\pi$. Let $Q$ be a Rankin--Selberg parabolic subgroup such that $Q_\pi \subset Q^\std$ and $Q \subset P_{\res}$, and let $w \in W^{P_{\res}}(Q_\pi,Q^\std)$. Note that we have described in \eqref{eq:w_1_action_1} and \eqref{eq:w_1_action_2} the standard Levi of $Q_\pi$, and it follows that $Q_\pi \cap \bfM_{\res}^2=Q \cap \bfM_{\res}^2=\bfM_{\res}^2$. If we set $\cQ_\pi=Q_\pi \cap \cM_{\res}$ and $\cQ=Q \cap \cM_{\res}$, then $w \in W^{\cM_{\res}}(\cQ_\pi,\cQ^{\std})$ and 
    \begin{equation}
            \label{eq:theta_non_zero}
        \hat{\theta}_Q^{P_{\res}}(w_Q^\std ww_{1}  w_{\pi,n+1}^* \lambda+\underline{\rho}_Q)=\hat{\theta}_{\cQ}^{\cM_{\res}}(w_{\cQ}^\std w w_1 w_{\pi,n+1}^* \lambda + \underline{\rho}_{\cQ}).
    \end{equation}
    However, the map $\lambda \in \cH_+ \cap \cH_- \mapsto (w_1 w_{\pi,n+1}^* \lambda)_{|\fa_{\cQ_\pi,\cc}} \in \fa_{\cQ_\pi,\cc}^*$ is surjective given the explicit description of \eqref{eq:a_pi_defi}. Lemma~\ref{lem:regular} concludes that \eqref{eq:theta_non_zero} is indeed non-zero for $\lambda \in \cH_+ \cap \cH_-$ in general position. 

    Finally, we claim that the map
    \begin{equation*}
        \lambda \mapsto E^{Q^\std}\left(M(w,w_{1}  w_{\pi,n+1}^* \lambda)  M^*(w_{1}  w_{\pi,n+1}^*,\lambda)\phi,w w_{1}  w_{\pi,n+1}^* \lambda\right)
    \end{equation*}
    is meromorphic on $\cH_+$ and that $\cH_+ \cap \cH_-$ is not contained in any of its singularities. Note that both $E^{Q^\std}$ and $M(w,w_1 w_{\pi,n+1}\lambda)$ live in $\cM_{\res}$ because of the conditions on $Q^\std$ and $w$ described above. But as $\lambda \in \cH_+ \cap \cH_- \mapsto (w_1 w_{\pi,n+1}^* \lambda)_{|\fa_{\cQ_\pi,\cc}} \in \fa_{\cQ_\pi,\cc}^*$ is surjective, we easily conclude by \cite[Theorem~2.3]{BL} that both $M(w,w_{1}  w_{\pi,n+1}^* \lambda)$ and $E^{Q^\std}(\cdot,w w_{1}  w_{\pi,n+1}^* \lambda)$ satisfy the claim. 

    By Corollary~\ref{cor:P_regular}, we have the equality of meromorphic functions on $\cH_+$
    \begin{equation*}
       \underset{L_+}{\Res} \; \cP(\phi,\lambda)=\frac{\cP^{P_{\res}}\left(M^*(w_{1}  w_{\pi,n+1}^*,\lambda)\phi,w_{1}  w_{\pi,n+1}^*\lambda,1\right)}{\prod_{\Lambda_- \in L_-}\Lambda_-(\lambda)}.
    \end{equation*}
    Moreover, $\cH_+ \cap \cH_-$ is not contained in any of the singularities of the numerator. We now conclude by Lemma~\ref{lem:residues_commute} that
     \begin{equation*}
        \underset{L_-}{\Res} \; \underset{L_+}{\Res} \; \cP(\phi,\lambda)= \cP^{P_{\res}}\left(M^*(w_{1}  w_{\pi,n+1}^*,\lambda)\phi,w_{1}  w_{\pi,n+1}^*\lambda\right).
    \end{equation*}
    \end{proof}    

    The proof of the following lemma is exactly the same.

    \begin{lem}
    \label{lem:second_residue_n}
        For $\lambda \in \cH_+$, define the meromorphic map 
        \begin{equation}
        \label{eq:regularized_L'}
            M^*(w_{2}, w_{\pi,n}^*\lambda)=\left(\prod_{i=1}^{m_2} \Lambda_-(2,i,d(2,i)-1)(\lambda)\right) M(w_{2}, w_{\pi,n}^*\lambda).
        \end{equation}
        Then we have the equality of meromorphic functions on $\cH_+ \cap \cH_-$
        \begin{equation*}
             \underset{L_-}{\Res} \; \underset{L_+}{\Res} \; \cP(\phi,\lambda)= \cP^{P_{\res}}\left(M^*(w_{2}, w_{\pi,n}^*\lambda)M^*( w_{\pi,n}^*,\lambda) \phi,w_{2} w_{\pi,n}^*\lambda\right),
        \end{equation*}
        where $M^*( w_{\pi,n}^*,\lambda) \phi$ is the regularized global intertwining operator (on the $\GL_{n}$ component) defined in \eqref{eq:reg_operator}.
    \end{lem}

    \begin{prop}
    \label{prop:residue_construction}
        The linear form
        \begin{equation*}
             \underset{L_+ \cup L_-}{\Res} \; \cP(\cdot,\lambda-\nu_\pi) : \cA_{P_\pi,\sigma_\pi,\lambda-\nu_\pi}(G) \to \cc,
        \end{equation*}
        well defined for $\lambda \in \fa_{\pi,\cc}^*$ in general position factors through the quotient
        \begin{equation*}
            \cA_{P_\pi,\sigma_\pi,\lambda-\nu_\pi}(G) \xrightarrow{E^{P,*}(\cdot,\lambda-\nu_\pi)} \cA_{P,\pi,\lambda}(G) 
        \end{equation*}
        and yields a $H(\bA)$-invariant linear form $ \cP_{\pi}(\cdot,\lambda) : \cA_{P,\pi,\lambda}(G) \to \cc$.
    \end{prop}

    \begin{proof}
        Recall that by \eqref{eq:intersection_hyperplanes} we have $\cH_+ \cap \cH_-=\fa_{\pi,\cc}^*-\nu_{\pi}$. The proposition follows from Corollary~\ref{cor:same_quotient},  Lemma~\ref{lem:second_residue_n+1} and Lemma~\ref{lem:second_residue_n}. The linear form $\cP_{\pi}(\cdot,\lambda)$ is $H(\bA)$-invariant as a residue of the $H(\bA)$-invariant linear form $\cP(\cdot,\lambda-\nu_\pi)$ (Theorem~\ref{thm:reg}).
    \end{proof}

    \subsubsection{Proof of Theorem~\ref{thm:non_tempered_periods}}

    We can now end the proof of Theorem~\ref{thm:non_tempered_periods}. Because $\cP(\phi,\lambda)=Z_{\sigma_\pi}(\phi,\lambda)$ by Proposition~\ref{prop:Zydor_RS}, the first part is a reformulation of Proposition~\ref{prop:residue_construction}. The second part is proved by computing the residues of $Z_{\sigma_\pi}(\phi,\lambda)$ from its Euler product expansion in \eqref{eq:global_facto_zeta}, as we know that the local factors don't contribute by the proof of Lemma~\ref{lem:iterated_residue}.

    \subsection{Functional equations of \texorpdfstring{$\cP_\pi$}{the period}}
    \label{sec:functional_equation}

    The regularized linear form $\cP_\pi(\cdot,\lambda)$ on $\cA_{P,\pi,\lambda}(G)$ a priori depends on the choice of inducing datum $(P,\pi)$. We now show that changing it results in a functional equation for $\cP_\pi$.

    Let $(P,\pi) \in \Pi_H$. Note that any $w \in W(P)$ can be identified with a pair $(w_n,w_{n+1}) \in \fS(m_1+m_2)^2$, where we recall that we allow blocks of size zero in the case $d(1,i)=1$ or $d(2,j)=1$. We define $W_\Delta(\pi)$ to be the set of $w \in W(P)$ which satisfy
    \begin{equation}
    \label{eq:W_Delta_defi}
        w_n=w_{n+1} \in \fS(m_1) \times \fS(m_2) \subset \fS(m_1+m_2).
    \end{equation}
    Therefore, with the definitions in \S\ref{subsubsec:relevant_inducing}, $W_{\Delta}(\pi)$ is exactly the set of $w_\Delta \in W(P)$ such that $(w_\Delta. P,w_\Delta \pi) \in \Pi_H$. Moreover, we have $w_\Delta\fa_{\pi,\cc}^*=\fa_{w_\Delta \pi,\cc}^*$. 

    The functional equations satisfied by $\cP_{\pi}$ are summarized in the following proposition.

    \begin{prop}
    \label{prop:independence_choice_couple}
        Let $w_\Delta \in W_{\Delta}(\pi)$. For $\varphi \in \cA_{P,\pi}(G)$ and $\lambda \in \fa_{\pi,\cc}^*$ in general position we have
         \begin{equation}
    \label{eq:independence_choice_couple}
        \cP_{\pi}(\varphi,\lambda)=\cP_{w_\Delta \pi}(M(w_\Delta,\lambda) \varphi,w_\Delta\lambda).
    \end{equation}
    \end{prop}

    \begin{proof}
       Assume that $\varphi=E^{P,*}(\phi,-\nu_\pi)$ for $\phi \in \cA_{P_\pi,\sigma_\pi}(G)$. By Lemma~\ref{lem:ME=EM} and \cite[Theorem~2.2]{Lap} we have for $\lambda$ in general position the equality $M(w_\Delta,\lambda)\varphi=E^{w_\Delta .P,*}(M(w_\Delta,\lambda-\nu_\pi)\phi,w_\Delta\lambda-\nu_{w_\Delta \pi})$. We have $E(M(w_\Delta ,\mu)\phi,w_\Delta \mu)=E(\phi,\mu)$ so that the proposition follows from the definition of $\cP_\pi$ and the fact that we can take the residues in any order by Lemma~\ref{lem:iterated_residue}.
    \end{proof}

    \subsection{The residue-free construction}
    \label{subsubsec:residue_free}
    We explain an alternative construction of $\cP_{\pi}$ without residues. The idea is to realize $\cA_{P,\pi}(G)$ as a subrepresentation of some parabolic induction rather than as a quotient. 

    Let $P_{+,\pi}$ be the standard parabolic subgroup of $G$ with standard Levi factor
    \begin{align*}
        M_{P_{+,\pi},n}&= \prod_{i=1}^{m_1} \left( \GL_{(d(1,i)-1)r(1,i)}\times \GL_{r(1,i)} \right)\prod_{i=1}^{m_2} \GL_{(d(2,i)-1)r(2,i)} ,\\
          M_{P_{+,\pi},n+1}&= \prod_{i=1}^{m_1}  \GL_{(d(1,i)-1)r(1,i)} \prod_{i=1}^{m_2} \left(\GL_{(d(2,i)-1)r(2,i)}\times \GL_{r(2,i)}\right).
    \end{align*}
    Then $P_\pi \subset P_{+,\pi} \subset P$. Recall that we have defined in \eqref{eq:w_cL_defi} an element $w_{1,n+1}$ and in \eqref{eq:w_cL_'_defi} an element $w_{2,n}$. Set $w_+=(w_{2,n},w_{1,n+1})$. Then $w_+ \in W(P_{+,\pi})$. Set $Q_{+,\pi}:=w_+ . P_{+,\pi}$. Then we have    
    \begin{align}
        M_{Q_{+,\pi},n}&= \prod_{i=1}^{m_1} \GL_{(d(1,i)-1)r(1,i)} \prod_{i=1}^{m_2} \GL_{(d(2,i)-1)r(2,i)}  \prod_{i=1}^{m_1} \GL_{r(1,i)} , \label{eq:M_R_defi}\\
          M_{Q_{+,\pi},n+1}&= \prod_{i=1}^{m_1}  \GL_{(d(1,i)-1)r(1,i)} \prod_{i=1}^{m_2} \GL_{(d(2,i)-1)r(2,i)} \prod_{i=1}^{m_2} \GL_{r(2,i)} . \label{eq:M_R_defi_+1}
    \end{align}
    In words, $w_+$ sends the $\GL_{r(1,i)}$ and $\GL_{r(2,i)}$ blocks at the bottom right corners, while preserving their order.
    
    Let $P_{+,n+1}^\std$ be the standard parabolic subgroup of $\GL_{n+1}$ with standard Levi  factor
    \begin{equation}
    \label{eq:P_+_Levi}
         M_{P^\std_{+,n+1}}=\prod_{i=1}^{m_1}  \GL_{(d(1,i)-1)r(1,i)} \times \prod_{i=1}^{m_2} \GL_{(d(2,i)-1)r(2,i)} \times \GL_{k+1} ,
    \end{equation}
    where we recall that $k+1=\sum_i r(2,i)$. Let $P_+$ be the Rankin--Selberg parabolic subgroup of $G$ corresponding to the pair $(P^\std_{+,n+1},m_1+m_2+1)$. It is standard. Note that $w_+\in {}_{P_+} W_P$ and that $P_\pi \subset P_{+,w_+}=P_{+,\pi}$. We also have $P_{\res} \subset P_+$.

    We now consider the regularized period $\lambda \in \fa_{P,\cc}^* \mapsto \cP^{P_+}(\varphi,\lambda,w_+)$ which is well-defined for $\lambda \in \fa_{P,\cc}^*$ in general position by Corollary~\ref{cor:P_regular}. 

    \begin{prop}
    \label{prop:alternative_construction}
        Let $\varphi \in \cA_{P,\pi}(G)$. Then $\lambda \mapsto \cP^{P_+}\left(\varphi,\lambda,w_+\right)$ is a well-defined meromorphic function on $\fa_{\pi,\cc}^*$. Moreover, for $\lambda$ in general position we have
        \begin{equation}
        \label{eq:two_defi_coincide}
            \cP^{P_+}\left(\varphi,\lambda,w_+\right)=\cP_{\pi}(\varphi,\lambda).
        \end{equation}
    \end{prop}

    \begin{proof}
        The proof is very similar to Proposition~\ref{prop:residue_construction}, so that we only sketch it. We start from $\varphi=E^{P,*}(\phi,-\nu_\pi)$. We compute $\Res_{L_+}\; \cP(\phi,\lambda)$ by taking the residues in the following order:
    \begin{equation*}
         \left( \xleftarrow[i=1]{m_2} \xleftarrow[j=d(2,i)-1]{1} \Lambda_+(2,i,j) \right)\leftarrow
         \left( \xleftarrow[i=1]{m_1} \xleftarrow[j=d(1,i)-1]{1} \Lambda_+(1,i,j) \right).
    \end{equation*}
    Recall that by Remark~\ref{rem:diagonal_Arthur} we know that the Arthur diagonal period computes the Petersson inner product. By the computation of the constant term of residual automorphic forms in Lemma~\ref{lem:contant_term}, it follows that \eqref{eq:two_defi_coincide} reduces to the adjunction between residues of Eisenstein series and constant terms (see \cite[Proposition~9.4.4.1]{BoiPhD}). We refer the reader to \cite[Proposition~12.2.4.1]{BoiPhD} where this computation is carried out.
    \end{proof}

\section{Split non-tempered Gan--Gross--Prasad conjectures}
\label{sec:ggp_conj_non_tempered}

We keep the notation from \S\ref{sec:IYZ_periods} and \S\ref{sec:RS_non_tempered}, so that $G=\GL_n \times \GL_{n+1}$ and $H=\GL_n$. 

In this section, we prove the global non-tempered GGP conjecture from Theorem~\ref{thm:GGP_global_intro}. By the Euler product factorization of Theorem~\ref{thm:non_tempered_periods}, it remains to study the local linear forms $Z_{\sigma_\pi,v}^\natural$. We prove in Theorem~\ref{thm:local_GGP_explicit} that they yield non-zero $H$-invariant linear forms on the local components of $\cA_{P,\pi}(G)$. As explained in \S\ref{subsubsec:residues_zeta_intro}, we will provide two proofs of this result. In this section, we show that it follows from the factorization property of the global regularized period $\cP_\pi$. We then end the proof of Theorem~\ref{thm:GGP_global_intro} in \S\ref{subsec:split_non_tempered_global}.

\subsection{The split non-tempered local Ichino--Ikeda conjecture}

We spell out the result that we need on $Z_{\sigma_\pi,v}^\natural$. Because the problem is purely local, we fix a place $v$ and henceforth drop it from the notation. In particular, $F$ is a local field of characteristic zero.

\subsubsection{Setting}
\label{subsubsec:setting_local_GGP}

Recall that we have defined in \S\ref{subsubsec:local_non_tempered_GGP} a notion of irreducible representation $\Pi$ of weak Arthur type for $G(F)$ with relevant parameter. Such $\Pi=\Pi_n \boxtimes \Pi_{n+1}$ decomposes as a parabolic induction
\begin{align}
    \Pi_n&=\bigtimes_{i=1}^{m_1}  \Speh(\delta_{1,i},d(1,i))_{\nu_{1,i}} \times \bigtimes_{i=1}^{m_2}  \Speh(\delta_{2,i}^\vee,d(2,i)-1)_{-\nu_{2,i}}, \label{eq:pi_n_local} \\
    \Pi_{n+1}&=\bigtimes_{i=1}^{m_1}  \Speh(\delta_{1,i}^\vee,d(1,i)-1)_{-\nu_{1,i}} \times \bigtimes_{i=1}^{m_2}  \Speh(\delta_{2,i},d(2,i))_{\nu_{2,i}}, \label{eq:pi_n+1_local}
\end{align}
where the $\delta_{1,i}$ and $\delta_{2,i}$ are irreducible square integrable representations of some $\GL_{r(1,i)}$ and $\GL_{r(2,i)}$ respectively, and the $\nu_{1,i}$ and $\nu_{2,i}$ are all real numbers of absolute value strictly less than $1/2$. 

For each $k \in \{1,2\}$ and index $1 \leq i \leq m_k$, set $\pi_{k,i}=\Speh(\delta_{k,i},d(k,i))$ and $\pi_{k,i}^-=\Speh(\delta_{k,i},d(k,i)-1)$. Define
    \begin{equation*}
        \pi=\left(\boxtimes_{i=1}^{m_1} \pi_{1,i} \boxtimes_{i=1}^{m_2} \pi_{2,i}^{-,\vee} \right) \boxtimes  \left(\boxtimes_{i=1}^{m_1} \pi_{1,i}^{-,\vee} \boxtimes_{i=1}^{m_2} \pi_{2,i} \right).
    \end{equation*}
    This is an unitary irreducible representation of $M_P(F)$ a standard Levi factor  of a standard parabolic subgroup $P$ of $G$. In particular, we have $\Pi=I_{P}^G \pi_{\nu_\Pi}$ for $\nu_\Pi \in \fa_P^*$ which accounts for the twist by the $\nu_{1,i}$ and $\nu_{2,j}$. Note that $(P,\pi)$ is of the same shape as the relevant inducing pairs in $\Pi_H$ that we considered in the global setting in \S\ref{subsubsec:relevant_inducing}. In particular, we may define the space $\fa_{\pi,\cc}^*$ as in \eqref{eq:a_pi_defi}. We have $\nu_\Pi \in \fa_{\pi}^*$.
    
    Let $P_{\pi}$ be the standard parabolic subgroup of $G$ such that 
    \begin{equation}
    \label{eq:delta_defi}
        \delta_{\pi}=\delta_{\pi,n} \boxtimes \delta_{\pi,n+1}:= \left(\boxtimes_{i=1}^{m_1} \delta_{1,i}^{\boxtimes d(1,i)} \boxtimes_{i=1}^{m_2} (\delta_{2,i}^\vee)^{\boxtimes (d(2,i)-1)} \right) \boxtimes \left(\boxtimes_{i=1}^{m_1} (\delta_{1,i}^\vee)^{\boxtimes (d(1,i)-1)} \boxtimes_{i=1}^{m_2} \delta_{2,i}^{\boxtimes d(2,i)} \right)
    \end{equation}
    is a square integrable representation of $M_{P_{\pi}}(F)$. We define $\nu_{\pi} \in \fa_{P_{\pi}}^*$ as in \S\ref{subsubsec:disc_gln} so that $\Pi$ is a quotient of $I_{P_{\pi}}^G \delta_{\pi,-\nu_{\pi}+\nu_\Pi}$. Note that this is not a standard module in general.

    \subsubsection{Normalized Zeta functional}
    
    By \cite[Theorem~9.7]{Zel}, the square integrable representation $\delta_\pi$ is generic. After choosing a Whittaker functional for $\delta_\pi$, for any $\phi \in I_{P_{\pi}}^G \delta_\pi$ and $\lambda \in \fa_{P_{\pi}}^*$, we can consider the Jacquet integral $W_{P_{\pi},\delta_\pi}^{\psi_0}(\phi,\lambda)$ and the Zeta integral $Z_{\delta_\pi}(\phi,\lambda)$ from \eqref{eq:local_Jacquet} and \eqref{eq:local_zeta_int}. We also have the local $L$-factors $b(\lambda,\delta_\pi)$ and $L(\lambda+1/2,\delta_{\pi,n} \times \delta_{\pi,n+1})$ from \eqref{eq:b_v} and \eqref{eq:local_RS}. We form the normalized Zeta integral $Z_{\delta_\pi}^\natural(\phi,\lambda)$ as in \eqref{eq:normalized_zeta_local}. The formula is 
    \begin{equation}
        \label{eq:normalized_zeta_local}
        Z^\natural_{\delta_\pi}(\phi,\lambda)=\frac{Z_{\delta_\pi}(\phi,\lambda)}{L(\lambda+1/2,\delta_{\pi,n} \times \delta_{\pi,n+1})} b(\lambda,{\delta_\pi}).
    \end{equation}
   By Lemma~\ref{lem:local_zeta_reg}, $\lambda \in \fa_{P_\pi,\cc}^* \mapsto Z_{\delta_\pi}^\natural(\phi,\lambda)$ is regular outside of the singularities of $b(\lambda,\delta_\pi)$, and by the proof of Lemma~\ref{lem:iterated_residue} we know that $\fa_{\pi,\cc}^*-\nu_{\pi}$ is not contained in this union of affine hyperplanes. We may therefore consider the restriction and define
    \begin{equation*}
        RZ_{\delta_\pi}^{\natural}(\phi,\lambda)=\left(\left.Z^\natural_{\delta_\pi}(\phi,\cdot) \right|_{\fa_{\pi,\cc}^*-\nu_\pi}\right)(\lambda), \quad \lambda \in \fa_{\pi,\cc}^*-\nu_\pi.
    \end{equation*}
    This is a meromorphic function on $\fa_{\pi,\cc}^*-\nu_{\pi}$.

    \subsubsection{Statement of the main theorem}

    The goal of \S\ref{sec:ggp_conj_non_tempered} is to prove that $RZ_{\delta_\pi}^\natural$ enjoys the following properties. 

    \begin{theorem}
        \label{thm:local_GGP_explicit}
        The following assertions hold.
        \begin{enumerate}
            \item Regularity: For every $\phi \in I_{P_{\pi}}^G \delta_\pi$, the meromorphic function $ RZ_{\delta_\pi}^{\natural}(\phi,\lambda)$ is regular at the point $-\nu_\pi+\nu_\Pi$.
            \item Non-vanishing: The map $RZ_{\delta_\pi}^{\natural}(\cdot,-\nu_\pi+\nu_\Pi)$ is non-zero and $H(F)$-invariant.
            \item Factorization:  The linear form $RZ_{\delta_\pi}^{\natural}(\cdot,-\nu_\pi+\nu_\Pi)$ factors through $I_{P_{\pi}}^G \delta_{\pi,-\nu_{\pi}+\nu_\Pi} \twoheadrightarrow \Pi$.
        \end{enumerate}
        \end{theorem}

    We will provide two proofs of Theorem~\ref{thm:local_GGP_explicit}. The first uses a local-global argument and the Euler product expression of $\cP_\pi$ given in Theorem~\ref{thm:non_tempered_periods}. It works equally well in the Archimedean and non-Archimedean cases, and is the content of this section. The second proof is local and specific to the non-Archimedean setting. It will be given in \S\ref{sec:local_zeta}.

    \subsubsection{Factorization is enough}

    It turns out that having condition 3 in Theorem~\ref{thm:local_GGP_explicit} \emph{generically} is enough to prove the other two points. To prove this, we first recall some properties of $Z^\natural_{\delta_\pi}$.

    \begin{lem}
        \label{lem:tempered_L_function}
        Let $\pi_1$ and $\pi_2$ be two irreducible representations of $\GL_{k_1}(F)$ and $\GL_{k_2}(F)$, with $\pi_1=I_{P_1}^{\GL_{k_1}} \delta_{1,\nu_1}$ and $\pi_2=I_{P_2}^{\GL_{k_2}} \delta_{2,\nu_2}$ for some square integrable representations $\delta_1$ and $\delta_2$ of $M_{P_1}(F)$ and $M_{P_2}(F)$ respectively, and let $\nu_1 \in \fa_{P_1}^*$ and $\nu_2 \in \fa_{P_2}^*$. Further decompose $\delta_{1}=\boxtimes_{i=1}^{m_1} \delta_{1,i}$ and $\delta_{2}=\boxtimes_{j=1}^{m_2} \delta_{2,j}$, where the $\delta_{1,i}$ and $\delta_{2,j}$ are representations of some general linear groups. Write $(\nu_{1,i})$ and $(\nu_{2,j})$ for the coordinates of $\nu_1$ and $\nu_2$. Then
        \begin{equation}
            \label{eq:tempered_L_function2}
            L(s,\pi_1 \times \pi_2)=\prod_{i=1}^{m_1} \prod_{j=1}^{m_2} L(s+\nu_{1,i}+\nu_{2,j},\delta_{1,i} \times \delta_{2,j}).
        \end{equation}
    \end{lem}

    \begin{proof}
        By \cite[Lemme~I.8]{MW89} we may assume that $\nu_1 \in \overline{\fa_{P_1}^{*,+}}$ and $\nu_2 \in \overline{\fa_{P_2}^{*,+}}$. Then we take $Q_1 \supset P_1$ and $Q_2 \supset P_2$ such that $\nu_1 \in \fa_{Q_1}^{*,+}$ and $\nu_2 \in \fa_{Q_2}^{*,+}$. Consider $\tau_1=I_{P_1}^{M_{Q_1}} \delta_1$ and $\tau_2=I_{P_2}^{M_{Q_2}} \delta_2$. They are irreducible tempered representation by \cite{Jac2}. In the non-Archimedean case, by \cite[Proposition~8.4]{JPSS83} we are reduced to proving \eqref{eq:tempered_L_function2} for $I_{Q_1}^{\GL_{k_1}}\tau_{1,\nu_1}$ and $I_{Q_2}^{\GL_{k_2}}\tau_{2,\nu_2}$. But this is now \cite[Proposition~9.1]{JPSS83}. In the Archimedean case, we use the analogous results from \cite{Jac09}. 
    \end{proof}

    Recall that we have defined a normalized intertwining operator $N_{\delta_\pi}$ in \S\ref{subsubsec:local_normalized_operator}. 

    \begin{lem}
        \label{lem:basic_Z}
        The following assertions hold.
        \begin{enumerate}
            \item Let $w \in W(P_\pi)$. For $\lambda \in \fa_{P_\pi,\cc}^*$ in general position, we have 
            \begin{equation*}
                Z_{\delta_\pi}^\natural(\phi,\lambda)=Z^\natural_{w \delta_\pi}(N_{\delta_\pi}(w,\lambda)\phi,w\lambda).
            \end{equation*}
            \item For any $\lambda$ outside of the singularities of $b(\lambda,\delta_\pi)$, the linear form $\phi \mapsto Z^\natural_{\delta_\pi}(\phi,\lambda)$ is non-zero.

        \end{enumerate}
    \end{lem}

    Note that for the first assertion to hold we need to pinpoint the measures and Whittaker functionals involved, otherwise it may only be true up to scalar. We will explain how to make adequate choices in \S\ref{subsec:local_measures}, and for now leave this issue unresolved to ease the exposition.

    \begin{proof}
        The first point is a consequence of the computation of $\gamma$-factors (see e.g. \cite[Section~11.1]{Kim}) (see also \S\ref{subsubsec:gamma} below where we recall the formulae).

        For the second, we first take $\lambda \in \fa_{P_\pi,\cc}^*$ such that $\Re(\lambda) \in \overline{\fa_{P_\pi}^{*,+}}$. Let $Q \supset P_\pi$ be a standard parabolic subgroup of $G$ such that $\Re(\lambda) \in \fa_{Q}^{*,+}$. Set $\tau=I_{P_\pi \cap M_Q}^{M_Q} \delta_\pi$. This is an irreducible tempered representation of $M_Q(F)$ by \cite{Jac2}. By the definition of Jacquet functionals and by the compatibility of $L$-factors with respect to induction from Lemma~\ref{lem:tempered_L_function}, we have $Z_{\delta_\pi}^\natural=Z_{\tau}^\natural$. The result now follows from \cite[Proposition~9.4]{JPSS83} in the non-Archimedean case, and \cite{Jac09} in the Archimedean case. If now $\lambda$ is any element outside of singularities of $b(\lambda,\delta_\pi)$ so that $\phi \mapsto Z^\natural_{\delta_\pi}(\phi,\lambda)$ is well-defined, we may choose $w \in W(P_\pi)$ such that $\Re(w \lambda) \in \overline{\fa_{w.P_\pi}^{*,+}}$, and then reduce to the previous case by the first point of the lemma. 
    \end{proof}

    We can now prove that "factorization is enough" in Theorem~\ref{thm:local_GGP_explicit}.

    \begin{prop}
        \label{prop:facto_is_enough}
        Assume that for $\lambda \in \fa_{\pi,\cc}^*$ in general position the linear form $RZ_{\delta_\pi}^{\natural}(\cdot,\lambda-\nu_\pi)$ factors through the quotient $I_{P_{\pi}}^G \delta_{\pi,\lambda-\nu_{\pi}} \twoheadrightarrow I_P^G \pi_\lambda$. Then Theorem~\ref{thm:local_GGP_explicit} holds.
    \end{prop}

\begin{proof}
    Let $w \in W(P_\pi)$ so that $w (-\nu_{\pi}+\nu_\Pi) \in \overline{\fa_{Q_{\pi}}^{*,+}}$, where $Q_{\pi}=w.P_{\pi}$. By Theorem~\ref{thm:N}, the normalized operator $N_{w \delta_\pi}(w^{-1},w\lambda)$ is regular at $-\nu_\pi+\nu_\Pi$ (as a meromorphic operator on $\fa_{P_\pi,\cc}^*$). We may therefore consider the composition
    \begin{equation*}
        I_{Q_\pi}^G w \delta_{\pi,-\nu_\pi+\nu_\Pi} \xrightarrow{N(w^{-1},w(-\nu_\pi+\nu_\Pi))} I_{P_\pi}^G \delta_{\pi,-\nu_\pi+\nu_\Pi} \xrightarrow{N(w_\pi^*,-\nu_\pi+\nu_\Pi)}  \Pi.
    \end{equation*}
    By \cite[Section~I.11]{MW89}, the last map realizes the quotient $I_{P_\pi}^G \delta_{\pi,-\nu_\pi+\nu_\Pi} \to \Pi$. We claim that the composition is also surjective. Because $\Pi$ is irreducible by \cite[Proposition~I.9]{MW89} (this is where we use the hypothesis that all the components of $\nu_\Pi$ are of absolute value strictly less that $1/2$), it is enough to show that it is non-zero. But by Theorem~\ref{thm:N}, it is $N_{w \delta}(w_\pi^*w^{-1},w(-\nu_\pi+\nu_\Pi))$ and by \cite[Equation~(I.1.2)]{MW89} it is indeed non-zero. 
    
    Let $\phi \in I_{P_\pi}^G \delta_\pi$. By the preceding discussion, in a neighborhood $\cU$ of $-\nu_\pi+\nu_\Pi$ in $\fa_{P,\cc}^*-\nu_\pi$  we may choose a local section of the holomorphic map $\lambda \mapsto N_{\delta_\pi}(w_\pi^*,\lambda) \phi$ by the operator $N_{w \delta_\pi}(w_\pi^*w^{-1},w\lambda)$. This means that there exists a holomorphic family $\lambda \in \cU \mapsto \phi(\lambda) \in I_{Q_\pi}^G w \delta_{\pi,\lambda}$ (where all these spaces are identified by restriction to $K$) such that for $\lambda \in \cU$ we have
    \begin{equation*}
        N_{\delta_\pi}(w_\pi^*,\lambda)\left(N_{w \delta_\pi}(w^{-1},w\lambda)\phi(\lambda)-\phi \right) =0.
    \end{equation*}
    By assumption, for $\lambda \in \fa_{\pi,\cc}^*-\nu_\pi$ in general position the linear map $\phi \mapsto Z_{\delta_\pi}^\natural(\phi,\lambda)$ factors through $N_{\delta_\pi}(w_\pi^*,\lambda)$. By Lemma~\ref{lem:basic_Z}, for $\lambda \in \cU \cap (\fa_{\pi,\cc}^*-\nu_\pi)$ in general position we have
    \begin{equation*}
        RZ_{w \delta_\pi}^\natural(\phi(\lambda),w\lambda)=RZ_{\delta_\pi}^\natural(\phi,\lambda).
    \end{equation*}
    By Lemma~\ref{lem:basic_Z}, the LHS is holomorphic at $-\nu_\pi+\nu_\Pi$, so that the RHS is as well. 
    
    That $\phi \mapsto RZ_{\delta}^\natural(\phi,-\nu_\pi+\nu_\Pi)$ factors through $\Pi$ follows from the fact that this holds generically and that for any $\mu \in \fa_{\pi,\cc}^*$ we have $N_{\delta_\pi}(w_{\pi}^*,\mu-\nu_\pi)\phi=(N_{\delta_\pi}(w_\pi^*,-\nu_\pi)\phi)_\mu$.
    
    Finally, it remains to prove that this linear map is non-zero. But this follows immediately from Lemma~\ref{lem:basic_Z} by lifting the problem to a standard module.
\end{proof}

\subsection{A local to global proof of the factorization property}

We prove that the factorization property required in Proposition~\ref{prop:facto_is_enough} holds.

\begin{prop}
    Let $F$ be a local field of characteristic zero. Let $\Pi$ be a representation of weak Arthur type of $G$, with relevant Arthur parameter. Then for $\lambda \in \fa_{\pi,\cc}^*$ in general position the linear form $RZ_{\delta_\pi}^{\natural}(\cdot,\lambda-\nu_\pi)$ factors through the quotient $I_{P_{\pi}}^G \delta_{\pi,\lambda-\nu_{\pi}} \twoheadrightarrow I_P^G \pi_\lambda$.
\end{prop}

\begin{proof}
    Recall that $\delta_\pi$ decomposes as a tensor product of $\delta_{1,i}$ and $\delta_{2,j}$ in \eqref{eq:delta_defi}, which are discrete series of some $\GL_{r(1,i)}(F)$ and $\GL_{r(2,j)}(F)$ respectively. By applying the globalization result of \cite[Theorem~1.B]{Clo}, we obtain a number field $K$, a place $v$ of $K$ such that $K_v =F$, and a family of cuspidal automorphic representations $\sigma_{1,i}$ and $\sigma_{2,i}$ whose components at $v$ are $\delta_{1,i}$ and $\delta_{2,j}$. Set
\begin{equation*}
    \sigma_{\pi}=\left(\boxtimes_{i=1}^{m_1} \sigma_{1,i}^{\boxtimes d(1,i)} \boxtimes_{j=1}^{m_2} (\sigma_{2,j}^\vee)^{\boxtimes (d(2,j)-1)}\right) \boxtimes \left(\boxtimes_{i=1}^{m_1} (\sigma_{1,i}^\vee)^{\boxtimes ( d(1,i)-1)} \boxtimes_{j=1}^{m_2} \sigma_{2,j}^{\boxtimes d(2,j)}\right),
\end{equation*}
so that $\sigma_{\pi} \in \Pi_\cusp(M_{P_{\pi}})$, and define $\pi_{\mathrm{aut}} \in \Pi_\disc(M_P)$ by 
\begin{align*}
    \pi_{\mathrm{aut},n}&=\boxtimes_{i=1}^{m_1} \Speh(\sigma_{1,i},d(1,i)) \boxtimes_{j=1}^{m_2} \Speh(\sigma_{2,j}^\vee,d(2,j)-1) \\
    \pi_{\mathrm{aut},n+1}&=\boxtimes_{i=1}^{m_1} \Speh(\sigma_{1,i}^\vee,d(1,i)-1) \boxtimes_{j=1}^{m_2} \Speh(\sigma_{2,j},d(2,j)).
\end{align*}
We now have $(P,\pi_{\mathrm{aut}}) \in \Pi_H$. Let $\phi=\otimes'_v \phi_v \in I_{P_\pi}^G \sigma_\pi$. By the Euler product expansion of Theorem~\ref{thm:non_tempered_periods}, there exists a finite set of places $\tS$ of $K$ (containing $v$) such that for $\lambda \in \fa_{\pi,\cc}^*$ in general position we have
\begin{equation*}
    \underset{L_-}{\Res} \;  \underset{L_+}{\Res} \; \cP(\phi,\lambda-\nu_\pi)=\varepsilon(\lambda)\cL(\lambda,\pi) \prod_{w \in \tS} RZ_{\sigma_{\pi,w}}^\natural(\phi_w,\lambda-\nu_\pi).
\end{equation*}
Here we recall that $\varepsilon(\lambda)\cL(\lambda,\pi)$ is a not identically zero meromorphic function on $\fa_{\pi,\cc}^*$. For $\lambda \in \fa_{\pi,\cc}^*$ in general position, we know by Lemma~\ref{lem:basic_Z} that the RHS is non-zero for a suitable choice of $\otimes_{w \in \tS} \phi_w$. Therefore, for any $\phi_v \in I_{P_\pi}^G \delta_\pi$ we obtain 
\begin{align*}
    &RZ^\natural_{\delta_\pi}(\phi_v,\lambda-\nu_\pi) =\underset{L_-}{\Res} \;  \underset{L_+}{\Res} \; \cP(\phi,\lambda-\nu_\pi)\varepsilon(\lambda)^{-1}\cL(\lambda,\pi)^{-1} \prod_{\substack{w \in \tS \\ w \neq v}} RZ_{\sigma_\pi,w}^\natural(\phi_w,\lambda-\nu_\pi)^{-1}
\end{align*}
But by Corollary~\ref{cor:same_quotient} and Proposition~\ref{prop:residue_construction}, we know that $\underset{L_-}{\Res} \;  \underset{L_+}{\Res} \; \cP(\cdot,\lambda-\nu_\pi)$ factors through $N_{\sigma_\pi}(w_\pi^*,\lambda-\nu_\pi)$ for $\lambda \in \fa_\pi^*$ in general position. Therefore, $\phi_v \mapsto RZ^\natural_{\delta_\pi}(\phi_v,\lambda-\nu_\pi)$ also factors through $N_{\delta_\pi}(w_\pi^*,\lambda-\nu_\pi)$ for such $\lambda$, i.e. through $I_{P_\pi}^G \delta_{\pi,\lambda-\nu_\pi} \to I_P^G \pi_\lambda$. This concludes.
\end{proof}

Using Proposition~\ref{prop:facto_is_enough}, we conclude that Theorem~\ref{thm:local_GGP_explicit} holds. 

\subsection{The split non-tempered global Gan--Gross--Prasad conjecture}
\label{subsec:split_non_tempered_global}

We now go back to the global setting to prove Theorem~\ref{thm:GGP_global_intro}. 

Let $(P,\pi) \in \Pi_H$ be a relevant inducing pair. By the Euler product expansion of Theorem~\ref{thm:non_tempered_periods} and by the first point of Theorem~\ref{thm:local_GGP_explicit}, we know that the only singularities of $\cP_\pi(\cdot,\lambda)$ at $\lambda=0$ come from the numerator in the quotient $\cL(\lambda,\pi)$, and all they all lie along affine hyperplanes. We denote by $\cP_\pi^*$ and $\cL^*(\pi)$ the regularized values of $\cP_\pi(\cdot,\lambda)$ and $\cL(\lambda,\pi)$ at $0$, defined by multiplying by the corresponding product of affine linear forms and evaluating. In particular, $\cP_\pi^*$ still defines a $H(\bA)$-invariant linear form on $\cA_{P,\pi}(G)$. The lemma that remains to prove is the following.

    \begin{lem}
        \label{lem:local_non_vanishing}
        Let $v$ be a place of $F$. The map $RZ_{\sigma_\pi,v}^\sharp(\cdot,\lambda)$ is regular at $\lambda=-\nu_\pi$ and the linear form $\phi_v \mapsto RZ_{\sigma_\pi,v}^\sharp(\phi_v,-\nu_\pi)$ is non-zero.
    \end{lem}

    \begin{proof}
        The representation $\sigma_{\pi,v}$ is unitary and generic. By \cite{JS} and \cite[Section~I.11]{MW89}, there exist a standard parabolic subgroup $P_{\delta}$ of $G$, a square integrable representation $\delta_v$ of $M_{P_{\delta}}(F_v)$ and an element $\nu_{v} \in \fa_{P_{\delta}}^*$ whose coordinates are all of absolute value strictly less than $1/2$ such that $\sigma_{\pi,v}=I_{P_{\delta} \cap M_{P_\pi}}^{M_{P_\pi}} \delta_{v,\nu_v}$. By Lemma~\ref{lem:tempered_L_function} we have $Z_{\sigma_\pi,v}^\sharp(\cdot,\lambda)=Z_{\delta_v}^\sharp(\cdot,\lambda+\nu_v)$. By the assumption that $(P,\pi) \in \Pi_H$, we see that there exists an element $w \in W(P_{\delta})$ such that $w \delta$ is of the form prescribed in \eqref{eq:delta_defi}. If we write $\Pi_v$ for the induction of the Speh representations built from $w \delta_v$ at $-\nu_\pi+\nu_v$ like in \eqref{eq:pi_n_local} and \eqref{eq:pi_n+1_local}, then $\Pi_v$ is a relevant representation of weak Arthur type. Moreover, because of the bound on $\nu_v$, by Theorem~\ref{thm:N} we can choose $w$ so that the normalized intertwining operator $N_{\delta_v}(w,\lambda)$ is regular at $-\nu_\pi+\nu_v$. It remains to use the functional equation of Zeta integrals from Lemma~\ref{lem:basic_Z} and the non-vanishing and regularity statements of Theorem~\ref{thm:local_GGP_explicit}.
    \end{proof}

\begin{theorem}
    The linear form $\cP_\pi^*$ is non-zero if and only if $\cL^*(\pi) \neq 0$.
\end{theorem}

\begin{proof}
    This is a direct consequence of Theorem~\ref{thm:non_tempered_periods} and Lemma~\ref{lem:local_non_vanishing}.
\end{proof}

\section{Residues of local Zeta integrals}
\label{sec:local_zeta}

In this section, we provide an alternative purely local proof for the explicit non-tempered GGP conjecture from Theorem~\ref{thm:local_GGP_explicit} in the case where $F$ is a $p$-adic field. We keep the notation from \S\ref{sec:IYZ_periods}, so that $G=\GL_n \times \GL_{n+1}$ and $H=\GL_n$.

    \subsection{Preliminaries on local factors}
    \label{sec:prelim_local}

    We recall some results on local intertwining operators and several normalizing factors. Let $q$ be the residual characteristic of $F$. Let $k$ be a positive integer. We fix $P$ a standard parabolic subgroup of $\GL_k$. Let $\psi$ be a non-zero unitary character of $F$.

    \subsubsection{Measures}
    \label{subsec:local_measures}
   If $N$ is the unipotent radical of a parabolic subgroup of $\GL_k$, we have an isomorphism of varieties $N \simeq \bA^m$ for some $m$. We equip $N(F)$ with the measure $\prod_{i=1}^m d_{\psi} x_i$, where $d_{\psi} x_i$ is the $\psi$-autodual measure on $F$. 
    
   For subgroups of $H(F)$, it is convenient to choose measures differently. Let $K_H=H(\oo_F)$ be the standard maximal compact subgroup of $H$, equipped with the probability measure. We equip the $F$-points of the maximal torus $T_{0,H}=T_0 \cap H$ and of the unipotent group $N_{0,H}=N_0 \cap H$ with the Haar measures such that $\vol(K_H \cap T_{0,H}(F))=\vol(K_H \cap N_{0,H}(F))=1$. This determines a Haar measure on $H(F)$ by the Iwasawa decomposition, and we use it to define the Zeta integrals $Z_\tau$ from \eqref{eq:local_zeta_int}.

    \subsubsection{Intertwining operators}
    \label{subsec:local_operator}

    Let $\tau$ be a smooth irreducible unitary representation of $M_P(F)$. Let $Q$ be another standard parabolic subgroup of $\GL_k$ and let $w \in W(P,Q)$. Recall that in \S\ref{subsubsec:blocks} we have chosen a representative $\dot{w} \in G(F)$ of $w$. For $\phi \in I_P^{\GL_k} \tau$ and $\lambda \in \fa_{P,\cc}^*$, we define an unnormalized intertwining operator by
    \begin{equation}
        \label{eq:J_tau_defi}
        J_\tau(w,\lambda) \phi(g)=\int_{(N_Q \cap \dot{w} N_P \dot{w}^{-1})(F) \backslash N_Q(F)} \phi_\lambda(\dot{w}^{-1} ng)dn, \quad g \in \GL_k(F).
    \end{equation}
    Here, $\phi_\lambda$ is the only element in $I_P^{\GL_k} \tau_\lambda$ whose restriction to $\GL_k(\oo_F)$ coincides with the restriction of $\phi$. This integral is absolutely convergent for $\lambda$ in some positive cone, and admits a meromorphic continuation to $\fa_{P,\cc}^*$ by \cite{Sha81}. For $\lambda$ in general position, we obtain $J_\tau(w,\lambda) : I_P^{\GL_k} \tau_\lambda \to I_Q^{\GL_k} w\tau_\lambda$. 

    We have already met these operators in a disguised way in \S\ref{subsubsec:local_normalized_operator}. Indeed, write $M_P=\GL_{k_1} \times \hdots \times \GL_{k_m}$ and $\tau=\tau_1 \boxtimes \hdots \boxtimes \tau_m$ accordingly. If $\alpha \in \Sigma_P$ is the positive root switching $\GL_{k_i}$ and $\GL_{k_j}$ with $1 \leq i < j \leq m$, set $\tau_\alpha=\tau_i \times \tau_j^\vee$. Write
    \begin{equation}
        \label{eq:L_formula}
        n_{\tau}(\alpha,s)=\frac{L(s,\tau_\alpha)}{L(1+s,\tau_\alpha) \varepsilon(s,\tau_\alpha,\psi)}.
    \end{equation}
    Here, the $L$ and $\varepsilon$ factors are those defined in \cite{JPSS83} (see also \cite{Sha83}). In particular, the latter is monomial in $q^{-s}$ and is regular and never-vanishing. Then we set for $\lambda \in \fa_{P,\cc}^*$ in general position
    \begin{equation}
        \label{eq:n_formula}
        n_\tau(w,\lambda)=\prod_{\substack{\alpha \in \Sigma_{P} \\ w \alpha <0}} n_\tau(\alpha,\langle  \lambda, \alpha^\vee \rangle).
    \end{equation}
    The normalized intertwining operator $N_\tau(w,\lambda)$ of \S\ref{subsubsec:local_normalized_operator} is defined by
    \begin{equation}
        \label{eq:N_n_norm}
        N_\tau(w,\lambda)=n_\tau(w,\lambda)^{-1} J_\tau(w,\lambda).
    \end{equation}
    The choice of representatives $\dot{w}$ ensures that $N(w_1 w_2,\lambda)=N(w_1,w_2 \lambda)N(w_2,\lambda)$ by \cite{KS}.

    \subsubsection{$\gamma$-factors}
    \label{subsubsec:gamma}
    We now assume that $\tau$ is generic. Let $\psi_k$ be a generic character of $N_k(F)$ the unipotent radical of the standard Borel of $\GL_k$. Recall that $w_P$ is the longest element in $W(P)$. Let $W_\tau^{w_P \psi_k}$ be a non-zero Whittaker functional for $\tau$ with respect fo $w_P \psi_k$. We then have the Jacquet functional $W_{P,\tau}^{\psi_k}(\cdot,\lambda)$ on $I_P^{\GL_k} \tau_\lambda$ from \eqref{eq:local_Jacquet}. 
    
    Let $w \in W(P,Q)$. Starting from the Whittaker functional $W_{w \tau}^{w w_P \psi_k}$ on $w \tau$ obtained by conjugating with $\dot{w}$, we obtain $W_{Q,w\tau}^{\psi_k}$. By uniqueness of Whittaker models we get a functional equation 
    \begin{equation}
        \label{eq:gamma_factor}
        W_{Q,w\tau}^{\psi_k}(J_\tau(w,\lambda)\phi,w\lambda)=\gamma_{\tau}(w,\lambda)^{-1} W_{P,\tau}^{\psi_k}(\phi,\lambda).
    \end{equation}
    The $\gamma$ factor is a meromorphic function in $\lambda$ which can be described. For any $\alpha \in \Sigma_P$, set 
    \begin{equation}
        \label{eq:gamma_formula}
        \gamma_{\tau}(\alpha,s)=\varepsilon(s,\tau_\alpha,\psi) \frac{L(1-s,\tau_\alpha^\vee)}{L(s,\tau_\alpha)}.
    \end{equation} 
    Then by \cite[Section~11.1]{Kim} the formula is
    \begin{equation}
        \label{eq:gamma_formula_prod}
        \gamma_{\tau}(w,\lambda)=\prod_{\substack{\alpha \in \Sigma_{P} \\ w \alpha <0}} \gamma_\tau(\alpha,\langle \lambda, \alpha^\vee \rangle).
    \end{equation}

    \subsubsection{Bernstein--Zelevinsky classification} \label{subsubsec:BZ_class}

    We recall the main results of the Bernstein--Zelevinsky classification for square integrable representations of $\GL_k(F)$. 
    
    Let $\delta$ be an unitary irreducible square integrable representation of $\GL_k(F)$. By \cite[Theorem~9.3]{Zel}, there exist some integers $d,r \geq 1$ and $\sigma$ an unitary irreducible supercuspidal representation of $\GL_{r}(F)$ such that $\delta$ is the unique irreducible quotient of the induction $\sigma_{-\frac{d-1}{2}} \times \hdots \times \sigma_{\frac{d-1}{2}}$. This is a generalized Steinberg representation and we write $\delta=\St(\sigma,d)$. This quotient can also be realized in the following way. Denote by $P_{\delta}$ the standard parabolic subgroup of $\GL_k$ with standard Levi factor $\GL_r^d$. Let $w_{\delta}^*$ be the longest element in $W(P_{\delta})$, set $\sigma_\delta=\sigma^{\boxtimes d}$ which is a supercuspidal representation of $M_{P_{\delta}}(F)$. We set $\nu_\delta=-\rho_{P_\delta}/r$. By \cite[Lemme~VII.3.2]{Renard} the operator $J_{\sigma_\delta}(w_{\delta}^*,\nu_{\delta})$ is well-defined and by the same argument as in \cite[Théorème~VII.4.2~(i)]{Renard}, its image is an irreducible subrepresentation of $I_{P_{\delta}}^{\GL_k} \sigma_{\delta,-\nu_{\delta}}$. Therefore, it has to be $\delta$.

    We now assume that $\delta$ is a square integrable representation of $M_P(F)$. We decompose $\delta=\delta_1 \boxtimes \hdots \boxtimes \delta_m$. Set $P_{\delta}=(\prod P_{\delta_i}) N_P$ (a standard parabolic subgroup of $G$), $\sigma_{\delta}=\boxtimes \sigma_{\delta_i}$, $w_\delta^*=\prod w_{\delta_i}^*$. Define $\nu_\delta$ as in \eqref{eq:nu_exp_defi}. By exactness of induction, $I_P^{\GL_k} \delta$ is equal to the image of $J_{\sigma_\delta}(w_\delta^*,\nu_\delta)$.

    If now $\pi=\Speh(\delta,d)$ is a Speh representation of $\GL_{kd}(F)$, we can define $P_\pi$, $w_\pi^*$ and $\nu_\pi$ as in \S\ref{subsubsec:residual_blocks}, so that $\pi$ is the image of $J_{\delta}(w_\pi^*,-\nu_\pi)$. We naturally generalize this notation to inductions of Speh's as above.
    
    \subsubsection{Formulae for local factors}
    \label{subsubsec:local_formulae}
    Denote by $\zeta$ the local zeta function $\zeta(s)=(1-q^{-s})^{-1}$. Let $\sigma_{1}$ and $\sigma_2$ be irreducible supercuspidal representations of $\GL_{k_1}(F)$ and $\GL_{k_2}(F)$ respectively. By \cite[Proposition~8.1]{JPSS83}, we have
    \begin{equation}
        \label{eq:cuspi_L}
        L(s,\sigma_1 \times \sigma_2)=\prod_{\substack{q^{-t} \in \cc \\ \sigma_{1,t} \simeq \sigma_2^\vee}}\zeta(s-t),
    \end{equation}
    where we recall that $\sigma_{1,t}=\sigma_{1} \otimes \Val{\det}^{t}$ ($\Val{\cdot}$ being the normalized absolute value on $F$). 

    Now, let $\delta_1$ and $\delta_2$ be irreducible square integrable representations of $\GL_{k_1}(F)$ and $\GL_{k_2}(F)$. Write $\delta_1=\St(\sigma_1,d_1)$ and $\delta_2=\St(\sigma_2,d_2)$. Assume that $d_2 \leq d_1$. By \cite[Theorem~8.3]{JPSS83} we have
    \begin{equation}
        \label{eq:L2L}
        L(s,\delta_1 \times \delta_2)=\prod_{i=1}^{d_2} L\left(s+\frac{d_1-1}{2}+\frac{d_2-2i+1}{2},\sigma_1 \times \sigma_2\right).
    \end{equation}

    \subsubsection{Compatibility of $\gamma$ factors with respect to induction} Let $\delta$ be an irreducible square integrable representation of $M_P(F)$. Denote by $W_{P_\delta,\sigma_\delta}^{M_P,\psi_k}$ the partial Jacquet functional built on $I_{P_\delta}^{M_P} \sigma_\delta$. 

    \begin{lem}
        \label{lem:induce_whittaker}
        The restriction of $W_{P_\delta,\sigma_\delta}^{M_P,\psi_k}(-\nu_\delta)$ to $\delta$ is non-zero and defines a Whittaker functional on this representation.
    \end{lem}

    \begin{proof}
        We can assume that $P=\GL_k$. By \eqref{eq:gamma_factor}, we have for any $\phi \in I_{P_\delta}^{\GL_k} \sigma_\delta$
        \begin{equation*}
            W_{P_\delta, \sigma_\delta}^{\psi_k}(J_{\sigma_\delta}(w_\delta^*,\nu_\delta)\phi,-\nu_\delta)=\gamma_{\sigma_\delta}(w_\delta^*,\nu_\delta)^{-1} W_{P,\sigma_\delta}^{\psi_k}(\phi,\nu_\delta).
        \end{equation*}
        By \eqref{eq:gamma_formula} and \eqref{eq:cuspi_L}, $\gamma_{\sigma_\delta}(w_\delta^*,\nu_\delta)^{-1}$ is well-defined and non-zero. By Rodier's theorem (see \cite[Theorem~1.6]{CS}), there exists $\phi$ such that $W_{P,\sigma_\delta}^{\psi_k}(\phi,\nu_\delta) \neq 0$. This concludes the proof. 
    \end{proof}

    Using Lemma~\ref{lem:induce_whittaker} we obtain the following compatibility of $\gamma$-factors. 

    \begin{lem}
        \label{lem:inductive_gamma}
        Let $Q$ be a standard parabolic subgroup of $\GL_k$ and let $w \in W(P,Q)$. Then for $\lambda \in \fa_{P,\cc}^*$ in general position we have the equality
        \begin{equation*}
            \gamma_\delta(w,\lambda)=\gamma_{\sigma_{\delta}}(w,\lambda-\nu_\delta).
        \end{equation*}
    \end{lem}

    \begin{proof}
        If we take $W_{P_\delta,\sigma_\delta}^{M_P,w_P \psi_k}(-\nu_\delta)$ as a Whittaker functional on $\delta$, which is possible by Lemma~\ref{lem:induce_whittaker}, we see that for every $\phi \in I_P^{\GL_k} \delta$ (identified with a submodule of $I_{P_\delta}^{\GL_k} \sigma_{\delta,-\nu_\delta}$) we have $W_{P,\delta}^{\psi_k}(\phi,\lambda)=W_{P,\sigma_\delta}^{\psi_k}(\phi,\lambda-\nu_\delta)$. But given the definition in \eqref{eq:J_tau_defi} we also have $J_\tau(w,\lambda)\phi=J_{\sigma_\tau}(w,\lambda-\nu_\delta)\phi$. We conclude by noting that $w\sigma_{\delta}=\sigma_{w \delta}$ and by also applying this discussion to $w \delta$.
    \end{proof}

    \subsection{Asymptotics of Whittaker functionals}
    \label{sec:asymptotic_whitt}
    Let $k \geq 1$. Let $P$ be a standard parabolic subgroup of $\GL_k$. For simplicity, we will simply write $\psi$ instead of $\psi_k$, so that all Whittaker functionals are denoted by $W^\psi$.

    \subsubsection{Positive cones in $A_P$}

    For any standard parabolic subgroup $Q$ of $\GL_k$, we write $A_Q$ for the split center of $M_Q$. For any $\lambda \in \fa_{Q,\cc}^*$ we write $\lambda$ for the character $t \in A_Q(F) \mapsto \exp( \langle \lambda, H_Q(t) \rangle)$. If $w \in W(Q)$, we denote by $w\lambda$ the character $\lambda(w^{-1} \cdot w)$ of $A_{w.Q}(F)$. 
    
    For every $\varepsilon >0$, define
        \begin{equation*}
            A_Q[\leq \varepsilon]=\{ a \in A_Q(F) \; \big| \; \forall \alpha \in \Delta_Q, \; \Val{\alpha(a)} \leq \varepsilon \}.
        \end{equation*}
    Our goal is to describe the behavior of the Jacquet functional $W_{P,\delta}^\psi$ on these $A_Q[\leq \varepsilon]$.

    \subsubsection{The supercuspidal case}

    We begin with $\sigma$ an irreducible supercuspidal representation of $M_P(F)$. We denote by $\chi_\sigma$ its central character. We fix $W_\sigma^\psi$ a Whittaker functional on $\sigma$. Let $Q$ be a standard parabolic subgroup of $G$ containing $P$. For simplicity, we will write $I_P^{M_Q}$ for $I_{P \cap M_Q}^{M_Q}$. Then we have the partial induction $I_{P}^{M_Q} \sigma$ and the partial Jacquet functional $W_{P,\sigma}^{M_Q,\psi}$ defined on it. If $\phi \in I_P^{\GL_k} \sigma$, we denote again by $W_{P,\sigma}^{M_Q,\psi}(\phi,\lambda)$ the image of $\phi_{|M_Q(F)} \in \delta_Q^{1/2} \otimes I_{P}^{M_Q} \sigma$ by $W_{P,\sigma}^{M_Q,\psi}(\cdot,\lambda)$. 
    
    Let $w_Q$ be the longest element in $W(Q)$. We set $Q'=w_Q . Q$. We also have $w_Q \in W(P)$ and we set $R=w_Q .P$. Then we can also consider $w_R \in W(R)$ and set $R'=w_R .R$. Note that $w_R=w_P^Q w_Q^{-1}$ where $w_P^Q$ is the longest element in $W(P\cap M_Q)$. 

    \begin{lem}
        \label{lem:non_zero}
        For $\lambda \in \fa_{P,\cc}^*$ in general position, there exists $\phi \in I_{R}^{\GL_k} w_Q \sigma$ and $\varepsilon >0$ such that for any $a \in A_Q[\leq \varepsilon]$ we have 
        \begin{equation}
            \label{eq:partial_whitt}
            W_{R,w_Q \sigma}^{\psi}(a,\phi,w_Q \lambda)=\delta_Q^{\frac{1}{2}}(a)(\lambda \chi_{\sigma})(a) W_{P,\sigma}^{M_Q,\psi}(J_{w_Q \sigma}(w_Q^{-1},w_Q \lambda) \phi,\lambda),
        \end{equation}
        and moreover such that both sides of this equality are non-zero.
    \end{lem}

    \begin{proof}
        We set $\Sigma_\lambda=I_P^{M_Q} \sigma_\lambda$. If $\lambda=0$, we drop the subscript. We have an isomorphism $\phi \in I_R^{\GL_k} w_Q \sigma \mapsto \Phi(\phi) \in I_{Q'}^{\GL_k} w_Q\Sigma$. We denote by $W_{Q',w_Q\Sigma_\Lambda}^{\psi}$ the Jacquet integral on $I_{Q'}^{\GL_k} w_Q \Sigma_\lambda$ built from the Whittaker functional $W_{R,w_Q \sigma}^{M_{Q'},w_Q^{-1}\psi}(\cdot,w_Q \lambda)$ on $w_Q \Sigma_\lambda$. For any $\phi \in I_R^{\GL_k} w_Q \sigma$ with support in $R(F) w_R^{-1} N_{R'}(F)$, we see that 
        \begin{equation}
            \label{eq:induced_whittaker}
            W_{R,w_Q \sigma}^{\psi}(\phi,w_Q \lambda)=W_{Q',w_Q\Sigma_\lambda}^{\psi}(\Phi(\phi),0).
        \end{equation}
        By \cite[Theorem~1.6]{CS}, this implies that \eqref{eq:induced_whittaker} holds for any $\phi \in I_R^{\GL_k} w_Q \sigma$. 

        Let $\Phi=\Phi(\phi) \in I_{Q'}^{\GL_k} w_Q \Sigma$ with support in $Q'(F) w_Q N_Q(F)$. Then using \eqref{eq:induced_whittaker} we obtain for $a \in A_Q(F)$
         \begin{align}
            W_{R,w_Q \sigma}^{\psi}(a,\phi,w_Q \lambda)&=\int_{N_{Q}(F)} W_{R,w_Q \sigma}^{M_{Q'},w_Q^{-1}\psi}(\Phi(w_Qn a),w_Q \lambda) \overline{\psi(n)}dn  \nonumber   \\
            &= \delta_Q^{\frac{1}{2}}(a) (\lambda \chi_{\sigma})(a)\int_{N_Q(F)}  W_{R,w_Q \sigma}^{M_{Q'},w_Q^{-1}\psi}(\Phi(w_Qn),w_Q \lambda)\overline{\psi(a^{-1}na)}dn \label{eq:whitt_J}
        \end{align}
        Because $\Phi$ has compact support modulo $Q'(F)$, the integral takes place over a compact set $\omega_Q$. But there exists $\varepsilon$ such that $a \in A_Q[\leq \varepsilon]$ implies $\overline{\psi(a^{-1}na)}=1$ for every $n \in \omega_Q$. By induction by stages for intertwining operators (\cite[Proposition~VII.3.5~(ii)]{Renard}) and the integral formula for $J$ from \eqref{eq:J_tau_defi}, we see that \eqref{eq:whitt_J} reduces to \eqref{eq:partial_whitt} (note that $w_Q^{-1} \psi$ is replaced by $\psi$ by conjugation).

        It remains to show that we can choose $\phi$ such that \eqref{eq:partial_whitt} is non-zero. We may take $\lambda$ in general position so that $\Sigma_\lambda$ is irreducible. Denote by $J_{w_Q \Sigma}(w_Q^{-1},w_Q \lambda)$ the intertwining operator $I_{Q'}^{\GL_k} w_Q \Sigma_\lambda \to I_{Q}^{\GL_k} \Sigma_\lambda$ obtained by induction by stages. Let $\varphi \in I_P^{M_Q} \sigma$ such that $W_{P,\sigma}^{M_Q,\psi}(\varphi,\lambda) \neq 0$. By \cite[Proposition~VII.3.4]{Renard}, there exists $\Phi \in I_{Q'}^{\GL_k}w_Q \Sigma_\lambda$ with support in $Q'(F)w_Q N_Q(F)$ such that we have $J_{w_Q \Sigma}(w_Q^{-1},w_Q \lambda)(\Phi)(1) \neq 0$. Because this condition on the support is stable under the action by right translation of $M_Q(F)$, there exists such $\Phi$ which satisfies $J_{w_Q \Sigma}(w_Q^{-1},w_Q \lambda)(\Phi)(1) =\varphi$. If $\Phi=\Phi(\phi)$, this yields $W_{P,\sigma}^{M_Q,\psi}(J_{w_Q \sigma}(w_Q^{-1},w_Q \lambda) \phi)\neq 0$, which concludes.
    \end{proof}
    
    We can now state our result on the asymptotics of Whittaker functionals induced from supercuspidal representations. It is a generalization of a formula used in the proof of \cite[Theorem~3.4]{CasSha}, and was inspired by \cite{CasLetter}.

    \begin{prop}
        \label{prop:Whittaker_explicit}
            Let $Q$ be any standard parabolic subgroup of $\GL_k$. For every $\phi \in I_P^{\GL_k} \sigma$ there exists $\varepsilon >0$ such that for $\lambda \in \fa_{P,\cc}^*$ in general position and $a \in A_Q[\leq \varepsilon]$ we have
            \begin{equation}
            \label{eq:asymptotic_Whittaker}
                W_{P,\sigma}^{\psi}(a,\phi,\lambda)=\sum_{w \in W(P;Q)} c_\sigma^Q(w,\lambda) w(\lambda\chi_\sigma)(a) \delta_{Q}^\frac{1}{2}(a) W_{Q_w,w\sigma}^{M_Q,\psi} (N_\sigma(w,\lambda)\phi,w \lambda),
                \end{equation}
                where
                \begin{equation}
                \label{eq:c_w^Q_formula}
                    c_\sigma^Q(w,\lambda)=n_\sigma(w,\lambda)\gamma_\sigma(w,\lambda) \gamma_{w\sigma^\vee}(w_Q,-w\lambda)^{-1}.
                \end{equation}
        \end{prop}
        
        \begin{proof}
            Let $(I_P^{\GL_k} \sigma_\lambda)_{N_Q}$ be the (unnormalized) Jacquet module of $I_P^{\GL_k} \sigma_\lambda$ with respect to $Q$. This is a representation of $M_Q(F)$. Let $(I_P^{\GL_k} \sigma_\lambda)_{\psi,N_k}$ be the twisted Jacquet module of $(I_P^{\GL_k} \sigma_\lambda) \otimes \psi^{-1}$ with respect to the minimal parabolic $B_k$ of $\GL_k$. The Jacquet functional factors through $(I_P^{\GL_k} \sigma_\lambda)_{\psi,N_k}$. By \cite[Proposition~6.4]{CS} there exists a map of $\cc$-vector spaces $s_Q : (I_P^{\GL_k} \sigma_\lambda)_{N_Q} \to (I_P^{\GL_k} \sigma_\lambda)_{\psi,N_k}$, and by \cite[Lemma~3.6]{CasSha} for every $\phi \in I_P^{\GL_k} \sigma$ there exists $\varepsilon >0$ such that for all $a \in A_Q[\leq \varepsilon]$ we have
        \begin{equation}
            \label{eq:asymptotic_whitt}
            W_{P,\sigma}^{\psi}(a,\phi,\lambda)=W_{P,\sigma}^\psi \left( s_Q \left( I_{\lambda,N_Q}(a) \phi_{\lambda,N_Q} \right),\lambda\right),
        \end{equation}
        where $\phi_{\lambda,N_Q}$ is the image of $\phi_\lambda$ in $(I_P^{\GL_k} \sigma_\lambda)_{N_Q}$ and $I_{\lambda,N_Q}$ is the action of $M_Q(F)$ on this representation. It is easily seen by \cite[Section~4.1]{Cas} that $\varepsilon$ can be chosen independently from $\lambda$. 
        
        More precisely, the application $s_Q$ is the composition of the canonical lifting $(I_P^{\GL_k} \sigma_\lambda)_{N_Q} \to I_P^{\GL_k} \sigma_\lambda$ from \cite[Section~4.1]{Cas} with the projection to $(I_P^{\GL_k} \sigma_\lambda)_{\psi,N_k}$. It follows that the map $W_{P,\sigma}^\psi \circ s_Q$ is a Whittaker functional on $(I_P^{\GL_k} \sigma_\lambda)_{N_Q}$. For $\lambda$ in general position we have a morphism
          \begin{equation}
          \label{eq:explicit_Jacquet}
                \phi \in I_P^{\GL_k} \sigma_\lambda \mapsto \sum_{w \in W(P;Q)} \left(N_\sigma(w,\lambda)\phi\right)_{|M_Q(F)} \in \bigoplus_{w \in W(P;Q)} \delta_Q^{\frac{1}{2}} \otimes I_{Q_w}^{M_Q} w \sigma_\lambda.
            \end{equation}
        It factors through $(I_P^{\GL_k} \sigma_\lambda)_{N_Q}$. By \cite[Proposition~VII.3.4]{Renard}, the projection on each components of the RHS is non-zero. Because $\lambda$ is in general position, these components are all irreducible. By computing their Jacquet modules using the geometric lemma \cite[Théorème~VI.5.1]{Renard}, we see that they are also mutually non-isomorphic. But by applying the geometric lemma to $I_P^{\GL_k} \sigma_\lambda$ instead, we know that $(I_P^{\GL_k} \sigma_\lambda)_{N_Q}$ is isomorphic to the RHS of \eqref{eq:explicit_Jacquet}, so that this map must be an isomorphism. Because the representations $\delta_Q^{\frac{1}{2}} \otimes I_{Q_w}^{M_Q} w\sigma_\lambda$ are generic, we conclude by \eqref{eq:asymptotic_whitt} that there exist constants $c_\sigma^Q(w,\lambda)$ such that \eqref{eq:asymptotic_Whittaker} holds. We now have to compute them.

        We begin with $c_\sigma^Q(1,\lambda)$, so that in particular $P \subset Q$. Let $\lambda$ be in general position. Let $\phi \in I_R^{\GL_k}(w_Q \sigma)$ be given by Lemma~\ref{lem:non_zero}. By definition, we have 
        \begin{equation*}
            W_{R,w_Q \sigma}^{\psi}(\phi,w_Q \lambda)=\gamma_{w_Q \sigma}(w_Q^{-1},w_Q \lambda)W_{P,\sigma}^{\psi}(J_{w_Q\sigma}(w_Q^{-1},w_Q\lambda)\phi,\lambda).
        \end{equation*}
        But it is easily checked using \eqref{eq:gamma_factor} that $\gamma_{w_Q \sigma}(w_Q^{-1},w_Q \lambda)=\gamma_{\sigma^\vee}(w_Q,-\lambda)$. It now follows from \eqref{eq:asymptotic_Whittaker}  that the formula for $c_\sigma^Q(1,\lambda)$ holds.

        We now compute $c_\sigma^Q(w,\lambda)$ for $w \in W(P;Q)$. By \eqref{eq:N_n_norm} and \eqref{eq:gamma_factor}, we obtain for any $\phi \in I_P^{\GL_k} \sigma$ and $\lambda$ in general position
        \begin{equation*}
            W_{Q_w,w\sigma}^\psi(N_\sigma(w,\lambda)\phi,w\lambda)=n_{\sigma}(w,\lambda)^{-1}\gamma_{\sigma}(w,\lambda)^{-1} W_{P,\sigma}^\psi(\phi,\lambda).
        \end{equation*}
        By isolating the terms in $w(\lambda \chi_\sigma)$ on each sides we obtain
        \begin{equation*}
            c_{w \sigma}^Q(1,w\lambda)=n_{\sigma}(w,\lambda)^{-1}\gamma_{\sigma}(w,\lambda)^{-1}c_{\sigma}^Q(w,\lambda).
        \end{equation*}
        This concludes as we already know that $c_{w \sigma}^Q(1,w\lambda)=\gamma_{w \sigma^\vee}(w_Q,-w\lambda)^{-1}$ by the previous case.
        
        \end{proof}

    Let $Q$ be a standard parabolic subgroup of $\GL_k$ and let $w \in W(P;Q)$. Using \eqref{eq:L_formula} and \eqref{eq:gamma_formula}, the function $c_\sigma^Q(w,\lambda)$ can be explicitly written in terms of $L$-factors. More precisely, let $w_{Q_w}^Q$ be the longest element in $W(Q_w \cap M_Q)$. Set
    \begin{equation*}
        n_{w \sigma}^{Q}(w_{Q_w}^Q,w\lambda)=\prod_{\alpha \in \Sigma_{Q_w}^Q}n_{w \sigma}(\alpha,\langle w \lambda,\alpha^\vee \rangle).
    \end{equation*}
    Then we have
    \begin{equation}
        \label{eq:explicit_c}
        c_\sigma^Q(w,\lambda)=b(\lambda,\sigma)^{-1}n_{w \sigma}^{Q}(w_{Q_w}^Q,w\lambda)^{-1} \prod_{\alpha \in \Sigma_{Q_w}} \frac{L(\langle w_Q w \lambda, \alpha^\vee \rangle,(w_Qw\sigma)_\alpha)}{\varepsilon(\langle w_Q w \lambda, \alpha^\vee \rangle,(w_Qw\sigma)_\alpha,\psi)}.
    \end{equation}

    \subsubsection{The case of square integrable representations}
    We now take $\delta$ an irreducible square integrable representation of $M_P(F)$. We have $\sigma_\delta$ the supercuspidal representation of $M_{P_\delta}(F)$ given by the Bernstein-Zelevinsky classification recalled in \S\ref{subsubsec:BZ_class}. We will consider $I_P^{\GL_k} \delta$  as a subrepresentation of $I_{P_\delta}^{\GL_k} \sigma_{\delta,-\nu_\delta}$. 

    \begin{lem}
        \label{lem:explicit_geom_lemma}
        Let $Q$ be a standard parabolic subgroup of $\GL_k$ and let $w \in W(P_\delta;Q)$. Then $N_{\sigma_\delta}(w,\lambda-\nu_\delta)$ is well-defined for $\lambda \in \fa_{P,\cc}^*$ in general position. When restricted to $I_P^{\GL_k} \delta$, it is zero unless $w \in {}_Q W_P$.
    \end{lem}

    \begin{proof}
        That $N_{\sigma_\delta}(w,\lambda-\nu_\delta)$ is well-defined follows from Theorem~\ref{thm:N} and the fact that $-\nu_\delta \in \fa_{P_\delta}^{P,*,+}$. We have to study the restriction of the composition $N_{\sigma_\delta}(w,w_\delta^* \mu)J_{\sigma_\delta}(w_\delta^*,\mu)$ for $\mu \in \fa_{P,\cc}^*+\nu_\delta$, both operators being defined there generically. If we assume that $w \notin {}_Q W_P$, there exists a positive root $\alpha \in \Sigma_{P_\delta}^P$ such that $w \alpha<0$, which is equivalent to the existence of $\beta \in \Sigma_{P_\delta}^P$ with $w w_\delta^* \beta >0$. Without loss of generality, we can assume that $\beta \in \Delta_{P_\delta}^P$. By Theorem~\ref{thm:N}, the composition 
        \begin{equation*}
            n_{\sigma_\delta}(w_\delta^*,\mu)^{-1} N_{\sigma_\delta}(w,w_\delta^* \mu)J_{\sigma_\delta}(w_\delta^*,\mu)=N_{\sigma_\delta}(ww_\delta^*,\mu)
        \end{equation*}
        is regular generically along the affine hyperplane $\langle \mu,\beta^\vee \rangle=-1$ which contains $\fa_{P,\cc}^*+\nu_\delta$. But by the formulae for the normalizing factor in \eqref{eq:L_formula} and for Rankin--Selberg $L$-functions of supercuspidal representations in \eqref{eq:cuspi_L}, we see that $n_{\sigma_\delta}(w_\delta^*,\mu)^{-1}$ has a pole along this affine hyperplane. Therefore, $N_{\sigma_\delta}(w,w_\delta^* \mu)J_{\sigma_\delta}(w_\delta^*,\mu)$ must be zero there, which concludes the proof.
    \end{proof}

    \begin{prop}
        \label{prop:Whittaker_explicit_discrete}
            Let $Q$ be any standard parabolic subgroup of $\GL_k$. For every $\phi \in I_P^{\GL_k} \delta$ there exists $\varepsilon >0$ such that for $\lambda \in \fa_{P,\cc}^*$ in general position and $a \in A_Q[\leq \varepsilon]$ we have
            \begin{equation}
            \label{eq:asymptotic_Whittaker_discrete}
                W_{P,\delta}^{\psi}(a,\phi,\lambda)=\sum_{\substack{w \in {}_Q W_P \\ P_\delta \subset P_w}} c_{\sigma_\delta}^Q(w,\lambda-\nu_\delta) w((\lambda-\nu_\delta)\chi_{\sigma_\delta})(a) \delta_{Q}^\frac{1}{2}(a) W_{w.P_\delta,w\sigma_\delta}^{M_Q,\psi} (N_{\sigma_\delta}(w,\lambda-\nu_\delta)\phi,w (\lambda-\nu_\delta)),
                \end{equation}
                where $c_{\sigma_\delta}^Q(w,\lambda-\nu_\delta)$ is defined in \eqref{eq:c_w^Q_formula} and well-defined on $\fa_{P,\cc}^*$. 
    \end{prop}

    \begin{proof}
        By Lemma~\ref{lem:induce_whittaker}, we may take $W_{P_\delta,\sigma_\delta}^{M_P,\psi}$ as Whittaker functional on $I_P^{\GL_k} \delta$. Therefore, Proposition~\ref{prop:Whittaker_explicit_discrete} is a direct consequence of Proposition~\ref{prop:Whittaker_explicit} and Lemma~\ref{lem:explicit_geom_lemma} once we know that for every $w \in W(P_\delta;Q)$ the affine hyperplane $\fa_{P,\cc}^*-\nu_\delta \subset \fa_{P_\delta,\cc}^*$ is not contained in any of the singularities of $c_{\sigma_\delta}^Q(w,\mu)$. But this follows easily from the expression given in \eqref{eq:explicit_c}.
    \end{proof}

    \subsection{A formula for the Zeta function}
    \label{sec:formula_zeta}

    Let $Q_{n+1}$ be a standard parabolic subgroup of $\GL_{n+1}$, with standard Levi factor $M_{Q_{n+1}}=\prod_{i=1}^m \GL_{n_i}$. We denote by $Q$ the Rankin--Selberg parabolic subgroup of $G$ associated to the pair $(Q_{n+1},m)$ by the bijection of Corollary~\ref{cor:param_RS}. This is simply the standard parabolic subgroup $(Q_{n+1}\cap \GL_n) \times Q_{n+1}$ of $G$, and every standard Rankin--Selberg parabolic subgroup of $G$ arises in this way.

We identify the group of cocharacters $X_*(A_Q)$ of $A_{Q}$ with a lattice in $\fa_{Q}$. Set $\Lambda_{Q,H}=X_*(A_{Q}) \cap \fz_Q$, which is a lattice of $\fz_Q$ (see \eqref{eq:zp_descri}). We will denote its group law additively, and its neutral element by zero. Recall that we have defined in \S\ref{subsubsec:RS_coord} a set $\Delta_{Q,H}$. These are the linear forms on $\fz_Q$ which are obtained by restricting those in $\Delta_{Q_{n+1}}$, and this defines a bijection between these two sets. For every $M \in \rr$, set
\begin{equation*}
    \Lambda_{Q,H}[\geq M]=\{ a \in \Lambda_{Q,H} \; | \; \forall \alpha \in \Delta_{Q,H}, \; \langle \alpha,a \rangle \geq M \}.
\end{equation*}
We identify all the groups $\Lambda_{Q,H}$ as subgroups of $\Lambda_{P_0,H}$, where we recall that $P_0$ is the standard Borel of $G$. We now state a combinatorial lemma on partitions of $\Lambda_{P_0,H}$.

\begin{lem}
\label{lem:partitions}
    For every $a \in \Lambda_{P_0,H}$, choose $M_a \in \zz$. Then for every $M \in \zz$ there exists a finite family $(a_{i,Q})$ of elements in $\Lambda_{P_0,H}[\geq M]$, indexed by standard Rankin--Selberg parabolic subgroups $Q$ and $1 \leq i \leq i_Q$ for some integer $i_Q$, such that
    \begin{equation*}
        \Lambda_{P_0,H}[\geq M]=\bigsqcup_Q \bigsqcup_{i=1}^{i_Q} \left(a_{i,Q} + \Lambda_{Q,H}[\geq M_{a_{i,Q}}']\right),
    \end{equation*}
    for some $M_{a_{i,Q}}' \geq M_{a_{i,Q}}$.
\end{lem}

\begin{proof}
    We prove the stronger statement: for every standard Rankin--Selberg parabolic subgroup $R$ of $G$ and every $M$ we have a decomposition of the form
    \begin{equation*}
        \Lambda_{R,H}[\geq M]=\bigsqcup_{Q \supset R} \bigsqcup_{i=1}^{i_Q} \left( a_{i,Q} + \Lambda_{Q,H}[\geq M_{a_{i,Q}}'] \right).
    \end{equation*}
    Let $M_0$ be the integer associated to $0 \in \Lambda_{R,H}$. Assume that $M_0 > M$. For each subset $I \subset \Delta_{R,H}$, set
    \begin{equation*}
        \Lambda_{R,I,H}[M_0,M]=\left\{ a \in \Lambda_{R,H} \; \middle| \begin{array}{l}
           \forall \alpha \in I, \; M_0 > \langle \alpha, a \rangle \geq M, \\
            \forall \alpha \in \Delta_{R,H} \setminus I, \; \langle \alpha,a \rangle \geq M_0. 
        \end{array} \right\}.
    \end{equation*}
    Note that $\Delta_{R,H}$ is a basis of the dual lattice of $\Lambda_{R,H}$. The map
    \begin{equation}
        \label{eq:bij_RS_parab}
        Q \mapsto \Delta_{R,H}^Q := \{ \alpha \in \Delta_{R,H} \; | \; \alpha_{| \fz_Q}=0 \}
    \end{equation}
    induces a bijection between the set of standard Rankin--Selberg parabolic subgroups of $G$ containing $R$, and the subsets of $\Delta_{R,H}$. Moreover, $\Delta_{Q,H}$ is exactly the set of the restrictions of the elements in $\Delta_{R,H} \setminus \Delta_{R,H}^Q$ to $\fz_Q$. We denote by $I \mapsto R_I$ the inverse of \eqref{eq:bij_RS_parab}. Note that $R_{\emptyset}=R$. If $I \neq J$ we have $\Lambda_{R,I,H}[M_0,M] \cap \Lambda_{R,J,H}[M_0,M]=\emptyset$, and moreover for every $I$ there exist $a_{1,I}, \hdots, a_{i_I,I} \in \Lambda_{R,H}[\geq M]$ for some integer $i_I$ such that
    \begin{equation*}
        \Lambda_{R,I,H}[M_0,M]=\bigsqcup_{i=1}^{i_I} \left(a_{i,I} + \Lambda_{R_I,H}[\geq M_0] \right).
    \end{equation*}
    It follows that 
    \begin{equation*}
        \Lambda_{R,H}[\geq M]=\Lambda_{R,H}[\geq M_0] \bigsqcup \bigsqcup_{R_I \supsetneq R} \bigsqcup_{i=1}^{i_I} \left(a_{i,I}+ \Lambda_{R_I,H}[\geq M_0] \right).
    \end{equation*}
    We can now conclude by decreasing induction on the rank of $R$ up to replacing the function $a \mapsto M_a$ by $a \mapsto M_{a_{i,I} + a}$ for each $I$ and $i$.
\end{proof}

Take the elements $a_{i,Q}$ given by Lemma~\ref{lem:partitions}. For every $Q$ and every $1 \leq i \leq i_Q$, there exists $a'_{i,Q} \in \Lambda_{Q,H}$ such that $\Lambda_{Q,H}[\geq M'_{a_{i,Q}}]=a'_{i,Q}+\Lambda_{Q,H}[\geq 0]$. Then we set $b_{i,Q}= a_{i,Q}a'_{i,Q}$. We also identify $X_*(A_Q)$ with a subgroup of $A_Q$ by evaluating at $\varpi$ a uniformizer of $F$. 

Let $P$ be a standard parabolic subgroup of $G$. Let $\delta$ be an irreducible square integrable representation of $M_P(F)$, and let $\sigma_\delta$ and $\nu_\delta$ be respectively the associated supercuspidal representation of $M_{P_\delta}(F)$ and element of $\fa_{P_\delta}^*$ from \S\ref{subsubsec:BZ_class}. To ease notation, we simply write $\sigma$ for $\sigma_\delta$. We denote again by $\chi_{\sigma}$ an element in $\fa_{P_{\delta},\cc}^*$ such that for every $a \in X_*(A_{P_\delta})$ we have $\chi_{\sigma}(a)=q^{-\langle \chi_{\sigma},a \rangle}$. We now write a formula for the Zeta function.

\begin{prop}
\label{prop:Zeta_unfold_local}
    For every $\phi \in I_{P}^G \delta$, there exists a finite family $(b_{i,Q})$ such that for $\lambda \in \fa_{P,\cc}^*$ in general position we have
    \begin{align}
        Z_{\delta}(\phi,\lambda)=&\sum_{Q} \sum_{\substack{w \in {}_Q W_{P} \\ P_\delta \subset P_{w}}} c_{\sigma}^Q(w,\lambda-\nu_\delta) \prod_{\varpi^\vee \in \hat{\Delta}^\vee_{Q,H}} \zeta(\langle w (\lambda-\nu_\delta+\chi_{\sigma}) + \underline{\rho}_{Q},\varpi^\vee \rangle) \nonumber \\
        &\times \sum_{i=1}^{i_Q} \delta_{P_{0,H}}^{-1}(b_{i,Q})W_{w.P_\delta,w\sigma}^{M_Q,\psi_0}\left(b_{i,Q},R(e_{K_H}) N_{\sigma}(w,\lambda-\nu_\delta)\phi,w(\lambda-\nu_\delta) \right). \label{eq:Zeta_unfold_local}
    \end{align}
    Here $Q$ ranges over the standard Rankin--Selberg parabolic subgroups of $G$, $\hat{\Delta}^\vee_{Q,H}$ is the set of coweights defined in \S\ref{subsubsec:RS_coord}, and $e_{K_H}$ is the normalized characteristic function of $K_H$. 
\end{prop}

\begin{proof} 
    By the condition on the support of the Whittaker functions from \cite[Proposition~6.1]{CS}, the definition of $Z_{\delta}$ in \eqref{eq:local_zeta_int} and the Iwasawa decomposition, there exists $M \in \zz$ such that
    \begin{equation}
    \label{eq:Z_first_expansion}
        Z_{\delta}(\phi,\lambda)=\sum_{a \in \Lambda_{P_0,H}[\geq M]} \delta_{P_{0,H}}^{-1}(a)W_{P,\delta_0}^\psi(a,R(e_{K_H}) \phi,\lambda) .
    \end{equation}
    We will show below that this sum is absolutely convergent for $\Re(\lambda)$ in some open subset. For each $a \in \Lambda_{P_0}[\geq M]$, let $\varepsilon_a$ be the element given for $I_{\lambda}(a)(R(e_{K_H})\phi)$ in Proposition~\ref{prop:Whittaker_explicit_discrete}, where $I_\lambda$ is the action of $G(F)$ on $I_{P}^G \delta_{\lambda}$. Then choose $M_a \in \zz$ such that $q^{-M_a} < \varepsilon_a$. By Proposition~\ref{prop:Whittaker_explicit_discrete} and Lemma~\ref{lem:partitions}, we obtain a finite family $(b_{i,Q})$ such that 
    \begin{align}
        Z_{\delta}(\phi,\lambda)=&\sum_{Q} \sum_{\substack{w \in {}_Q W_{P} \\ P_\delta \subset P_{w}}} c_{\sigma}^Q(w,\lambda-\nu_\delta) \sum_{a \in \Lambda_{Q,H}[\geq 0]} (w(\lambda-\nu_\delta + \chi_{\sigma})+ \underline{\rho}_{Q})(a) \nonumber \\
        &\times \sum_{i=1}^{i_Q} \delta_{P_{0,H}}^{-1}(b_{i,Q})W_{w.P_\delta,w\sigma}^{M_Q,\psi_0}\left(b_{i,Q},R(e_{K_H}) N_{\sigma}(w,\lambda-\nu_\delta)\phi,w(\lambda-\nu_\delta) \right). \label{eq:first_zeta_expansion}
    \end{align}
    These manipulations are legitimate if the sums $\sum_{a \in \Lambda_{Q,H}[\geq 0]} (w(\lambda-\nu_\delta + \chi_{\sigma})+\underline{\rho}_Q)(a)$ are absolutely convergent. But they are for any $\lambda \in \fa_{P,\cc}^*$ such that for all standard Rankin--Selberg parabolic subgroups $Q$, all $w \in {}_Q W_{P}$ and all $\varpi^\vee \in \hat{\Delta}^\vee_{Q,H}$ we have $\langle \Re(w (\lambda-\nu_\delta+\chi_{\sigma})) + \underline{\rho}_{Q},\varpi^\vee \rangle > 0$. Using the explicit description of \S\ref{subsubsec:RS_coord}, we see that this defines a non-empty open set of $\fa_{P,\cc}^*$. This can also be seen by translating $\lambda$ by an element $s \in \fa_G^*$ with $s \gg 0$. In any case, we conclude that \eqref{eq:Zeta_unfold_local} holds for $\lambda$ in this subset, and that it holds for $\lambda$ in general position by analytic continuation.
\end{proof}

\subsection{Computation of the residues}

For the rest of this section, we let $\Pi$ be a smooth irreducible representation of $G(F)$ of weak Arthur type with relevant parameter. We define the representations $\pi$ of $M_P(F)$ and $\delta_\pi$ of $M_{P_\pi}(F)$ as in \S\ref{subsubsec:setting_local_GGP}. We also have the associated supercuspidal representation $\sigma_\delta$ of $M_{P_\delta}(F)$ from the Bernstein--Zelevinsky classification of \S\ref{subsubsec:BZ_class}. Therefore, $P_\delta \subset P_\pi \subset P$. Because the representation $\pi$ is fixed, we will simply write $\delta$ for $\delta_\pi$ and $\sigma$ for $\sigma_\delta$. We also have the space $\fa_{\pi,\cc}^*$ as well as $\nu_\Pi$ from \S\ref{subsubsec:setting_local_GGP}. 

Our goal is to compute $RZ^\natural_{\delta}(\cdot,\lambda)$ the restriction of $Z_{\delta}^\natural(\cdot,\lambda)$ to $\fa_{\pi,\cc}^*-\nu_\pi$ for $\lambda$ in general position. The calculation mirrors that of Proposition~\ref{prop:easy_residue}, but the combinatorics is slightly more involved.

\subsubsection{Preliminary notation}
\label{subsubsec:preliminary_notation}

The pair $(P,\pi)$ is of the same combinatorial shape as the global "relevant inducing pairs" $\Pi_H$ from \S\ref{subsubsec:relevant_inducing}. In particular, we can define the two sets $L_+$ and $L_-$ of affine linear forms on $\fa_{P_\pi,\cc}^*$ as in \eqref{eq:Lambda} and \eqref{eq:Lambda'}, and the spaces $\cH_+$ and $\cH_-$. Recall that $\cH_+ \cap \cH_-=\fa_{\pi,\cc}^*-\nu_\pi$. We had described some orders on the elements of $L_+$ in \eqref{eq:first_order} and \eqref{eq:second_order}. We denote the corresponding ordered sets by $L_{n+1}^\uparrow$ and $L_{n}^\uparrow$ respectively. For $m \in \{n,n+1\}$, if $f$ is a meromorphic function on $\fa_{P_{\pi},\cc}^*$ we denote by $\underset{L_m^\uparrow}{\mathrm{Rest}} \; f$ its restriction to $\cH_+$, computed by restricting to the affine hyperplanes in the prescribed order. This object has the same properties as the iterated residue from \S\ref{subsubsec:naive_residue}.

Recall that we have defined in \eqref{eq:Q_cL_defi} a Rankin--Selberg parabolic subgroup $P_{\res}$. With the notation of \eqref{eq:delta_defi}, assume that each $\delta_{1,i}$ (resp. each $\delta_{2,i}$) is a square integrable representation of $\GL_{r(1,i)}$ (resp. $\GL_{r(2,i)}$) and write $\pi_{1,i}=\Speh(\delta_{1,i},d(1,i))$ and $\pi_{2,i}=\Speh(\delta_{2,i},d(2,i))$. Then $P_{\res}$ is standard with standard Levi factor
\begin{equation}
\label{eq:M_Q_L_Levi}
    M_{\res}:=M_{P_{\res}}=\left( \prod_{i=1}^{m_1} \GL_{r(1,i)}^{d(1,i)-1} \prod_{j=1}^{m_2} \GL_{r(2,i)}^{d(2,i)-1} \times \GL_k \right) \times \left( \prod_{i=1}^{m_1} \GL_{r(1,i)}^{d(1,i)-1} \prod_{j=1}^{m_2} \GL_{r(2,i)}^{d(2,i)-1}  \times \GL_{k+1} \right),
\end{equation}
where $k=\sum r(1,i)$. We may further write $M_{\res}=\bfM_{\res}^2 \times \cM_{\res}$, where $\cM_{\res}=\GL_k \times \GL_{k+1}$. Recall that we have defined in \S\ref{subsubsec:residues_cL} two elements $w_1$ and $w_2$ (see also \eqref{eq:w_1_action_1} and \eqref{eq:w_1_action_2} where their action on $M_{P_\pi}$ was explicitly described). Then $w_2 w_{\pi,n}^*$ and $w_1 w_{\pi,n+1}^*$ belong to $W(P_{\pi};P_{\res})$. For simplicity, we set $w^*_n=w_2 w_{\pi,n}^*$ and $w^*_{n+1}=w_1 w_{\pi,n+1}^*$, but we warn the reader that they do not belong to $\GL_n$ or $\GL_{n+1}$. We denote by $Q_{\pi}$ the standard parabolic subgroup $w^*_n. P_\pi=w^*_{n+1}. P_\pi$. We have $Q_\pi \subset P_{\res}$ and we set $\cQ_\pi=Q_\pi \cap \cM_{\res}$. We also write $Q_\delta$ for $w^*_n. P_\delta=w^*_{n+1}. P_\delta$, and we set $\cQ_{\delta}=Q_\delta \cap \cM_{\res}$.  

Let $Q$ be a standard Rankin--Selberg parabolic subgroup $G$ such that $Q \subset P_{\res}$. We set $\bfQ^2=Q \cap \bfM_{\res}^2$ and $\cQ=Q \cap \cM_{\res}$. Note that the latter is a standard Rankin--Selberg parabolic subgroup of $\cM_{\res}$. We consider the two subsets of $W$ 
\begin{equation*}
    W_{\pi,n}(Q)=\{ w' \in {}_{\cQ} W_{\cQ_\pi} \; | \; \cQ_{\delta} \subset \cQ_{\pi,w'} \} w^*_n, \quad
    W_{\pi,n+1}(Q)=\{ w' \in {}_{\cQ} W_{\cQ_\pi} \; | \; \cQ_{\delta} \subset \cQ_{\pi,w'} \} w^*_{n+1}.
\end{equation*}
By Lemma~\ref{lem:W_CH} we see that they are both subsets of $\{ w \in {}_Q W_{P_\pi} \; | \; P_\delta \subset P_{\pi,w} \}$. Finally, if $\lambda \in \fa_{P_\pi,\cc}^*$ or $\fa_{P_\delta,\cc}^*$ we write $(w_m^* \lambda)_{\cQ}$ for its restriction to $\fa_{\cQ_\pi,\cc}^*$ or $\fa_{\cQ_\delta,\cc}^*$. In particular, $(-w_m^* \nu_\delta)_{\cQ} \in \fa_{\cQ_\delta}^{\cQ_{\pi},*,+}$. More precisely, if we decompose $w_m^* \sigma$ as $(w_m^* \sigma)_{\bfQ_\delta} \boxtimes (w_m^* \sigma)_{\cQ_\delta}$, then $(w_m^* \nu_\delta)_{\cQ}=\nu_{(w_m^* \delta)_\cQ}$. We will also write $\lambda=(\lambda_n,\lambda_{n+1})$, and use the notation of \eqref{eq:lambda_coordinates} for the coordinates of $\lambda$.

\subsubsection{Restriction of meromorphic functions}

In this subsection, we compute the restriction of the different terms appearing in $Z_\delta^\natural$. We fix a standard Rankin--Selberg parabolic subgroup $Q$ of $G$ and a $w \in {}_Q W_{P_\pi} $ such that $ P_\delta \subset P_{\pi,w}$. We take $m \in \{n,n+1\}$. We set for $\lambda \in \fa_{P_\pi,\cc}^*$ in general position
\begin{equation}
    \label{eq:F^Q_defi}
    F^Q_{\sigma}(w,\lambda)=\frac{\prod_{\varpi^\vee \in \hat{\Delta}^\vee_{Q,H}} \zeta(\langle w (\lambda-\nu_\delta+\chi_{\sigma}) + \underline{\rho}_{Q},\varpi^\vee \rangle) }{L(\lambda+1/2,\delta_{n} \times \delta_{n+1})}.
\end{equation}

\begin{lem}
    \label{lem:holo_simplification}
    The restriction  $\underset{L_{m}^\uparrow}{\mathrm{Rest}}\;F^Q_{\sigma}(w,\lambda)$ is a well-defined meromorphic function on $\cH_+$. It is identically zero unless $Q \subset P_{\res}$ and $w \in W_{\pi,m}(Q)$. In that case, none of the singularity of $ \underset{L_{m}^\uparrow}{\mathrm{Rest}}\;F^Q_{\sigma}(w,\lambda)$ contain $\cH_+ \cap \cH_-$.
\end{lem}

\begin{proof}
    First, let us remark that the element $w$ belongs to $W(P_{\delta})$. If $\GL_r$ is a factor of $P_\pi$, we say that $w$ preserves the block $\GL_r$ if it preserves the product of factors of $M_{P_\delta}$ contained in $\GL_r$.
    
    We now prove the lemma. We deal with the $m=n$ case, the other one being treated with the same argument. We can assume that $d(1,1)>1$. The first affine hyperplane that we restrict to is $\Lambda_+(1,1,1)^{-1}(\{0\})$. Recall that this affine linear form is $-(\lambda(1,1)_{n,d(1,1)}+\lambda(1,1)_{n+1,1}+1/2)$. By direct computation, we see that it is singular (with multiplicity $1$) for the product $\prod_{\varpi^\vee \in \hat{\Delta}^\vee_{Q,H}} \zeta$ in the numerator of $F^Q_{\sigma}(w,\lambda)$ if and only if the three following conditions are satisfied: 
    \begin{itemize}
        \item $M_Q$ is contained in $(\GL_{r(1,1)} \times \GL_{n-r(1,1)}) \times (\GL_{r(1,1)} \times \GL_{n+1-r(1,1)})$;
        \item $w_n$ preserves the $d(1,1)^{\text{th}}$ $\GL_{r(1,1)}$ block of $M_{P_\pi,n}$, and sends it in first position;
        \item $w_{n+1}$ preserves the first $\GL_{r(1,1)}$ block of $M_{P_\pi,n+1}$, and leaves it in first position.
    \end{itemize}
    Note that in that case the pole comes from only one factor of the product, which corresponds to the first element in $\hat{\Delta}_{P_{\res},H}^\vee$, i.e. $e_1$ with the notation of \eqref{eq:coweight_defi}. On the other hand, using the formula \eqref{eq:cuspi_L} for $L$-functions of supercuspidal representations and \eqref{eq:L2L} for square integrable ones, we see that $L(\lambda+1/2,\delta_{n} \times \delta_{n+1})$ always has a simple zero passing through $\Lambda_+(1,1,1)^{-1}(\{0\})$, which comes from $L(-\Lambda_+(1,1,1)(\lambda),\delta_{1,1}\times \delta_{1,1}^\vee)$. Therefore, we conclude that the restriction of $F^Q_{\sigma}(w,\lambda)$ to $\Lambda_+(1,1,1)^{-1}(\{0\})$ is always well-defined, and is zero unless $Q$ and $w$ meet the three conditions described above. By repeating the procedure for each linear form, we see that $\underset{L_{m}^\uparrow}{\mathrm{Rest}}\;F^Q_{\sigma}(w,\lambda)$ is indeed well-defined and zero unless $Q \subset P_{\res}$ and $w \in W_{\pi,m}(Q)$ as claimed.

    We now fix $Q$ and $w$ and assume that $Q \subset P_{\res}$ and $w \in W_{\pi,m}(Q)$. Denote by $\Delta^*$ the subset of $\varpi^\vee \in \hat{\Delta}_{Q,H}^\vee$ that give rise to poles in the precedent construction. Then we see that for $\lambda \in \cH_+$ in general position 
    \begin{equation}
        \label{eq:explicit_L_factor}
        \prod_{\varpi^\vee \in \hat{\Delta}^\vee_{Q,H}\setminus \Delta^*} \zeta(\langle w (\lambda-\nu_\delta+\chi_{\sigma}) + \underline{\rho}_{Q},\varpi^\vee \rangle)=c \prod_{\varpi^\vee \in \hat{\Delta}^\vee_{\cQ,H}} \zeta\left(\langle w (\lambda-\nu_\delta+\chi_{\sigma}) + \underline{\rho}_{\cQ},\varpi^\vee \rangle\right),
    \end{equation}
    where $c$ is some constant. If we write $w=w' w_m^*$, the RHS only depends on $w' (w_m^* \lambda)_{\cQ}$. But the map $\lambda \in \fa_{\pi,\cc}^*-\nu_\pi \mapsto (w_m^* \lambda)_{\cQ} \in \fa_{\cQ_\pi,\cc}^*$ is surjective and we see by a variation of Lemma~\ref{lem:regular} that for all $\varpi^\vee \in \hat{\Delta}^\vee_{\cQ,H}$ the map $\mu \in \fa_{\cQ_\pi,\cc}^*\mapsto \langle w'\mu,\varpi^\vee \rangle$ is non-zero. Because $\underset{L_{m}^\uparrow}{\mathrm{Rest}}\;F^Q_{\sigma}(w,\lambda)$ is equal to \eqref{eq:explicit_L_factor} divided by some additional $L$ functions, none of its singularities contain $\cH_+ \cap \cH_-$.
\end{proof}

\begin{lem}
    \label{lem:holo_factor}
    Let $\phi \in I_P^G \delta$. Then no singularities of the meromorphic map 
    \begin{equation*}
        c_{\sigma}^Q(w,\lambda-\nu_\delta)N_{\sigma}(w,\lambda-\nu_\delta)\phi
    \end{equation*}
    contain $\cH_+$. Moreover, if $Q \subset P_{\res}$ and $w=w' w_m^* \in W_{\pi,m}(Q)$, we have a meromorphic function $G_\sigma^Q(w,\lambda)$ on $\fa_{P_\pi,\cc}^*$ such that
    \begin{equation*}
        c_{\sigma}^Q(w,\lambda-\nu_\delta)N_{\sigma}(w,\lambda-\nu_\delta)\phi=G_\sigma^Q(w,\lambda) N_{w_m^* \sigma}(w',w^*_m(\lambda-\nu_\delta))N_{\delta}(w^*_m,\lambda)\phi,
    \end{equation*}
    where no singularity of $G_\sigma^Q(w,\lambda)$ nor of the intertwining operators on the RHS contain $\cH_+ \cap \cH_-$.
\end{lem}

\begin{proof}
    Note that the maps $\lambda \in \cH_+ \mapsto \lambda_n \in \fa_{P_{\pi,n},\cc}^*$ and $\lambda \in \cH_+ \mapsto \lambda_{n+1} \in \fa_{P_{\pi,n+1},\cc}^*$ are surjective. Therefore, the first assertion follows from Lemma~\ref{lem:explicit_geom_lemma} and Proposition~\ref{prop:Whittaker_explicit_discrete}.
    
    Let now $Q$ and $w$ be as in the second part. By Theorem~\ref{thm:N}, because $w_m^* \in W(P_\pi)$ and $\phi \in I_{P_\pi}^G \delta$, we can write using \eqref{eq:N_n_norm}
    \begin{equation*}
        N_{\sigma}(w,\lambda-\nu_\delta)\phi=n_{\sigma}(w_m^*,\lambda-\nu_\delta)^{-1}n_{\delta}(w_m^*,\lambda)N_{w_m^* \sigma}(w',w^*_m(\lambda-\nu_\delta))N_{\delta}(w^*_m,\lambda)\phi.
    \end{equation*}
    This holds a least for $\lambda \in \cH_+$ in general position as above. We study each term.
    
    Assume that $m=n+1$. Decompose $N_{\delta}(w^*_m,\lambda)=N_{\delta}(w_1,w_{\pi,n+1}^*\lambda)N_{\delta}(w_{\pi,n+1}^*,\lambda)$. For $\lambda \in \fa_{\pi,\cc}^*-\nu_\pi$, $N_{\delta}(w^*_{\pi,n+1},\lambda)=N_{\delta}(w^*_{\pi,n+1},-\nu_\pi)$ is regular by Theorem~\ref{thm:N}. Moreover, by Lemma~\ref{lem:w_1_action} no singularity of $N_{\delta}(w^*_m,\lambda)$ contain $\cH_+ \cap \cH_-$, and the same applies to $n_{\delta}(w_m^*,\lambda)$. This arguments also works if $m=n$. 
    
    The regularity of $N_{\sigma}(w',w^*_m(\lambda-\nu_\delta))$ only depends on $(w_m^*\lambda)_\cQ$. Because $\cQ_{\delta} \subset \cQ_{\pi,w'}$ and $(-w_m^* \nu_\delta)_{\cQ} \in \fa_{\cQ_\delta}^{\cQ_{\pi},*,+}$, for any $\alpha \in \Sigma_{\cQ_\delta}$ such that $w' \alpha <0$ the map $\mu \in \fa_{\cQ_\pi,\cc}^* \mapsto \langle \mu-w_m^* \nu_\delta,\alpha^\vee \rangle$ is either non-zero, either equal to a positive integer. But $\lambda \in \cH_+ \cap \cH_- \mapsto (w_m^* \lambda)_{\cQ} \in \fa_{\cQ_\pi,\cc}^*$ is surjective so no singularity of the operator contain $\cH_+ \cap \cH_-$ by Theorem~\ref{thm:N}.

    It now remains to deal with $ c_{\sigma}^Q(w,\lambda-\nu_\delta)n_{\sigma}(w_m^*,\lambda-\nu_\delta)^{-1}$. Let $w_Q$ (resp $w_{P,\res}$) be the longest element in $W(Q)$ (resp. $W(P_{\res})$). Because $Q \subset P_{\res}$, we have $w_Q=w_{P,\res}w_Q^{P,\res}$ where $w_Q^{P,\res}$ is the longest element in $W(Q \cap M_{P_{\res}})$. As $w=w'w_m^*$ (resp. $w_Q=w_{P,\res}w_Q^{P,\res}$) is longer than $w_m^*$ (resp. $w_{P,\res}$), we see by going back to the definition of the factors $n(w,\lambda)$ and $\gamma(w,\lambda)$ in \eqref{eq:n_formula} and \eqref{eq:gamma_formula_prod}, and by using the inductive properties of $\gamma$-factors (Lemma~\ref{lem:inductive_gamma}), that for $\lambda \in \cH_+$ in general position we have
    \begin{equation*}
        c_{\sigma}^Q(w,\lambda-\nu_\delta)n_{\sigma}(w_m^*,\lambda-\nu_\delta)^{-1}=\gamma_{\delta}(w_m^*,\lambda) \gamma_{w_m^* \delta^\vee}(w_{P,\res},-w_m^* \lambda)^{-1} c_{w_m^* \sigma}^{Q,M_{\res}}(w',w_m^*(\lambda-\nu_\delta)),
    \end{equation*}
    where the coefficient $c_{w_m^* \sigma}^{Q,M_{\res}}(w',\cdot)$ is defined as in \eqref{eq:c_w^Q_formula} for the partial induced representation $I_{Q_\delta}^{M_{\res}} w_m^* \sigma$ with respect to the standard parabolic subgroup $\bfQ^2 \times \cQ$ of $M_{\res}$. 
    
    We claim that no singularity of $c_{w_m^* \sigma}^{Q,M_{\res}}(w',w_m^*(\lambda-\nu_\delta))$ contain $\cH_+ \cap \cH_-$. Indeed, because $M_{\res}=\bfM_{\res}^2 \times \cM_{\res}$, this term breaks into a product $c^{\bfQ^2} c^{\cQ}$. Using \eqref{eq:explicit_c}, we see that the first is actually constant and well-defined along $\cH_+ \cap \cH_-$. By Proposition~\ref{prop:Whittaker_explicit_discrete} and because $(w_m^* \nu_\delta)_{\cQ}=\nu_{(w_m^* \delta)_\cQ}$, we know that $c^\cQ(\mu-w_m^* \nu_{\delta})$ is well-defined for $\mu \in \fa_{\cQ_\pi,\cc}^*$ in general position. But $\lambda \in \cH_+ \cap \cH_- \mapsto (w_m^* \lambda)_{\cQ} \in \fa_{\cQ_\pi}^*$  is surjective as seen above, which concludes.

    We finally deal with $\gamma_{\delta}(w_m^*,\lambda) \gamma_{w_m^* \delta^\vee}(w_{P,\res},-w_m^* \lambda)^{-1}$. By \eqref{eq:gamma_formula_prod}, we have the explicit description
      \begin{equation*}
        \gamma_{\delta}(w_m^*,\lambda) \gamma_{w_m^* \delta^\vee}(w_{P,\res},-w_m^* \lambda)^{-1}=\prod_{w_m^* \alpha <0} \frac{L(1-\langle \lambda, \alpha^\vee \rangle,\delta^\vee_\alpha)}{L(\langle \lambda,\alpha^\vee\rangle,\delta_\alpha)}\prod_{w_{P,\res} \alpha <0} \frac{L(\langle -w_m^*\lambda,\alpha^\vee\rangle,(w_m^*\delta)_\alpha)}{L(1+\langle w_m^* \lambda, \alpha^\vee \rangle,(w_m^*\delta)^\vee_\alpha)}.
    \end{equation*}
    This can further be decomposed as a product over roots of $\GL_n$ and $\GL_{n+1}$. We deal with the $\GL_n$ case. Because $\lambda \in \cH_+ \cap \cH_- \mapsto \lambda_n \in \fa_{P_n,\cc}^*-\nu_{\pi,n}$ is surjective, it suffices to show that no singularities contain $\fa_{P_n,\cc}^*-\nu_{\pi,n}$. We first assume that $m=n$, so that $w_n^*=w_2 w_{\pi,n}^*$. We claim that the $L(\langle -w_n^*\lambda,\alpha^\vee\rangle,(w_n^*\delta)_\alpha)$ terms cause no issues. Indeed, for $\lambda \in \fa_{P_n,\cc}^*-\nu_{\pi,n}$ we have $-w_n^*\lambda=-w_2 \lambda + w_2 \nu_{\pi,n}$. But if $w_{P,\res}\alpha<0$ with $\lambda \mapsto \langle -w_2 \lambda, \alpha^\vee \rangle$ constant, then we must have $\alpha \in \Sigma_{P_{\res}}^{P_+}$ (see \S\ref{subsubsec:residue_free}), and therefore $\langle w_2 \nu_{\pi,n}, \alpha^\vee \rangle<0$. Moreover, let $\alpha$ such that $w_n^* \alpha<0$ and $\lambda \in \fa_{P_n,\cc}^*-\nu_{\pi,n} \mapsto \langle \lambda,\alpha^\vee \rangle$ is constant. Then $\beta \in \Sigma_{P_\delta}^P$. If we set, $\beta=-w_n^* \alpha$ then one can check that $\beta>0$ and $w_{P,\res}\beta <0$. It follows that this pole is compensated by a factor $L(1+\langle w_n^* \lambda, \alpha^\vee \rangle,(w_n^*\delta)^\vee_\alpha)$. The argument for $m=n+1$ is the same.

    We conclude that $G_\sigma^Q(w,\lambda)= c_{\sigma}^Q(w,\lambda-\nu_\delta)n_{\sigma}(w_m^*,\lambda-\nu_\delta)^{-1}n_{\delta}(w_m^*,\lambda)$ works.
\end{proof}

\subsubsection{The formula for the normalized Zeta integral}

We can now write our formula for the restriction of $Z_\delta^\sharp$ to $\fa_{\pi,\cc}^*-\nu_\pi$.

\begin{prop}
    \label{prop:Z^sharp_unfold_local}
    Let $\phi \in I_{P_\pi}^G \delta$. For $m \in \{n,n+1\}$ and $\lambda \in \fa_{\pi,\cc}^*-\nu_\pi$ in general position we have 
    \begin{align}
        Z_{\delta}^\natural&(\phi,\lambda)=\sum_{Q \subset P_{\res}} \sum_{w \in W_{\pi,m}(Q)}  \underset{L_{m}^\uparrow}{\mathrm{Rest}}\;F^Q_{\sigma}(w,\lambda) \times G_\sigma^Q(w,\lambda) \nonumber \\
        &\times \sum_{i=1}^{i_Q} \delta_{P_{0,H}}^{-1}(b_{i,Q})W_{w.P_\delta,w\sigma}^{M_Q,\psi_0}\left(b_{i,Q},R(e_{K_H}) N_{w_m^* \sigma}(w',w^*_m(\lambda-\nu_\delta))N_{\delta}(w^*_m,\lambda)\phi,w(\lambda-\nu_\delta) \right), \label{eq:Z^sharp_unfold}
    \end{align}
    where we write $w=w' w^*_m$ and $G_\sigma^Q(w,\lambda)$.
\end{prop}

\begin{proof}
    By Lemmas \ref{lem:holo_simplification} and \ref{lem:holo_factor} this holds if we restrict to $\cH_+$, and by Lemma~\ref{lem:holo_factor} again we may further restrict to $\fa_{\pi,\cc}^*-\nu_\pi$.
\end{proof}

\begin{rem}
    The functions $\underset{L_{m}^\uparrow}{\mathrm{Rest}}\;F^Q_{\sigma}(w,\lambda)$ and $G_\sigma^Q(w,\lambda)$ can be explicitly computed. This can be used to show that $RZ_\delta^\sharp(\phi,\lambda)$ is a local factor of the global regularized period $\cP^{P_+}(\varphi,\lambda,w_+)$ from Proposition~\ref{prop:alternative_construction}. For more details, we refer the reader to \cite[Proposition~13.6.0.1]{BoiPhD}.
\end{rem}

\subsection{The split local non-tempered Ichino--Ikeda conjecture - second proof}

The upshot of Proposition~\ref{prop:Z^sharp_unfold_local} is that we obtain the factorization property of $Z^\natural_\delta$.

\begin{cor}
    \label{cor:generic_facto}
    For $\lambda \in \fa_{\pi,\cc}^*-\nu_\pi$ in general position, the linear form $\phi \mapsto Z_{\delta}^\natural(\phi,\lambda)$ factors through the quotient $I_{P_\pi}^{G} \delta_{\lambda} \to I_{P}^G \pi_{\mu}$, where $\mu=\lambda+\nu_\pi$.
\end{cor}

\begin{proof}
    By \cite[Section~I.11]{MW89}, the quotient is realized by $N_{\delta}(w_\pi^*,\lambda)$. The corollary is now a direct consequence of Proposition~\ref{prop:Z^sharp_unfold_local} once we know that $N_{\delta}(w_2,w_{\pi,n}^* \lambda)$ and $N_{\delta}(w_1,w_{\pi,n+1}^*\lambda)$ are regular for $\lambda$ in general position, which follows from the proof of Lemma~\ref{lem:holo_factor}.
\end{proof}

By Proposition~\ref{prop:facto_is_enough}, we know that this property implies the split non-tempered Ichino--Ikeda conjecture from Theorem~\ref{thm:local_GGP_explicit}. This therefore gives a local alternative proof of this result.

\printbibliography

\begin{flushleft}
Paul Boisseau \\
Max Planck Institute for Mathematics, \\
Vivatsgasse 7, \\
53111 Bonn, Germany
\medskip
	
email:\\
boisseau@mpim-bonn.mpg.de \\
\end{flushleft}

\end{document}